\documentclass[9pt,a4paper]{amsart}
\usepackage{verbatim}

\usepackage[toc]{appendix}
\usepackage[T1]{fontenc}
\usepackage{graphicx}
\usepackage{enumerate}
\usepackage{amsmath,amsfonts,amssymb}
\usepackage{color}
\def\loc{\operatorname{loc}}
\usepackage{cite}
\definecolor{citation}{rgb}{0.11,0.67,0.84}
\definecolor{formula}{rgb}{0.1,0.2,0.6}
\definecolor{url}{rgb}{0.11,0.67,0.84}
\usepackage{pgf,tikz}
\usepackage{mathrsfs}
\usepackage{fancyhdr}
\usepackage{dutchcal}

\usepackage{dsfont}

\newcommand{\reqnomode}{\tagsleft@false}

\usepackage{hyperref}

\def\GG{\textnormal{\texttt{gl}}_{\theta, \delta}}
\def\GGp{\textnormal{\texttt{gl}}_{\theta, \delta}^{+}}

\def\dx{\,{\rm d}x}

\def\TT{\mathcal{T}}
\def\dtau{\,{\rm d}\nu}

\def\dy{\,{\rm d}y}

\def\snail{\textnormal{\texttt{snail}}}
\def\rhs{\textnormal{\texttt{rhs}}_{\theta}}
\def\rhsp{\textnormal{\texttt{rhs}}_{\theta}^{+}}

\def \d{\,{\rm d}}
\def \diver{\,{\rm div}}
\def\dist{\,{\rm dist}}

\def \cacc{\textnormal{\texttt{ccp}}_{\delta}}
\def \caccs{\textnormal{\texttt{ccp}}_{\delta,*}}

\def \ecc{\textnormal{\texttt{exs}}_{\delta}}
\def \av{\textnormal{\texttt{av}}}
\newcommand\ccc{\mathfrak{c}}

\allowdisplaybreaks
\makeatletter
\DeclareRobustCommand*{\bfseries}{%
  \not@math@alphabet\bfseries\mathbf
  \fontseries\bfdefault\selectfont
  \boldmath
}

\DeclareMathOperator*{\osc}{osc}

\makeatother

\newlength{\defbaselineskip}
\newcommand{\setlinespacing}[1]
           {\setlength{\baselineskip}{#1 \defbaselineskip}}
\newcommand{\mint}{\mathop{\int\hskip -1,05em -\, \!\!\!}\nolimits}

\newtheorem{theorem}{Theorem}

\newtheorem{remark}{Remark}[section]
\newtheorem{lemma}{Lemma}[section]
\newtheorem{proposition}{Proposition}[section]
\numberwithin{equation}{section}
\allowdisplaybreaks[4]

\newcommand{\kk}{\kappa}

\def\en{\mathbb N}
\def\er{\mathbb R}

\newcommand{\ti}[1]{\tilde{#1}}

\newcommand\eps\varepsilon
\newcommand{\um}{u_{\textnormal{\texttt{m}}}}
\newcommand{\ttdd}{\textnormal{\texttt{d}}}
\newcommand{\umm}{\ti{u}_{\textnormal{\texttt{m}}}}
\newcommand{\gmm}{\ti{g}_{\textnormal{\texttt{m}}}}
\newcommand{\wmm}{\ti{w}_{\textnormal{\texttt{m}}}}
\def\eqn#1$$#2$${\begin{equation}\label#1#2\end{equation}}

\newcommand{\Kst}{\ti{K}_{\texttt{s}}}

\newcommand{\data}{\textnormal{\texttt{data}}}
\newcommand{\datah}{\textnormal{\texttt{data}}_{\textnormal{h}}}
\newcommand{\datar}{\textnormal{\texttt{data}}_{\gamma}}

\newcommand{\datab}{\textnormal{\texttt{data}}_{\textnormal{b}}}
\newcommand{\be}{\begin{equation}}
\newcommand{\ee}{\end{equation}}

\newcommand{\sK}{K_{\texttt{s}}}

\newcommand{\rr}{\varrho}

\newcommand{\snr}[1]{\lvert #1\rvert}

\newcommand{\nr}[1]{\lVert #1 \rVert}

\newcommand{\N}{\mathbb{N}}

\def\name[#1, #2]{#1 #2}

\newcommand{\rif}[1]{(\ref{#1})}
\newcommand{\trif}[1] {\textnormal{\rif{#1}}}

\newcommand{\stackleq}[1]{\stackrel{\eqref{#1}}{ \leq}}

\title[Gradient regularity in mixed local and nonlocal problems]{Gradient regularity in mixed local \\and nonlocal problems}

\author[De Filippis]{Cristiana De Filippis}   \address{Cristiana De Filippis\\Dipartimento SMFI, Universit\`a di Parma, Viale delle Scienze 53/a, Campus, 43124 Parma, Italy} \email{\url{cristiana.defilippis@unipr.it}}
\author[Mingione]{Giuseppe Mingione}  \address{Giuseppe Mingione\\Dipartimento SMFI, Universit\`a di Parma, Viale delle Scienze 53/a, Campus, 43124 Parma, Italy} \email{\url{giuseppe.mingione@unipr.it}}
\begin{document}

\subjclass[2010]{49N60, 35J60\vspace{1mm}} 

\keywords{Regularity, mixed local and nonlocal operators, $p$-Laplacian\vspace{1mm}}

\thanks{{\it Acknowledgements.}\ This work is supported by 
the University of Parma via the project ``Regularity, Nonlinear Potential Theory and related topics".
\vspace{1mm}}

\maketitle

\begin{abstract}
Minimizers of functionals of the type
$$w\mapsto \int_{\Omega}[\snr{Dw}^{p}-fw]\dx+\int_{\mathbb{R}^{n}}\int_{\mathbb{R}^{n}}\frac{\snr{w(x)-w(y)}^{\gamma}}{\snr{x-y}^{n+s\gamma}}\dx\dy$$
with $p, \gamma>1>s >0$ and $p> s\gamma$, are locally $C^{1, \alpha}$-regular in $\Omega$ and globally H\"older continuous. 
\end{abstract}
\vspace{3mm}
\setcounter{tocdepth}{1}
\allowdisplaybreaks
\setlinespacing{1.08}
\section{Introduction}\label{si}
Mixed local and nonlocal problems are a subject of recent, emerging interest and intensive investigation. Essentially, the main object in question is an elliptic operator that combines two different orders of differentiation, the simplest model case being $-\Delta + (-\Delta)^{s}$, for $s \in (0,1)$. Here, the simultaneous presence of a leading local operator, and a lower order fractional one, constitutes the essence of the matter. In this special case, from a variational viewpoint, one is considering energies of the type
$$
w\mapsto \int_{\Omega}\snr{Dw}^{2}\dx+\int_{\mathbb{R}^{n}}\int_{\mathbb{R}^{n}}\frac{\snr{w(x)-w(y)}^{2}}{\snr{x-y}^{n+2s}}\dx\dy\,, \qquad 0 < s < 1\,.
$$
Here, as in all the rest of the paper, $\Omega \subset\er^n$ denotes at a bounded, Lipschitz regular domain and $n\geq 2$. 
First results in this direction have been obtained in \cite{chen, chen1, chen2, foo}, via probabilistic methods. More recently, in a series of interesting papers, Biagi, Dipierro, Valdinoci, and Vecchi \cite{BDVV, BDVV2, BDVV3, BDVV4} have started a systematic investigation of problems involving mixed operators, proving a number of results concerning regularity and qualitative behaviour for solutions, maximum principles, and related variational principles. Up to now, the literature is mainly devoted to the study of linear operators. As for nonlinear cases, for instance those arising from functionals as 
\eqn{model}
$$
w\mapsto \int_{\Omega}[\snr{Dw}^{p}-fw]\dx+\int_{\mathbb{R}^{n}}\int_{\mathbb{R}^{n}}\frac{\snr{w(x)-w(y)}^{\gamma}}{\snr{x-y}^{n+s\gamma}}\dx\dy\,,
$$
the study of regularity of solutions has been confined to $L^\infty_{\loc}(\Omega)$ and $C^{0, \alpha}_{\loc}(\Omega)$ estimates (for small $\alpha$), that is, the classical De Giorgi-Nash-Moser theory. In this paper, our aim is to propose a different approach, aimed at proving maximal regularity of solutions to variational mixed problems in nonlinear, possibly degenerate cases as in \rif{model}. Specifically, we shall prove the local H\"older continuity of the gradient of minimizers. A sample of our results is indeed
\begin{theorem}\label{t1} Let $u \in W^{1,p}_0(\Omega)\cap W^{s,\gamma}(\er^n)$ be a minimizer of \eqref{model}, with $p, \gamma>1>s >0$ and $p> s\gamma$, and such that $u\equiv 0$ on $\er^n\setminus \Omega$. If $\partial \Omega \in C^{1, \alpha_b}$ for some $\alpha_b \in (0,1)$, and $f \in L^d(\Omega)$ for some $d>n$, then $Du$ is locally H\"older continuous in $\Omega$ and $u \in C^{0, \alpha}(\er^n)$ for every $\alpha <1$. 
\end{theorem}
Moreover, considering the more familiar case of the sum of two $p$-Laplaceans, we have  
\begin{theorem}\label{ccc6}
Let $u\in W^{1,p}_{\loc}(\Omega)\cap W^{s, p}(\er^n)$ be a solution to 
$-\Delta_p u+(-\Delta_p )^su =f$ in $\Omega$, with $f \in L^d_{\loc}(\Omega)$ for some $d>n$. Then $Du$ is locally H\"older continuous in $\Omega$. 
\end{theorem} 
Our approach is flexible and allows us to consider general functionals of the type 
\eqn{fun.1}
$$
\mathcal{F}(w):=\int_{\Omega}\left[F(Dw)-fw\right]\dx+\int_{\mathbb{R}^{n}}\int_{\mathbb{R}^{n}}\Phi(w(x)-w(y))K(x,y)\dx\dy
$$
modelled on the one in \rif{model}, i.e. $F(Dw)\approx |Dw|^p$ in the $C^2$-sense, $\Phi(t)\approx t^{\gamma}$ in the $C^1$-sense and $K(x,y) \approx |x-y|^{-n-s\gamma}$. Note that, although we specialize to the variational setting, the regularity estimates we are presenting here actually work for general mixed equations almost verbatim, as our analysis is essentially based on the use of the Euler-Lagrange equation of functionals as in \rif{fun.1}; for this, see Section \ref{estensione}. For the correct notion of minimality, and the related functional setting, as well as for results in full generality, see Section \ref{exunisec}. Theorem \ref{t1} achieves the maximal regularity of minima, namely, the local H\"older continuity of the gradient of minimizers in $\Omega$. This is the best possible result already in the purely local case given by the $p$-Laplacean equation $-\Delta_pu=0$, which is covered by Uraltseva-Uhlenbeck theory and related counterexamples \cite{le1, manth1, manth2, uh, ur}. In addition, the case $p\not= \gamma$ is here considered for the first time, thereby allowing a full mixing between local and nonlocal terms. In this respect, the central assumption is \eqn{centralass}
$$
p>s\gamma\,,
$$
that says, roughly speaking, that the fractional $W^{s, \gamma}$-capacity generated by the nonlocal term in \rif{model} can be controlled by the $W^{1, p}$-capacity (the standard $p$-capacity) generated by $w \mapsto \int  |Dw|^p\dx$. This is exactly the point ensuring that the nonlocal term in \rif{model} has less regularizing effects that the local one, as it happens in the basic case $-\Delta + (-\Delta)^{s}$, when $p=\gamma=2$, and also in the nonlinear models of the type $-\Delta_p+  (-\Delta)^{s}_p$, where the fractional $p$-Laplacean operator appears \cite{dss, fsz, fz, gaki,  kkl, kkp, kms1, kms2}. We also note that, as far as we known, allowing the condition $p\not= \gamma$ is a new, non-trivial feature already when $p=2$ and that even the basic De Giorgi-Nash-Moser theory is not available when $p\not =\gamma$. As a matter of fact, all our estimates simplify in the case $p=\gamma$. 

We have reported Theorem \ref{t1} for the sake of exposition but it is actually a very special case of more general results, i.e., Theorems \ref{t2}-\ref{t4}, whose statements are necessarily more involved due to their greater generality. Before stating the precise assumptions and the results in full generality, we spend a few words about the techniques we are going to use, and on some relevant connections. Up to now, the methods proposed in the literature to deal with mixed operators are, in a sense, direct. More precisely, both the local terms and the nonlocal ones stemming from the equations interact simultaneously via energy methods. These techniques ultimately rely on those used in the nonlocal case \cite{brasco1, brasco2, brasco3, DKP, DKP2, kms1, kms2, kkl, kkp} for purely nonlocal operators. This approach does not allow to prove regularity of solutions beyond that allowed by nonlocal operators techniques, which is not the best one can hope for, as, in mixed operators, the leading regularizing term is the local one. In this paper we reverse the approach, relying more on the methods, and, especially, on the estimates available in  regularity theory of local operators. In a sense, we separate the local and nonlocal part combining energy estimates of Caccioppoli type with a perturbative like approach. The crucial point is to fit the terms stemming from the nonlocal term in the iteration procedures that would naturally come up from considering the local part only. For this we have to consider a complex scheme of quantities, interacting with each other, and controlling simultaneously both the oscillations of the solution on small balls, and those averaging the oscillations over their complement (such quantities are detailed in Section \ref{lalista}). This first leads to H\"older regularity of solutions with every exponent (Theorem \ref{t2}) and then to the same kind of estimates globally (Theorem \ref{t3}); combining these  ingredients with a priori regularity estimates from the classical local theory, leads to Theorem \ref{t4}. We mention that, due to the assumption $p\not=\gamma$, functionals as in \rif{model}-\rif{fun.1}  connect to a large family of problems featuring anisotropic operators and integrands with so-called nonstandard growth conditions \cite{CM, ciccio, dm, dq, ELM, koch, liebe}, and to some other classes of anisotropic nonlocal problems \cite{byun, byun2, chaker0, chaker, depa, Lin16}. We mention that a further connection has been established in \cite{dm3}, where a class of mixed functionals has been used to approximate local functionals with $(p,q)$-growth in order to prove higher integrability of minimizers. Further approximations via mixed operators occur in the interesting paper \cite{sttz}. 
\vspace{-1.3mm}

 \subsection{Assumptions and results}\label{exunisec}
When considering the functional $\mathcal{F}$ in \rif{fun.1}, the integrand $F\colon \mathbb{R}^{n}\to \mathbb{R}$ is assumed to be $C^{2}(\mathbb{R}^{n}\setminus \{0\})\cap C^{1}(\mathbb{R}^{n})$-regular and to satisfy the following standard $p$-growth and coercivity assumptions (see \cite{laur, manth1, manth2})
\begin{eqnarray}\label{assf}
\begin{cases}
\ \Lambda^{-1}(\snr{z}^{2}+\mu^{2})^{p/2}\le F(z)\le \Lambda(\snr{z}^{2}+\mu^{2})^{p/2}\\
\ \snr{\partial_{z} F(z)}+(\snr{z}^{2}+\mu^{2})^{1/2}\snr{\partial_{zz}F(z)}\le \Lambda(\snr{z}^{2}+\mu^{2})^{(p-1)/2}\\
\ \Lambda^{-1}(\snr{z}^{2}+\mu^{2})^{(p-2)/2}\snr{\xi}^{2}\leq \partial_{zz}F(z)\xi\cdot \xi 
\end{cases}
\end{eqnarray}
for all $z\in \mathbb{R}^{n}\setminus \{0\}$, $\xi\in \mathbb{R}^{n}$, where $\mu\in [0,1]$ and $\Lambda\geq 1$ are fixed constants. The function $\Phi\colon \mathbb{R}\to \mathbb{R}$ is assumed to satisfy
\begin{eqnarray}\label{assph}
\begin{cases}
\ \Phi(\cdot)\in C^{1}(\mathbb{R})\,, \quad  t\mapsto \Phi(t) \ \textnormal{is convex}\\
\ \Lambda^{-1}\snr{t}^{\gamma}\le \Phi(t)\le \Lambda\snr{t}^{\gamma}\,, \quad  \Lambda^{-1}\snr{t}^{\gamma}\le \Phi'(t)t\le \Lambda \snr{t}^{\gamma}
\end{cases}
\end{eqnarray}
for all $t\in \mathbb{R}$. The kernel $K\colon \er^n\times \er^n \to \er$ satisfies
\begin{eqnarray}\label{assk}
\frac{1}{\Lambda\snr{x-y}^{n+s\gamma}}\le K(x,y)\le \frac{\Lambda}{\snr{x-y}^{n+s\gamma}}
\end{eqnarray}
for all $x,y\in \er^n, x\not =y$. As already mentioned, unless otherwise stated, $p,s,\gamma$ are such that $p,\gamma >1 >s >0$, with $p> s\gamma$. We shall consider a boundary datum $g\in W^{1,p}(\Omega)\cap W^{s,\gamma}(\mathbb{R}^{n})$. 
In order to get global continuity of minimizers, we consider the following requirements on the boundary $\partial \Omega$:
\eqn{g3}
$$
\begin{cases}
\partial\Omega \in C^{1, \alpha_b},\ \  \alpha_b \in (0,1)\\
g\in W^{1,q}(\Omega)\cap W^{a,\chi}(\mathbb{R}^{n})\\
q>p,\ \ a>s,\ \  \chi>\gamma, \ \  \kappa:=\min \{1-n/q, a-n/\chi\}>0\,.
\end{cases}
$$
In particular, this implies that $q, a\chi>n$. Interior H\"older estimates, both for minima and their gradients, need less, and essentially no boundary assumption; for this, we shall replace \rif{g3} by the weaker
\eqn{g33}
$$
 g\in L^{\infty}(\mathbb{R}^{n})
$$
that in fact will only be needed in when $\gamma >p$. Note that $W^{a,\chi}(\mathbb{R}^{n})\subset L^\infty(\er^n)$ holds provided $a-n/\chi>0$. Conditions \eqref{assf}-\eqref{g33} lead to consider the following natural functional setting:
$$
\begin{cases}
\ \mathbb{X}_{g}(\Omega):=\left\{w\in g + W^{1,p}_0(\Omega)\cap W^{s,\gamma}(\mathbb{R}^{n})\colon w\equiv g \ \mbox{in} \ \ \mathbb{R}^{n}\setminus \Omega\right\}\\
\ \mathbb{X}_{0}(\Omega):=\left\{w\in W^{1,p}_0(\Omega)\cap W^{s,\gamma}(\mathbb{R}^{n})\colon w\equiv 0 \ \mbox{in} \ \mathbb{R}^{n}\setminus \Omega\right\}\,.
\end{cases}
$$
Note that some ambiguity arises in the definition of $\mathbb{X}_{g}$; in fact, this is actually meant as the subspace of functions $w\in W^{s, \gamma}(\er^n)$ whose restriction on $\Omega$ belongs to $g + W^{1,p}_0(\Omega)$. Compare for instance with the discussion made in \cite{BDVV, BDVV4}, where related functional settings are considered.  
Under assumptions \eqref{assf}-\eqref{assk} and \eqref{g33}, there exists a unique solution $u\in \mathbb{X}_{g}(\Omega)$ to 
\eqn{fun}
$$
\mathbb{X}_{g}(\Omega)\ni u\mapsto  \min_{w \in \mathbb{X}_{g}(\Omega)} \mathcal{F}(w)\,.
$$
Moreover
\eqn{el}
$$\int_{\Omega}\left[\partial_{z} F(Du)\cdot D\varphi -f\varphi\right]\dx
+\int_{\mathbb{R}^{n}}\int_{\mathbb{R}^{n}}\Phi'(u(x)-u(y))(\varphi(x)-\varphi(y)) K(x,y)\dx\dy=0 
$$
holds for every $\varphi\in \mathbb{X}_{0}(\Omega)$. The proof of these facts is quite standard, and relies on the application of Direct Methods of the Calculus of Variations. The details can be found for instance in \cite[Sections 3.3-3.5]{dm3}, where actually a more delicate case of mixed operators is considered. As for the derivation of the Euler-Lagrange equation, this is standard once \rif{assf}-\rif{assk} are assumed, and, for the nonlocal part, proceeds as in \cite{DKP, dm3}. Assumptions \rif{assf}-\rif{g33} come along with two different lists of parameters (the data of the problem) that we shall use to simplify the dependence on the various constants. These are
\eqn{idati}
$$
\begin{cases}
\, \datab:=(n,p,s,\gamma,\Lambda,\|f\|_{L^n(\Omega)},\nr{g}_{W^{1,p}(\Omega)},\nr{g}_{W^{s,\gamma}(\er^n)}, \|g\|_{L^{\infty}(\er^n)}, \Omega)\\
\, \datah:=(n,p,s,\gamma,\Lambda,\|f\|_{L^n(\Omega)},\nr{u}_{W^{1,p}(\Omega)},\nr{u}_{W^{s,\gamma}(\er^n)}) \  \ \mbox{if $\gamma \leq p$}\\
\, \datah:=\datab \ \  \mbox{if $\gamma >p$}\\
\, \data:=(n,p,s,\gamma,\Lambda,\|f\|_{L^d(\Omega)},\nr{g}_{W^{1,q}(\Omega)},\nr{g}_{W^{s,\gamma}(\er^n)},\nr{g}_{W^{a,\chi}(\er^n)},\Omega)\,.
\end{cases}
$$
For the sake of brevity we shall sometimes indicate a dependence of a constant $c$ on one of the lists in \rif{idati}, also when it will actually occur on a subset of the parameters involved. For example, a constant $c$ depending only on $n,p,s,\gamma$ might be still indicated as $c \equiv c (\datah)$. \begin{theorem}[Almost local Lipschitz continuity]\label{t2}
Under assumptions \eqref{assf}-\eqref{assk} and \eqref{g33}, with $f\in L^n(\Omega)$, let $u\in \mathbb{X}_{g}(\Omega)$ be as in \eqref{fun}. Then $u\in C^{0,\alpha}_{\loc}(\Omega)$ for every $\alpha\in (0,1)$ and, for every open subset $\Omega_0\Subset \Omega$,   
$
[u]_{0, \alpha;\Omega_0} \leq c 
$
holds with $c\equiv c(\datah,\alpha, \dist(\Omega_0, \partial \Omega))$. Assumption \eqref{g33} can be dropped when $\gamma \leq p$. 
\end{theorem}
\begin{theorem}[Global H\"older continuity]\label{t3}
Under assumptions \eqref{assf}-\eqref{g3} with $f\in L^n(\Omega)$, let $u\in \mathbb{X}_{g}(\Omega)$ be as in \eqref{fun}. Then $u\in C^{0,\alpha}(\er^n)$ for every $\alpha < \kappa$ and $[u]_{0, \alpha;\er^n}\leq c (\data)$. In particular, if in addition $g \in W^{1, \infty} (\er^n)$, then $u\in C^{0,\alpha}(\er^n)$ for every $\alpha <1$. 
\end{theorem}
\begin{theorem}[Gradient local H\"older continuity]\label{t4}
Under assumptions \eqref{assf}-\eqref{assk} and \eqref{g33}, with $f\in L^d(\Omega)$ for some $d>n$, let $u\in \mathbb{X}_{g}(\Omega)$ be as in \eqref{fun}. Then there exists $\alpha \equiv \alpha (n,p,s,\gamma, \Lambda,d) \in (0,1)$, such that $Du\in C^{0,\alpha}_{\loc}(\Omega;\er^n)$ and, for every open subset $\Omega_0\Subset \Omega$,  
$
[Du]_{0, \alpha;\Omega_0} \leq c 
$
holds with $c\equiv c(\datah,\dist(\Omega_0, \partial \Omega))$. Assumption \eqref{g33} can be dropped when $\gamma \leq p$. 
\end{theorem}
Let us briefly comment on the assumptions considered in Theorems \ref{t2}-\ref{t4}. These are essentially sharp. For instance, the statement of Theorem \ref{t2} does not hold when only assuming that $f \in L^t$, for any $t<n$. As for Theorem \ref{t4}, one cannot obtain in general the gradient H\"older continuity only assuming that $f\in L^n$; counterexamples arise already in the purely local (and linear) case $-\Delta u=f$ \cite{cianchi}. In Theorem \ref{t3} the assumptions on $g$ guarantee that $g\in C^{0, \kappa}(\er^n)$ with $\kappa = \min\{1-n/q, a-n/\chi\}$; indeed, note that $W^{a, \chi}\subset C^{0,\kappa}$ (see \cite[Theorem 8.2]{guide}). This is the natural assumption in this setting in order to guarantee that the boundary regularity of solutions obtained matches with the one of the boundary data. Note that, accordingly, Morrey-Sobolev embedding gives $W^{1,q}  \subset C^{0,\kappa}$. In other words, assumption \rif{g3}$_3$ encodes the necessary H\"older continuity of the boundary data $g$ both with respect to the Sobolev space related to the local part of the functional in \rif{fun.1}, and with respect to the nonlocal one. In the following, letting $f\equiv 0$ in $\er^n\setminus \Omega$, we can always take $f \in L^d(\er^n)$ and $f \in L^n(\er^n)$ in Theorems \ref{t2}-\ref{t3} and \ref{t4}, respectively.

\begin{remark}\label{localerem}\emph{
When considering the case $\gamma \leq p$, in Theorems \ref{t2} and \ref{t4} no assumption is put on boundary datum $g$, and, in fact, our results are purely local. See Remark \ref{dipendenza} and Theorem \ref{t6}. Note that, in the case $\gamma \leq p$, on the contrary of other papers devoted to the subject, we dot not need to prove that $u$ is bounded to get its H\"older continuity.}
\end{remark}

Theorems \ref{t2}-\ref{t4} come along with explicit a priori estimates. These can be directly inferred from the proofs and whose shape reflects the optimal approach used here. For brevity we confine ourselves to report the a priori estimate related to Theorem \ref{t2}. This is in the next
\begin{theorem}[Campanato type estimate for Theorem \ref{t2}]\label{t5}
Under assumptions \eqref{assf}-\eqref{assk} and \eqref{g33}, let $u\in \mathbb{X}_{g}(\Omega)$ be as in \eqref{fun}. For every $\alpha <1$, there exist $r_*>0$ and $c\geq 1$ such that
\begin{flalign}
\notag  \mint_{B_{\rr}}|u-(u)_{B_{\rr}}|^p\, dx &
\leq c\left(\frac{\rr}{r}\right)^{\alpha p} 
\left[ \mint_{B_r}|u-(u)_{B_{r}}|^p\, dx+  r^{\alpha p}\nr{f}_{L^{n}(B_{r})}^{p/(p-1)}+r^{\alpha p}\right] \\
&
\quad +c\left(\frac{\rr}{r}\right)^{\alpha p} 
\mint_{\mathbb{R}^{n}\setminus B_{r}} \snr{u-(u)_{B_{r}}}^{\gamma}\, \d\lambda_{x_0}\,, \qquad \d\lambda_{x_0}(x):= \frac{\dx}{\snr{x-x_{0}}^{n+s\gamma}}   \label{campanato}
\end{flalign}
holds whenever $B_{\rr}\equiv B_{\rr}(x_0)\subset B_{r}(x_0)\equiv B_{r}  \subset \Omega$ are concentric balls with $r\leq r_*$. Both $r_*$ and $c$ depend on $n,p,s,\gamma,\Lambda, \alpha$ if $\gamma\leq p$, and on $\datah, \alpha$ when $\gamma >p$. Assumption \eqref{g33} can be dropped when $\gamma \leq p$.  
\end{theorem}
When $\gamma \leq p$ the constant $c$ in \rif{campanato} only depends on $n,p,s,\gamma,\Lambda$. Therefore in this case \rif{campanato}, if reduced to the content of the first line, gives back the classical Campanato type decay estimate for solutions to local non-homogeneous equations (see for instance \cite[Theorem 7.7]{giu}). As it is well-known, such decay estimates on the integral average of $u-(u)_{B_r}$ imply the local $C^{0, \alpha}$-regularity of solutions. Instead, the second line of \rif{campanato} encodes the long-range interactions due to the presence of the nonlocal term in the functional. In this respect, the average $u-(u)_{B_r}$ is performed with respect to a suitable measure, on the complement of $B_r$; the resulting term is often  called $\snail$, it is essentially the nonlocal counterpart of the integrals appearing in the first line and some variations of it are of common use in nonlocal problems (see Section \ref{lalista} for more). In the range $\gamma >p$, the nonlocal term exhibits a growth larger than the local one, and a careful analysis of the proofs, actually reveals that the constant $c$ appearing in \rif{campanato}, depends on $n,p,s, \gamma, \Lambda$ and $\|u\|_{L^\infty}$ (see Remark \ref{dipendenza} for details). This typically happens in all those situations when anisotropic operators are considered, especially in the setting of nonuniformly elliptic problems (see for instance the a priori estimates in \cite{CM, dm3, ciccio, dm}). Apart from this unavoidable detail, the shape of \rif{campanato} still neatly reproduces the one known for the classical local case. We note that estimate \rif{campanato} can be further improved including the decay rate of the last term appearing in \rif{campanato}; this follows from the estimate on certain (fractional) sharp maximal operators considered in Section \ref{intit}, estimate \rif{maximalestimate}, eventually implying \rif{campanato}. 

\subsection{Possible extensions, local solutions}\label{estensione}
The results in this paper can be extended in several directions. For instance, one can consider more general functionals of the type
$$
w \mapsto \int_{\Omega}\left[F(x,Dw)-fw\right]\dx+\int_{\mathbb{R}^{n}}\int_{\mathbb{R}^{n}}\Phi(w(x)-w(y))K(x,y)\dx\dy\,,
$$
where this time we assume that $z \mapsto F(x, z)$ satisfies \rif{assf} uniformly with respect to $x\in \Omega$. The assumption regulating coefficients is 
\eqn{condizione}
$$
\snr{\partial_z F(x,z)-\partial_z F(y, z)} \leq \Lambda \omega(|x-y|)(\snr{z}^{2}+\mu^{2})^{(p-1)/2}\,,
$$
to hold for every choice $x, y \in \Omega$ and $z \in \er^n$. Here $\omega\colon [0, \infty) \to [0, 1)$ is a modulus of continuity, that is, a continuous and non-decreasing function, such that $\omega(0)=0$. Under assumption \rif{condizione}, it is then easy to see that Theorems \ref{t2}-\ref{t3} continue to hold. In order to get an analog of Theorem \ref{t4} we assume in addition that $\omega(t) \leq t^{\sigma}$ holds for some $\sigma \in (0,1)$, this condition being necessary; then the H\"older exponent of $Du$ does not exceed $\sigma$. We note the proof of these assertions is in fact implicit in the proof of boundary regularity provided in Proposition \ref{campi} below. 

Another extension, already mentioned above, is about general solutions to nonlinear integroredifferential operators, not necessarily coming from integral functionals. Moreover, a purely  local regularity approach can be considered. For this, we consider a general vector field $A \colon \er^n\to \er^n$ such that $A \in C^{0}(\er^n)\cap C^1(\er^n\setminus \{0\})$, and a function $\Psi \in C^{0}(\er)$ such that 
\eqn{assA}
$$
\begin{cases}
\ \snr{ A(z)}+(\snr{z}^{2}+\mu^{2})^{1/2}\snr{\partial_{z}A(z)}\le \Lambda(\snr{z}^{2}+\mu^{2})^{(p-1)/2}\\
\ \Lambda^{-1}(\snr{z}^{2}+\mu^{2})^{(p-2)/2}\snr{\xi}^{2}\leq \partial_{z}A(z)\xi\cdot \xi \\
\ \Lambda^{-1}\snr{t}^{\gamma}\le \Psi(t)t\le \Lambda \snr{t}^{\gamma}
\end{cases}
$$
with the same meaning of \rif{assf}-\rif{assph}. Note that the classical $p$-Laplacean operator given by $A(z)\equiv |z|^{p-2}$z is covered by \rif{assA}. We consider functions $u\in W^{1,p}(\Omega)\cap W^{s, \gamma}(\er^n)$, where $\Omega\subset \er^n$ is as usual a bounded and Lipschitz-regular domain, such that 
\eqn{elA}
$$
 \int_{\Omega}\left[A(Du)\cdot D\varphi -f\varphi\right]\dx
+\int_{\mathbb{R}^{n}}\int_{\mathbb{R}^{n}}\Psi(u(x)-u(y))(\varphi(x)-\varphi(y)) K(x,y)\dx\dy=0 
$$
holds for every $\varphi\in \mathbb{X}_{0}(\Omega)$. Note that here no boundary datum $g$ appears.  The definition of solution is instead purely local. In this case we have 
\begin{theorem}\label{t6}
Under assumptions \eqref{assk} and \eqref{assA}, let $u\in W^{1,p}(\Omega)\cap W^{s, \gamma}(\er^n)$ be a solution to \eqref{elA}.
\begin{itemize}
\item 
If $u \in L^{\infty}_{\loc}(\Omega)$  when $\gamma >p$, and $f \in L^n_{\loc}(\Omega)$, then $u\in C^{0,\alpha}_{\loc}(\Omega)$ for every $\alpha\in (0,1)$. 
\item If $u \in \mathbb{X}_{g}(\Omega)$ and conditions \eqref{g3} hold, then $u \in C^{0,\alpha}(\er^n)$ for every $\alpha < \kappa$. 
\item If $u \in L^{\infty}_{\loc}(\Omega)$ when $\gamma >p$, and $f \in L^d_{\loc}(\Omega)$ for some $d>n$, then $u\in C^{1,\alpha}_{\loc}(\Omega)$ for some $\alpha\in (0,1)$. 
\end{itemize}
\end{theorem}
The proof of Theorem \ref{t6} follows verbatim the ones for Theorems \ref{t2}-\ref{t4}. Again, the assumption $u \in L^\infty_{\loc}$ is only needed when $\gamma> p$. Note that Theorem \ref{ccc6} follows as a corollary of Theorem \ref{t6}. 
 \section{Preliminaries}
\subsection{Notation}
Unless otherwise specified, we denote by $c$ a general constant larger or equal than~$1$. Different occurrences from line to line will be still denoted by $c$. Special occurrences will be denoted by $c_*, c_1$ or likewise. Relevant dependencies on parameters will be as usual emphasized by putting them in parentheses. In the following, given $a\in \er$, we denote $a_+:=\max\{a,0\}$. We denote by $ B_r(x_0):= \{x \in \er^n  :   |x-x_0|< r\}$ the open ball with center $x_0$ and radius $r>0$; we omit denoting the center when it is not necessary, i.e., $B \equiv B_r \equiv B_r(x_0)$; this especially happens when various balls in the same context will share the same center. With $B_{r}^{+}(x_{0})$ we mean the upper half ball $B_{r}(x_{0})\cap \left\{x\in \mathbb{R}^{n}\colon x_{n}>0 \right\}$; in connection, we denote $\Gamma_{r}(x_{0}):=B_{r}(x_0)\cap \{x_n=0\}$, whenever $x_0 \in \{x_n=0\}$. Moreover ,given a domain $\Omega\subset \er^n$. 
With $\mathcal B \subset \er^{n}$ being a measurable subset with respect to a Borel (non-negative) measure $\lambda_0$ in $\er^n$, with bounded positive measure $0<\lambda_0(\mathcal B)<\infty$, and with $b \colon \mathcal B \to \er^{k}$, $k\geq 1$, being a measurable map, we denote  
$$
   (b)_{\mathcal B} \equiv \mint_{\mathcal B}  b(x)\, \d\lambda_0(x)  :=  \frac{1}{\lambda_0(B)}\int_{\mathcal B}  b(x)\, \d\lambda_0(x)\;.
$$
According to the standard notation, given $b \colon \mathcal B \to \er^k$, we denote
$$
[b]_{0,\alpha; \mathcal B}:= \sup_{x,y\in \mathcal B; x\not= y}\, \frac{|b(x)-b(y)|}{|x-y|^\alpha}\,, \qquad 
\osc_{\mathcal B} b := \sup_{x, y \in \mathcal B} \, |b(x)-b(y)|
$$
for $0< \alpha \leq 1$ and $\mathcal B \subset \er^{n}$ being a set. 
\subsection{Fractional spaces}\label{fss}
For $\gamma \geq 1$ and $s\in (0,1)$, the space $W^{s,\gamma}(\mathbb{R}^{n})$ is defined via 
$$
W^{s,\gamma}(\mathbb{R}^{n}):=\left\{w\in L^{\gamma}(\mathbb{R}^{n})\colon \int_{\mathbb{R}^{n}}\int_{\mathbb{R}^{n}}\frac{\snr{w(x)-w(y)}^{\gamma}}{\snr{x-y}^{n+s\gamma}}\dx\dy<\infty\right\},
$$
and it is endowed with the norm
$$
\nr{w}_{W^{s,\gamma}(\mathbb{R}^{n})}:=\left(\int_{\mathbb{R}^{n}}\snr{w}^{\gamma}\dx\right)^{1/\gamma}+\left(\int_{\mathbb{R}^{n}}\int_{\mathbb{R}^{n}}\frac{\snr{w(x)-w(y)}^{\gamma}}{\snr{x-y}^{n+s\gamma}}\dx\dy\right)^{1/\gamma}\,.
$$
With $w\in W^{s,\gamma}(\mathbb{R}^{n})$, we also denote 
$$
[w]_{s,\gamma;A}:= \left(\int_{A}\int_{A}\frac{\snr{w(x)-w(y)}^{\gamma}}{\snr{x-y}^{n+s\gamma}}\dx\dy\right)^{1/\gamma}
$$
whenever $A \subset \er^n$ is measurable. 
In a similar way, by replacing $\er^n$ by $\Omega$ in the domain of integration, it is possible to define the fractional Sobolev space $W^{s,\gamma}(\Omega)$ in an open domain $\Omega\subset \mathbb{R}^{n}$. Good general references for fractional Sobolev spaces are \cite{Adams, guide}. For the next result, see also \cite{AKM} and related references. 
\begin{lemma}[Fractional Poincar\'e]
Let $\gamma \in [1, \infty)$, $s\in (0,1)$, $B_{\rr}\subset \mathbb{R}^{n}$ be a ball. If $w\in W^{s,\gamma}(B_{\rr})$, then
\eqn{fraso}
$$
\left(\mint_{B_{\rr}}\snr{w-(w)_{B_{\rr}}}^{\gamma}\dx\right)^{1/\gamma} \le c\rr^s\left(\int_{B_{\rr}}\mint_{B_{\rr}}\frac{\snr{w(x)-w(y)}^{\gamma}}{\snr{x-y}^{n+s\gamma}}\dx\dy\right)^{1/\gamma}
$$
holds with $c\equiv c(n,s,\gamma)$. 
 \end{lemma}
\begin{lemma}[Embedding]\label{em}
Let $1\leq \gamma\le p<\infty$, $s\in (0,1)$ and $B_{\rr} \subset \mathbb{R}^{n}$ be a ball. If $w\in W^{1,p}_{0}(B_{\rr})$, then $w\in W^{s,\gamma}(B_{\rr})$ and 
$$
\left(\int_{B_{\rr} }\mint_{B_{\rr} }\frac{\snr{w(x)-w(y)}^{\gamma}}{\snr{x-y}^{n+s\gamma}}\dx\dy\right)^{1/\gamma}\le c\rr^{1-s}\left(\mint_{B_{\rr}}\snr{Dw}^{p}\dx\right)^{1/p}
$$
holds with $c\equiv c(n,p,s,\gamma)$.
\end{lemma}
\begin{proof}
By standard rescaling - i.e., passing to $B_1\ni x\mapsto w(x_{0}+\rr x)$, with $x_0$ being the center of $B_{\rr}$ - we can reduce to the case $B_{\rr}\equiv B_1(0)$. The assertion then follows by \cite[Proposition 2.2]{guide} and standard Poincar\'e's inequality, as $w\in W^{1,p}_{0}(B_{1})$.
\end{proof}
Using interpolation from \cite{brmi} (see also \cite{brasco4}), we can also prove the following improved imbedding: 
\begin{lemma}[Localized interpolation]\label{bm}
Let $1<p <  \gamma \leq  p/s$ and $s\in (0,1)$, 
$B_{\rr}\subset \mathbb{R}^{n}$. If $w\in W^{1,p}_{0}(B_{\rr})\cap L^{\infty}(B_{\rr})$, then $w\in W^{s,\gamma}(B_{\rr})$ and 
\eqn{bm1}
$$
\left(\int_{B_{\rr}}\mint_{B_{\rr} }\frac{\snr{w(x)-w(y)}^{\gamma}}{\snr{x-y}^{n+s\gamma}}\dx\dy\right)^{1/\gamma}\le c\nr{w}_{L^{\infty}(B_{\rr})}^{1-s}\left(\mint_{B_{\rr}}\snr{Dw}^{p}\dx\right)^{s/p}
$$
holds with $c\equiv c(n,p,s,\gamma)$.
\end{lemma}
\begin{proof}
Note that, on the contrary to the rest of the paper, here we are allowing $p=s\gamma$; this is not really needed in what follows, but we include this case for completeness. Again we can assume that $B_{\rr}(x_{0})\equiv B_1(0)$, and, letting $w\equiv 0$ outside $B_1(0)$, we can assume $w\in W^{1,p}_{0}(\er^n)\cap L^{\infty}(\er^n)$. We first consider the case $p>s\gamma$. We shall use the off-diagonal interpolation results from \cite{brmi} in Triebel-Lizorkin spaces $\ti{F}^{\sigma}_{\lambda, t}$ \cite[2.3.1]{triebel}. Specifically, we use the following interpolation inequality, that holds whenever $0\le \sigma_{1}<\sigma_{2}<\infty$ and $\lambda_{1},\lambda_{2}\in (1,\infty)$ and $t, t_1, t_2>0$
\eqn{bm.0}
$$
\nr{w}_{\ti{F}^{\sigma}_{\lambda, t}(\mathbb{R}^{n})}\le c\nr{w}_{\ti{F}^{\sigma_1}_{\lambda_1, t_1}(\mathbb{R}^{n})}^{\theta}\nr{w}_{\ti{F}^{\sigma_2}_{\lambda_2, t_2}(\mathbb{R}^{n})}^{1-\theta}\,,
$$
provided $
\theta\in (0,1)$ is such that $\sigma=\theta\sigma_{1}+(1-\theta)\sigma_{2}$ and $1/\lambda=\theta/\lambda_{1}+(1-\theta)/\lambda_{2}$,  
where $c\equiv c(n,\sigma_{i},\lambda_{i},t_{i},\theta)$; see \cite[Lemma 3.1]{brmi}. Note the off-diagonal character of \rif{bm.0}, that lies in the fact that $t,t_1,t_2$ can be chosen arbitrarily. From \cite[Pag. 390]{brmi} we recall the identities $\ti{F}^{0}_{\lambda_1,2}\equiv L^{\lambda_1} $, $\ti{F}^{1}_{\lambda_2,2}\equiv W^{1,\lambda_2}$ and $\ti{F}^{\sigma}_{\lambda,\lambda}\equiv W^{\sigma,\lambda}$ when $\sigma \in (0,1)$. This means that \rif{bm.0} turns into
\eqn{bm.00}
$$
\nr{w}_{W^{\sigma,\lambda}(\mathbb{R}^{n})}\le c\nr{w}_{L^{\lambda_1}(\mathbb{R}^{n})}^{1-\sigma}\nr{w}_{W^{1,\lambda_2}(\mathbb{R}^{n})}^{\sigma}\,,
$$
with $\sigma\in (0,1)$ and $1/\lambda=(1-\sigma)/\lambda_{1}+\sigma/\lambda_{2}$. Now, observe that $1<p< \gamma$ and $p>s\gamma$ imply   
$
(1-s)\gamma p/(p-s\gamma)>1,
$
therefore in \rif{bm.00}, we can take $\lambda_{1}=(1-s)\gamma p/(p-s\gamma)$, $\sigma=s$ and $\lambda_{2}=p$; via Poincar\'e's inequality this yields
$$
[w]_{s,\gamma;B_1} 
\leq c\nr{w}_{L^{\lambda_{1}}(B_1)}^{1-s}\nr{w}^{s}_{W^{1,p}(B_1)}
\le c\nr{w}_{L^{\infty}(B_1)}^{1-s}\nr{Dw}^{s}_{L^{p}(B_1)},
$$
with $c\equiv c(n,p,s,\gamma)$, that is \rif{bm1} when $s\gamma <p$. On the other hand, if $s\gamma=p$, we use \cite[Corollary 3.2, (c)]{brmi}, that is
$[w]_{\theta\sigma,\lambda/\theta;\mathbb{R}^{n}}\le c\nr{w}_{L^{\infty}(\mathbb{R}^{n})}^{1-\theta}\nr{w}_{W^{\sigma,\lambda}(\mathbb{R}^{n})}^{\theta}$, that holds whenever $ \theta\in (0,1)$, where 
$c\equiv c(n,\sigma,\lambda,\theta)$. We use this with $\sigma=1$, $\lambda=p$, $\theta=s$ and get
$$
[w]_{s,p/s;B_1}\leq c\nr{w}_{L^{\infty}(B_1)}^{1-s}\nr{w}_{W^{1,p}(B_1)}^{s} \leq c\nr{w}_{L^{\infty}(B_1)}^{1-s}\nr{Dw}_{L^{p}(B_1)}^{s},
$$
with $c\equiv c(n,p,s)$, and the proof is complete.
\end{proof}
We find it useful to have a unified reformulation of Lemmas \ref{em}-\ref{bm}. For this, we introduce, with reference to the exponents $p, s, \gamma$ considered in Theorems \ref{t1}-\ref{t4}, the following quantities: 
\begin{eqnarray}\label{t1t2}
 \vartheta:=\begin{cases}
\ s\quad &\mbox{if} \ \ \gamma>p\\
\ 1\quad &\mbox{if} \ \ \gamma\le p\,,
\end{cases}
\quad \mbox{and we set} \quad 
\begin{cases}
 \mathds{A}_{\gamma} := 1\ &\mbox{if $\gamma >p$ and $0$ otherwise}\\
 \mathds{B}_{\gamma} := 1\ &\mbox{if $\gamma <p$ and $0$ otherwise}\\
  \mathds{C}_{\gamma} := 1\ &\mbox{if $\gamma =p$ and $0$ otherwise}.
\end{cases}
\end{eqnarray}
Note that 
$ \mathds{A}_{\gamma}+ \mathds{B}_{\gamma}+ \mathds{C}_{\gamma}=1$. With this definition we note that \rif{centralass} translates into
\eqn{centralass2}
$$
p\not= \gamma \Longrightarrow p > \vartheta \gamma \quad \mbox{and} \quad p \geq \vartheta\gamma \,.
$$ 
We can now summarize the parts we need of Lemmas \ref{em} and \ref{bm} in the following:
\begin{lemma}\label{bmfinal}
Let $w\in W^{1,p}_{0}(B_{\rr})$, with $p,\gamma>1$, $s\in (0,1)$ be such that $s\gamma \leq p$; assume also that $w\in L^{\infty}(B_{\rr})$, when $\gamma >p$. Then $w\in W^{s,\gamma}(B_{\rr})$ and 
\eqn{bm111}
$$
\left(\int_{B_{\rr}}\mint_{B_{\rr} }\frac{\snr{w(x)-w(y)}^{\gamma}}{\snr{x-y}^{n+s\gamma}}\dx\dy\right)^{1/\gamma}\le c\nr{w}_{L^{\infty}(B_{\rr})}^{1-\vartheta}\rr^{\vartheta-s}\left(\mint_{B_{\rr}}\snr{Dw}^{p}\dx\right)^{\vartheta/p}
$$
holds with $c\equiv c(n,p,s,\gamma)$. In \trif{bm111} we interpret $\nr{w}_{L^{\infty}(B_{\rr})}^{1-\vartheta}=1$ when $\gamma\leq p$ and therefore $\vartheta=1$.
\end{lemma}

\subsection{Miscellanea}
We shall often use the auxiliary vector field $V_{\mu}\colon \er^{n} \to  \er^{n}$, defined by
\eqn{vpvqm}
$$
V_{\mu}(z):= (|z|^{2}+\mu^{2})^{(p-2)/4}z
$$
whenever $z \in \er^{n}$, where $p\in (1,\infty)$ and $\mu\in [0,1]$ are as in \rif{assf}. It follows that
\eqn{Vm}
$$
\snr{V_{\mu}(z_{1})-V_{\mu}(z_{2})}\approx (\snr{z_{1}}^{2}+\snr{z_{2}}^{2}+\mu^{2})^{(p-2)/4}\snr{z_{1}-z_{2}}, 
$$
where the equivalence holds up to constants depending only on $n,p$.  A standard consequence of \rif{assf}$_3$ is the following strict monotonicity inequality:
\eqn{monoin}
$$
\snr{V_{\mu}(z_1)-V_{\mu}(z_2)}^{2} \leq c(\partial_z F(z_2)-\partial_z F(z_1))\cdot(z_2-z_1)
$$
holds whenever $z_1, z_2 \in \er^n$, where $c \equiv c (n,p,\Lambda)$. The two inequalities in the last two displays are in turn based in on the following one 
\eqn{algebra}
$$
\int_{0}^{1}(\snr{z_{1}+\lambda(z_{2}-z_{1})}^{2}+\mu^2)^{\texttt{t}/2} \ \d\lambda\approx_{n,\texttt{t}} (\snr{z_{1}}^{2}+\snr{z_{2}}^{2}+\mu^2)^{\texttt{t}/2}
$$
that holds whenever $\texttt{t}>-1$ and $z_{1},z_{2}\in \mathbb{R}^{n}$ are such that $\snr{z_{1}}+\snr{z_{2}}+\mu>0$. 
As a consequence of \rif{Vm} and \rif{monoin}, it also follows that 
\eqn{usa}
$$
|z|^{p} \leq c\, \partial_{z} F(z) \cdot z  +  c\mu^{p}
$$
holds for every $z \in \er^n$, where, again, it is $c \equiv c (n,p,\Lambda)$; for the facts in the last four displays see for instance \cite{AKM, dm3, ha} and related references. Finally, three classical iteration lemmas. The first one can be obtained by  \cite[Lemma 6.1]{giu} after a straightforward adapatation. Lemma \ref{l5bis} comes via a reading of the proof of (the very similar) \cite[Lemma 2.2]{gg}. Finally, Lemma \ref{l5.1} is nothing but De Giorgi's geometric convergence lemma \cite[Lemma 7.1]{giu}. 
\begin{lemma}\label{l5}
Let $h\colon [\rr_{0},\rr_{1}]\to \mathbb{R}$ be a non-negative and bounded function, and let $\theta \in (0,1)$, $a_i,\gamma_{i}, b\ge 0$ be numbers, $i\leq k\in \en$.  Assume that
$$
h(t)\le \theta h(s)+\sum_{i=1}^k \frac{a_i}{(s-t)^{\gamma_{i}}}+b
$$
holds whenever $\rr_{0}\le t<s\le \rr_{1}$. Then 
$$ 
h(\rr_{0})\le c\sum_{i=1}^k \frac{a_i}{(\rr_{1}-\rr_{0})^{\gamma_{i}}}+c\,b$$ holds too, where $c\equiv c (\theta, \gamma_i)$. 
\end{lemma}
\begin{lemma}\label{l5bis}
\, Let $h\colon [0,r_0]\to \mathbb{R}$ be a non-negative and non-decreasing function such that the inequality 
$
h(t) \leq a[ (t/\rr )^{n} +\eps]h(\rr) + a\rr^{\beta} 
$
holds whenever $0 \leq t \leq \rr\leq r_0$, where $a>0$ and $0 < \beta < n$. For every positive $\textnormal{\texttt{b}} <n$, there exists $\eps_0 \equiv \eps_0(a,n, \beta, \textnormal{\texttt{b}})$ such that, if $\eps\leq \eps_0$, then 
$
h(t) \leq c(t/\rr )^{\textnormal{\texttt{b}}}h(\rr) + ct^{\beta}
$
holds too, whenever $0 \leq t \leq \rr\leq r_0$, where $c \equiv c (a,n,\beta, \textnormal{\texttt{b}})$. 
\end{lemma}
\begin{lemma}\label{l5.1}
Let $t>0$ and $\{\ti{v}_{i}\}_{i\in \N_{0}}\subset [0, \infty)$ be such that $\ti{v}_{i+1}\le c_{*}\textnormal{\texttt{a}}^{i}\ti{v}_{i}^{1+t}$ holds for every $i \geq 0$, with $c_{*}>0$, $\textnormal{\texttt{a}} \geq 1$ and $ t >0$. If $\ti{v}_{0}\le c_{*}^{-1/t}\textnormal{\texttt{a}}^{-1/t^{2}}$, then $\ti{v}_{i}\le \textnormal{\texttt{a}}^{-i/t}\ti{v}_{0}$ holds for every $i\geq 0$ and hence $ \ti{v}_{i}\to 0$.
\end{lemma}

\subsection{Global boundedness}\label{globou} Instrumental to the proof of Theorems \ref{t2}-\ref{t4}, is the boundedness of  minimizers. 
This proceeds via a variation of the classical De Giorgi's iteration scheme (see for instance \cite[Theorem 4.7]{BDVV}, \cite[Chapter 7]{giu}), and we report the full details for completeness in the subsequent 
\begin{proposition}\label{boundprop}
Under assumptions\eqref{assf}-\eqref{assk} and \eqref{g33}, let $u\in \mathbb{X}_{g}(\Omega)$ be as in \eqref{fun}. There exists a constant $c\equiv  c(\datab)$ such that 
$\nr{u}_{W^{1,p}(\Omega)}+\nr{u}_{L^{\infty}(\mathbb{R}^{n})}\le c\,.$ 
\end{proposition}
\begin{proof} We denote the Sobolev conjugate exponent $p^{*}$ as $p^{*}=np/(n-p)$ when $p<n$, and  $p^{*}> np/(n-1)=pn'$ when $p\geq n$.  By the minimality of $u$, Sobolev and Young's inequalities, we get, after a few standard manipulations involving in particular \eqref{assf}$_{1}$, \eqref{assph}$_{2}$ and \eqref{assk}
\begin{flalign*}
& \int_{\Omega}\snr{Du}^{p}\dx+\int_{\mathbb{R}^{n}}\int_{\mathbb{R}^{n}}\frac{\snr{u(x)-u(y)}^{\gamma}}{\snr{x-y}^{n+s\gamma}}\dx\dy\nonumber \\
&\quad \le c\int_{\Omega}\snr{Dg}^{p}\dx+c\int_{\mathbb{R}^{n}}\int_{\mathbb{R}^{n}}\frac{\snr{g(x)-g(y)}^{\gamma}}{\snr{x-y}^{n+s\gamma}}\dx\dy+c\left(\int_{\Omega}\snr{f}^{n}\dx\right)^{\frac{p}{n(p-1)}},
\end{flalign*}
for $c\equiv c(n,p,\gamma,\Lambda, \Omega)$. Note that this still holds for critical points, i.e., solutions to \rif{el}, and therefore connects to the setting of Theorem \ref{t6}; this goes via the use of \rif{usa}. Using Sobolev inequality of the left-hand side of the inequality in the above display yields
\eqn{bd.1}
$$
\|u\|^{p}_{L^{p^{*}}(\Omega)}\le c\nr{f}_{L^{n}(\Omega)}^{p/(p-1)}+c(\datab)=:\mathcal{M}\equiv \mathcal M(\datab),
$$
with $c\equiv c(n,p,\gamma,\Lambda,\Omega)$. This implies the bound $\nr{u}_{W^{1,p}(\Omega)}\le c(\datab)$. It remains to prove a similar bound for $\nr{u}_{L^{\infty}(\mathbb{R}^{n})}$. We start taking $m$ large enough to have 
\eqn{bd.0}
$$
m>\nr{g}_{L^{\infty}(\mathbb{R}^{n})}+\mathcal{M}^{1/p}+1\,.
$$
Eventually, we shall further enlarge the above lower bound on $m$. For $i\in \N_{0}$, define the increasing sequence $\{\kk_{i}\}_{i\in \N_{0}}:=\{2m(1-2^{-i-1})\}_{i\in \N_{0}}$ so that $2m\ge \kk_{i}\ge m$ holds for all $i \in \N_{0}$. By \eqref{bd.0} and $u\in \mathbb{X}_{g}(\Omega)$, we see that $v_{i}:=(u-\kk_{i})_{+}\in \mathbb{X}_{0}(\Omega)$ for all $i \in \N_{0}$. Testing \eqref{el} against $v_{i+1}$ we have
\begin{eqnarray}
\nonumber &&0=\int_{\Omega}\left[\partial_{z}F(Du)\cdot Dv_{i+1}-fv_{i+1}\right]\dx\nonumber \\
&& \qquad +\int_{\mathbb{R}^{n}}\int_{\mathbb{R}^{n}}\Phi'(u(x)-u(y))(v_{i+1}(x)-v_{i+1}(y))K(x,y)\dx\dy=:\mbox{(I)}+\mbox{(II)} \label{bd.54}
\end{eqnarray}
for every $i\geq 0$. 
Using \rif{usa}, Sobolev embedding and H\"older's inequalities yield
\begin{flalign}
\notag \mbox{(I)} & \ge\frac 1c\nr{v_{i+1}}_{L^{p^{*}}(\Omega)}^{p}-c\nr{f}_{L^{n}(\Omega)}\nr{v_{i+1}}_{L^{p^{*}}(\Omega)}\snr{\Omega\cap\{v_{i+1}>0\}}^{1/n'-1/p^{*}}\\
& \quad - c \snr{\Omega\cap\{v_{i+1}>0\}} \label{bd.5}
\end{flalign}
for $c\equiv c(n,p,\Lambda)$. To estimate term $\mbox{(II)}$, first consider the case $u(x)>\kk_{i+1}$ and $u(y)>\kk_{i+1}$, when we have, via \eqref{assph}$_{2}$
\begin{flalign*}
\Phi'(u(x)-u(y))(v_{i+1}(x)-v_{i+1}(y))&=\Phi'(v_{i+1}(x)-v_{i+1}(y))(v_{i+1}(x)-v_{i+1}(y))\nonumber \\
&\ge\Lambda^{-1}\snr{v_{i+1}(x)-v_{i+1}(y)}^{\gamma}\,.
\end{flalign*}
On the other hand, when $u(x)>\kk_{i+1}$ and $u(y)\le \kk_{i+1}$, by $\eqref{assph}_{2}$ it is
\begin{flalign*}
&\Phi'(u(x)-u(y))(v_{i+1}(x)-v_{i+1}(y))=\Phi'((u(x)-\kk_{i+1})_{+}+(\kk_{i+1}-u(y))_{+})(u(x)-\kk_{i+1})_{+}\nonumber \\
&\qquad \quad =\frac{\Phi'\left(v_{i+1}(x)+(\kk_{i+1}-u(y))_{+}\right)}{v_{i+1}(x)+(\kk_{i+1}-u(y))_{+}}\left[v_{i+1}(x)+(\kk_{i+1}-u(y))_{+}\right]v_{i+1}(x)\nonumber \\
&\qquad \quad \ge\Lambda^{-1}\snr{v_{i+1}(x)+(\kk_{i+1}-u(y))_{+}}^{\gamma-1}v_{i+1}(x)\nonumber  \ge\Lambda^{-1}[v_{i+1}(x)]^{\gamma}=\Lambda^{-1}\snr{v_{i+1}(x)-v_{i+1}(y)}^{\gamma}.
\end{flalign*}
In the opposite situation, i.e. when $u(x)\le \kk_{i+1}$ and $u(y)>\kk_{i+1}$, again by $\eqref{assph}_{2}$ we have
\begin{flalign*}
&\Phi'(u(x)-u(y))(v_{i+1}(x)-v_{i+1}(y))=-\Phi'\left(-\left((u(y)-\kk_{i+1})_++(\kk_{i+1}-u(x)_+\right)\right)v_{i+1}(y)\nonumber \\
&\qquad \quad =\frac{\Phi\left(-\left(v_{i+1}(y)+(\kk_{i+1}-u(x))_{+}\right)\right)}{v_{i+1}(y)+(\kk_{i+1}-u(x))_{+}}\left(-\left(v_{i+1}(y)+(\kk_{i+1}-u(x))_{+}\right)\right)v_{i+1}(y)\nonumber \\
&\qquad \quad \ge \Lambda^{-1}\snr{v_{i+1}(y)+(\kk_{i+1}-u(x))_{+}}^{\gamma-1}v_{i+1}(y)\nonumber  \ge \Lambda^{-1}[v_{i+1}(y)]^{\gamma}=\Lambda^{-1}\snr{v_{i+1}(x)-v_{i+1}(y)}^{\gamma}.
\end{flalign*}
Finally, when $u(x)\le \kk_{i+1}$ and $u(y)\le\kk_{i+1}$, it is
$
\Phi'(u(x)-u(y))(v_{i+1}(x)-v_{i+1}(y))=0.
$
Collecting all the above cases and recalling \eqref{assk}, leads to $\mbox{(II)}\ge 0$. Now note that 
$ v_{i}\ge v_{i+1}$ 
and $v_{i}\ge \kk_{i+1}-\kk_{i}=m /2^{i+1}$ on $\left\{v_{i+1}\geq 0\right\}$, so that $\Omega\cap\left\{v_{i+1}\geq 0\right\}\subseteq \Omega\cap \{v_{i}\ge m/2^{i+1}\}$, 
therefore we bound $$
\snr{\Omega\cap\left\{v_{i+1}>0\right\}}\leq\snr{ \Omega\cap \{v_{i}\ge m/2^{i+1}\} }\le(2^{i+1}/m )^{p^{*}}\nr{v_{i}}_{L^{p^{*}}(\Omega)}^{p^{*}}\,.
$$
Using these last inequalities in \eqref{bd.54}-\eqref{bd.5} together with $\mbox{(II)}\ge 0$, yields 
\begin{flalign}\label{bd.6}
\nr{v_{i+1}}_{L^{p^{*}}(\Omega)}^{p}&\leq c\nr{v_{i}}_{L^{p^{*}}(\Omega)}\snr{\Omega\cap \{v_{i} \geq  m/2^{i+1}\}}^{1/n'-1/p^{*}}+
c\snr{\Omega\cap \{v_{i} \geq  m/2^{i+1}\}}\nonumber \\
&\leq c(2^{i+1}/m)^{p^{*}/n'-1}\nr{v_{i}}_{L^{p^{*}}(\Omega)}^{p^{*}/n'}+ c(2^{i+1}/m)^{p^{*}}\nr{v_{i}}_{L^{p^{*}}(\Omega)}^{p^{*}}
\end{flalign}
with $c\equiv c(\datab)$. Setting $\tilde v_{i}:=m ^{-p}\nr{v_{i}}_{L^{p^{*}}(\Omega)}^{p}$, \rif{bd.1}-\rif{bd.0} imply $\tilde v_{i}\leq 1$, and \eqref{bd.6} reads as
$$
\tilde v_{i+1} \le c2^{(i+1)\left(p^{*}/n'-1\right)}m ^{1-p}\tilde v_{i}^{p^{*}/(pn')}+
c2^{(i+1) p^{*}}m ^{-p}\tilde v_{i}^{p^{*}/p} \stackrel{\eqref{bd.0}}{\le}c_*2^{ip^{*}}\tilde v_{i}^{1+t}
$$
for $c_*\equiv c_*(\datab)$ and $
t:=p^{*}/(pn')-1>0
$. 
In addition to \eqref{bd.0}, we increase $m $ in such a way that
$
m ^{p}\ge c_*^{1/t}2^{p^{*}/t^{2}}\mathcal{M}$ that implies, via \eqref{bd.1}, $\ti{v}_{0}\le c_*^{-1/t}2^{-p^{*}/t^{2}}$. Lemma \ref{l5.1} now applies and gives
$$
0=\lim_{i\to \infty}\tilde v_{i}=\lim_{i\to \infty}m ^{-p}\left(\int_{\Omega}v_{i}^{p^{*}}\dx\right)^{p/p^{*}}=m ^{-p}\left(\int_{\Omega}(u-2m )_{+}^{p^{*}}\dx\right)^{p/p^{*}}
$$
so that $\snr{\Omega\cap \left\{u>2m \right\}}=0$, and therefore $
u\le 2m$ holds a.e.\,in $\Omega$. For a lower bound, set $\hat{g}:=-g\in \mathbb{X}(\hat{g};\Omega)$, $\hat{f}:=-f\in L^{n}(\Omega)$ and consider functional $
\mathbb{X}_{\hat g}(\Omega)\ni w\mapsto \hat{\mathcal{F}}(w),
$ 
where 
$$
\hat{\mathcal{F}}(w):=\int_{\Omega}[\hat{F}(Dw)-\hat{f}w]\dx+\int_{\mathbb{R}^{n}}\int_{\mathbb{R}^{n}}\hat{\Phi}(w(x)-w(y))K(x,y)\dx\dy,
$$
$\hat{F}(z):=F(-z)$, $\hat{\Phi}(t):=\Phi(-t)$. $\hat{F}(\cdot)$ and $\hat{\Phi}(\cdot)$ satisfy \rif{assf}-\rif{assph} and $\hat{u}:=-u$ is the unique minimizer of $\hat{\mathcal{F}}(\cdot)$ in $ \mathbb{X}_{\hat g}(\Omega)$. The above argument apply to $\hat{u}$ and leads to $\hat{u}\le 2m$ a.e. in $\Omega$. All in all we have that $\snr{u}\le 2m $ a.e. in $\Omega$ and the proof is complete recalling the way $m$ has been determined.
\end{proof}
\begin{remark}
\emph{In the proof of Proposition \ref{boundprop} we do not need to assume that $p>s\gamma$; any choice of $p, \gamma>1$ and $s \in (0,1)$ suffices. Moreover, the assumption $f\in L^n(\Omega)$ can be relaxed in $f\in L^{\ti{n}}(\Omega)$, where $\ti{n}>n/p$ if $p<n$ and $\ti{n}>1$ otherwise, in that case we take $p^{*}> p\ti{n}/(\ti{n}-1)=p\ti{n}'$ when $p\geq n$. This is in accordance with the classical results for the local equation $-\Delta_p u=-\diver\, (|Du|^{p-2}Du)=f$.}
\end{remark}
\subsection{Rewriting the Euler-Lagrange equation}\label{rewrite}
Following \cite[Section 1.5]{kms1}, let us set 
\begin{flalign}
K'(x,y)& :=\begin{cases}
\ \frac{\Phi'(u(x)-u(y))K(x,y)}{\snr{u(x)-u(y)}^{\gamma-2}(u(x)-u(y))}\quad &\mbox{if} \ \ x\not =y, \ u(x)\not = u(y) \\
\ \snr{x-y}^{-n-s\gamma}\quad &\mbox{if} \  \ x\not =y, \ u(x) = u(y)\,,
\end{cases} \label{uu1}\\
\sK(x,y)&:=\frac{K'(x,y)+K'(y,x)}{2} \label{uu2}\,.
\end{flalign}
By $\eqref{assph}_{2}$, \eqref{assk} and \rif{uu1}-\rif{uu2}, it then follows that
\begin{eqnarray}\label{cacc.0}
\sK(x,y)=\sK(y,x)\quad \mbox{and}\quad \sK(x,y)\approx_{\Lambda} \frac{1}{\snr{x-y}^{n+s\gamma}}
\end{eqnarray}
hold for every $x, y \in \er^n$, provided $x\not =y$. Then, changing variables, \rif{el} can be rewritten as
\begin{flalign}
 \nonumber &\int_{\Omega}\left[\partial_{z} F(Du)\cdot D\varphi -f\varphi\right]\dx\\
 & \qquad 
+\int_{\mathbb{R}^{n}}\int_{\mathbb{R}^{n}}|u(x)-u(y)|^{\gamma-2}(u(x)-u(y))(\varphi(x)-\varphi(y)) \sK(x,y)\dx\dy=0 \label{el2}
\end{flalign}
that holds for every $\varphi\in \mathbb{X}_{0}(\Omega)$. From now on, we shall use \rif{el2} instead of \rif{el}. 
\section{Integral quantities measuring oscillations}\label{lalista}
In this section we fix two generic functions $w$ and $f$, such that, unless otherwise specified, $w \in W^{1,p}(\Omega)\cap W^{s, \gamma}(\er^n)$ and $f \in L^n(\er^n)$, and an arbitrary ball $B_{\rr}(x_{0})\subset \mathbb{R}^{n}$. We are going to list a number of basic quantities that will play an important role in this paper. In most of the times, such quantities give an integral measure of the oscillations of a function $w$ in $B_{\rr}(x_{0})$ or in its complement.  
A fundamental tool in the regularity theory of fractional problems is the nonlocal tail, first introduced in \cite{DKP}, which, in some sense, keeps track of long range interactions. In \cite{brasco1}, a related  nonlocal quantity, called \texttt{snail}, was considered, namely 
\eqn{snail0}
$$
\left(\rr^{s\gamma}\int_{\mathbb{R}^{n}\setminus B_{\rr}(x_{0})}\frac{\snr{w(x)}^{\gamma}}{\snr{x-x_{0}}^{n+s\gamma}} \dx\right)^{1/\gamma}\,.
$$
The \texttt{snail} can be essentially seen as the $L^{\gamma}$-average of $\snr{w}$ on $\mathbb{R}^{n}\setminus B_{\rr}(x_{0})$ with respect to the measure defined by $\d\lambda_{x_0}:=\snr{x-x_{0}}^{-n-s\gamma}\dx$.
We refer to \cite{brasco1,brasco2,brasco3,DKP,kkp,kms1,pala} for extra details on this matter. In this paper we use a Campanato-type variation of \rif{snail0}, that is
\eqn{snaildef}
$$
\snail_{\delta}(\rr)\equiv \snail_{\delta}(w,B_{\rr}(x_{0})):=\left(\rr^{\delta}\int_{\mathbb{R}^{n}\setminus B_{\rr}(x_{0})}\frac{\snr{w(y)-(w)_{B_{\rr}(x_{0})}}^{\gamma}}{\snr{x-x_{0}}^{n+s\gamma}} \dy\right)^{1/\gamma},\qquad \delta\ge s\gamma.
$$
Note that
\eqn{snail-l}
$$
\snail_{\delta}(w,B_{\rr}(x_{0})) \leq c(n,s,\gamma)r^{\delta/\gamma-s} \|w\|_{L^{\infty}(\er^n)} \qquad  \forall \ \rr  \leq r <  \infty\,.
$$
This clearly involves the oscillations of $u$ and it is a nonlocal version of the more classical object $$
\av_{\texttt{q}}(w,B_{\rr}(x_{0})):=\left(\mint_{B_{\rr}(x_{0})}\snr{w-(w)_{B_{\rr}(x_{0})}}^{\texttt{q}}\dx\right)^{1/\texttt{q}} \,,\qquad \mbox{ $\texttt{q}>0$}\,.
$$
The right notion of excess functional now combines the previous two quantities, i.e., 
\eqn{eccessi}
$$
\ecc(\rr) \equiv \ecc(w,B_{\rr}(x_{0})) :=\av_{p}(w,B_{\rr}(x_{0}))+\left[\snail_{\delta}(w,B_{\rr}(x_{0}))\right]^{\gamma/p}\,.$$
With $\theta \in (0,1)$ and $\delta\ge s\gamma$, we further define
\begin{flalign}
&[\rhs(\rr)]^{p}\equiv \left[\rhs(B_{\varrho}(x_0))\right]^{p}:=\rr^{p-\theta}\left(\nr{f}_{L^{n}(B_{\rr}(x_{0}))}^{p/(p-1)}+1\right)\label{rightdef} \\ 
&\caccs(\rr)\equiv \caccs(w,B_{\varrho}(x_0)):=\rr^{-p}[\av_{p}(w,B_{\rr}(x_{0}))]^{p}+\rr^{-s\gamma}[\av_{\gamma}(w,B_{\rr}(x_{0}))]^{\gamma}\notag \\
& \hspace{45mm}+\rr^{-\delta}[\snail_{\delta}(w,B_{\rr}(x_{0}))]^{\gamma}+\nr{f}_{L^{n}(B_{\rr}(x_{0}))}^{p/(p-1)}+1 \label{caccica}
\end{flalign}
\eqn{caccdef}
$$
\cacc(\rr)\equiv \cacc(w, B_\rr(x_0)):= \rr^{-p}[\av_{p}(\rr)]^{p}+\rr^{-\delta}[\snail_{\delta}(\rr)]^{\gamma}+\nr{f}_{L^{n}(B_{\rr})}^{p/(p-1)}+1$$
\eqn{ggdef}
$$
 [\GG(w,B_{\rr}(x_{0}))]^{p}:=[\ecc(w,B_{\rr}(x_{0}))]^p+[\rhs(B_{\rr}(x_{0}))]^{p}.
$$
Note that 
\eqn{triviala}
$$
p\geq \delta, \rr \leq 1 \Longrightarrow \rr^p \, \cacc(\rr) \leq [\GG(\rr)]^{p} \,.
$$
Abbreviations above such as $\av_{p}(\rr)\equiv \av_{p}(w,B_{\rr}(x_{0}))$, $\caccs(\rr)\equiv \caccs(w,B_{\varrho}(x_0))$, and the like, will be made in the following whenever there will be no ambiguity on what $w$ and $B_{\rr}(x_{0})$ are. Of course all the quantities defined above also depend on $f$, but this dependence will be omitted as it will be clear from the context. The motivation for the notation above is that terms of the type $\rhs(\cdot)$ appear as right-sides quantities of certain inequalities related to equations as in \rif{el}. Terms of the type $\cacc(\cdot)$ will instead occur in certain Caccioppoli type inequalities.
\begin{lemma}
Let $B_{ t}(x_{0})\subset B_{\rr}(x_{0})$ be two concentric balls, $\gamma\ge 1$, $\delta\ge s\gamma$ and $w\in W^{s,\gamma}(\mathbb{R}^{n})$.
\begin{itemize}
\item Whenever $0< t<\rr\le 1$, it holds that
\begin{flalign}
\notag \snail_{\delta}(w,B_{ t}(x_{0}))&\le c\left(\frac{t}{\rr}\right)^{\delta/\gamma}\snail_{\delta}(w,B_{\rr}(x_{0}))+c t^{\delta/\gamma-s}\int_{ t}^{\rr}\left(\frac{ t}{\nu}\right)^{s}\av_{\gamma}(w,B_{\nu}(x_{0})) \, \frac{\dtau}{\nu}\nonumber \\
&\qquad +c t^{\delta/\gamma-s}\left(\frac{t}{\rr}\right)^{s}\av_{\gamma}(w,B_{\rr}(x_{0})),\label{scasnail}
\end{flalign}
with $c\equiv c(n,s,\gamma)$.
\item With $\textnormal{\texttt{q}}\geq 1$, if $\nu >0$ and $\theta \in (0,1)$ are such that $\theta \varrho \leq \nu \leq \rr$, then 
\eqn{averages}
$$
\av_{\textnormal{\texttt{q}}}(w,B_{\nu}(x_{0})) \leq 2\theta^{-n/\textnormal{\texttt{q}}}\av_{\textnormal{\texttt{q}}}(w,B_{\rr}(x_{0}))\,.
$$
\end{itemize}
\end{lemma}
\begin{proof}
In the following all the balls will be centred at $x_0$. Let us first recall the standard property 
\eqn{basicav}
$$
\left(\mint_{B_{\rr}}\snr{w-(w)_{B_{\rr}}}^{\textnormal{\texttt{q}}}\dx\right)^{1/\textnormal{\texttt{q}}} \leq 2
\left(\mint_{B_{\rr}}\snr{w-\texttt{w}}^{\textnormal{\texttt{q}}}\dx\right)^{1/\textnormal{\texttt{q}}}
$$
that holds whenever $\texttt{w} \in \er$ and $\textnormal{\texttt{q}}\geq 1$; from this \rif{averages} follows immediately. For the proof of \rif{scasnail}, we shall use a few arguments developed in \cite{kms1}. Let $B_{ t} \subset B_{\rr}$, we then split
\begin{flalign}
\notag \snail_{\delta}( t)&\le c\left(\frac{t}{\rr}\right)^{\delta/\gamma}\snail_{\delta}(\rr)+c t^{\delta/\gamma-s}\left(\frac{t}{\rr}\right)^{s}\snr{(w)_{B_{t}}-(w)_{B_{\rr}}}\nonumber \\
&\qquad+c\left( t^{\delta}\int_{B_{\rr}\setminus B_{ t}}\frac{\snr{w(x)-(w)_{B_{t}}}^{\gamma}}{\snr{x-x_{0}}^{n+s\gamma}}  \dx\right)^{1/\gamma}=:c\left(\frac{t}{\rr}\right)^{\delta/\gamma}\snail_{\delta}(\rr)+cT_{1}+cT_{2},\label{tt1}
\end{flalign}
where $c\equiv c(n,s,\gamma)$. We have used 
\eqn{misura}
$$\d\lambda_{x_0}(\er^n\setminus B_t) = ct^{-s\gamma}\,, \quad \quad \d\lambda_{x_0}(x):= \frac{\dx}{\snr{x-x_{0}}^{n+s\gamma}} \,, $$ where $c\equiv c (n,s,\gamma) $. 
If $ \rr/4\leq t <\rr$, also using this last identity, standard manipulations based on \rif{averages} ensure that
$
T_{1}+T_{2}\le c t^{\delta/\gamma-s}(t/\rr)^{s}\av_{\gamma}(\rr) 
$
holds with $c\equiv c(n,s,\gamma)$. We can therefore assume that $ t< \rr/4$. This means that there exists $\lambda\in \left(1/4,1/2\right)$ and $\kk\in \mathbb{N}$, $\kk\ge 2$ so that $ t=\lambda^{\kk}\rr$. Using triangle and H\"older's inequalities, we estimate, using \rif{averages}-\rif{basicav} repeatedly
\begin{flalign*}
T_{1}&\le  t^{\delta/\gamma-s}\left(\frac{t}{\rr}\right)^{s}\snr{(w)_{B_{\lambda \rr}}-(w)_{B_{\rr}}}+ t^{\delta/\gamma-s}\left(\frac{t}{\rr}\right)^{s}\snr{(w)_{B_{\lambda \rr}}-(w)_{B_{\lambda^{\kk}\rr}}}\nonumber \\
&\le c t^{\delta/\gamma-s}\left(\frac{t}{\rr}\right)^{s}\av_{\gamma}(\rr)\nonumber+ t^{\delta/\gamma-s}\left(\frac{t}{\rr}\right)^{s}\sum_{i=1}^{\kk-1}\snr{(w)_{B_{\lambda^{i}\rr}}-(w)_{B_{\lambda^{i+1}\rr}}}\nonumber \\
&\le c t^{\delta/\gamma-s}\left(\frac{t}{\rr}\right)^{s}\av_{\gamma}(\rr)+c t^{\delta/\gamma-s}\left(\frac{t}{\rr}\right)^{s}\sum_{i=1}^{\kk-1}\left(\mint_{B_{\lambda^{i}\rr}}\snr{w(x)-(w)_{B_{\lambda^{i}\rr}}}^{\gamma}\dx\right)^{1/\gamma}\nonumber \\
&\le c t^{\delta/\gamma-s}\left(\frac{t}{\rr}\right)^{s}\av_{\gamma}(\rr)+c t^{\delta/\gamma-s}\left(\frac{t}{\rr}\right)^{s}\sum_{i=1}^{\kk-1}\int_{\lambda^{i}\rr}^{\lambda^{i-1}\rr}\av_{\gamma}(\lambda^{i}\rr) \, \frac{\dtau}{\nu}\nonumber \\
&\le c t^{\delta/\gamma-s}\left(\frac{t}{\rr}\right)^{s}\av_{\gamma}(\rr)+c t^{\delta/\gamma-s}\left(\frac{t}{\rr}\right)^{s}\sum_{i=1}^{\kk-1}\int_{\lambda^{i}\rr}^{\lambda^{i-1}\rr}\av_{\gamma}(\nu) \, \frac{\dtau}{\nu}\nonumber \\
&\le c t^{\delta/\gamma-s}\left(\frac{t}{\rr}\right)^{s}\av_{\gamma}(\rr)+c t^{\delta/\gamma-s}\left(\frac{t}{\rr}\right)^{s}\int_{ t}^{\rr}\av_{\gamma}(\nu) \, \frac{\dtau}{\nu},
\end{flalign*}
with $c\equiv c(n,s,\gamma)$. For $T_{2}$, we rewrite $\rr=\lambda^{-\kk} t$ and estimate, by telescoping and Jensen's inequality
\eqn{ttt}
$$
\left(\mint_{B_{\lambda^{-i} t}}\snr{w(x)-(w)_{B_{ t}}}^{\gamma}\dx\right)^{1/\gamma}\le 2^{n/\gamma+1}\sum_{m=0}^{i}\av_{\gamma}(\lambda^{-m} t),
$$
for $0 \leq i \leq k$. Then, via \eqref{basicav}, \eqref{ttt} and the discrete Fubini theorem, we obtain
\begin{flalign*}
T_{2}&\le c t^{\delta/\gamma-s}\left(\sum_{i=0}^{\kk-1}\lambda^{i s\gamma}(\lambda^{-i} t)^{-n}\int_{B_{\lambda^{-i-1} t}\setminus B_{\lambda^{-i} t}}\snr{w(x)-(w)_{B_{ t}}}^{\gamma}\dx\right)^{1/\gamma}\nonumber \\
&\le c t^{\delta/\gamma-s}\sum_{i=0}^{\kk}\left(\lambda^{i s\gamma}\mint_{B_{\lambda^{-i} t}}\snr{w(x)-(w)_{B_{ t}}}^{\gamma}\dx\right)^{1/\gamma}\nonumber \\
&\le c t^{\delta/\gamma-s}\sum_{i=0}^{\kk}\lambda^{i s}\sum_{m=0}^{i}\av_{\gamma}(\lambda^{-m} t)\\ 
&= c t^{\delta/\gamma-s}\sum_{m=0}^{\kk}\av_{\gamma}(\lambda^{-m} t) \sum_{i=m}^{\kk}\lambda^{i s}\nonumber \\
&\le c t^{\delta/\gamma-s}\sum_{m=0}^{\kk}\lambda^{m s}\av_{\gamma}(\lambda^{-m} t)\nonumber \\
&\le c t^{\delta/\gamma-s}\sum_{m=0}^{\kk-1}\int_{\lambda^{-m} t}^{\lambda^{-m-1} t}\lambda^{m s}\av_{\gamma}(\nu) \, \frac{\dtau}{\nu}+c t^{\delta/\gamma-s}\left(\frac{t}{\rr}\right)^{s}\av_{\gamma}(\rr)\nonumber \\
&\le c t^{\delta/\gamma-s}\sum_{m=0}^{\kk-1}\int_{\lambda^{-m} t}^{\lambda^{-m-1} t}\left(\frac{ t}{\nu}\right)^{s}\av_{\gamma}(\nu) \, \frac{\dtau}{\nu}+c t^{\delta/\gamma-s}\left(\frac{t}{\rr}\right)^{s}\av_{\gamma}(\rr)\nonumber \\
&\le c t^{\delta/\gamma-s}\int_{ t}^{\rr}\left(\frac{ t}{\nu}\right)^{s}\av_{\gamma}(\nu) \, \frac{\dtau}{\nu}+c t^{\delta/\gamma-s}\left(\frac{t}{\rr}\right)^{s}\av_{\gamma}(\rr),
\end{flalign*}
for $c\equiv c(n,s,\gamma)$. Merging the estimates found for $T_{1}$ and $T_{2}$ to \rif{tt1}, we obtain \eqref{scasnail}.
\end{proof}
\begin{lemma}
Let $w\in L^{\infty}(\mathbb{R}^{n})$ and $B_{t}(x_{0})\subset \er^n$ be a ball. Then
\eqn{cacc.2}
$$
 \int_{\mathbb{R}^{n}\setminus B_{t}}\frac{\snr{w(y)}^{\gamma-1}}{\snr{y-x_{0}}^{n+s\gamma}} \dy  
\le \frac{c}{t^{s}}\left(\int_{\mathbb{R}^{n}\setminus B_{\rr}}\frac{\snr{w(y)}^{\gamma}}{\snr{y-x_{0}}^{n+s\gamma}} \dy\right)^{1-1/\gamma}\,,
$$
where $c\equiv c (n,s,\gamma)$. 
\end{lemma}
\begin{proof} By \rif{misura} note that 
$$
 \int_{\mathbb{R}^{n}\setminus B_{t}}\frac{\snr{w(y)}^{\gamma-1}}{\snr{y-x_{0}}^{n+s\gamma}} \dy  = \frac{c}{ t^{s\gamma}}\mint_{\mathbb{R}^{n}\setminus B_{t}}\snr{w(y)}^{\gamma-1} \ \d\lambda_{x_{0}}(y)
$$
and apply Jensen's inequality with resect to the concave function $\tau  \mapsto \tau^{1-1/\gamma}$. 
\end{proof}
\section{Proof of Theorems \ref{t2} and \ref{t5}}\label{inthold}
The main steps of the proofs of Theorems \ref{t2} and \ref{t5} are contained in Sections \ref{ih1}-\ref{intit} below, where we permanently assume \rif{assf}-\rif{assk} and \rif{g33} and $u$ is as in \rif{fun}. Any ball $B_{\rr} \equiv B_{\rr}(x_0)\Subset \Omega$ will be such that $\rr\leq 1$. We yet introduce the notation
\eqn{datta}
$$\datar:=\left(n,p,s,\gamma,\Lambda,\|u\|_{L^\infty}^{1-\vartheta}\right)\,,$$ where $\vartheta$ is in \rif{t1t2}, i.e., no dependence on $\|u\|_{L^\infty}$ occurs in $\datar$ when $\gamma \leq p $ and therefore $\vartheta=1$.
\subsection{Step 1: Basic Caccioppoli inequality}\label{ih1}
This is in the following:
\begin{lemma}\label{cacclem}
The inequality 
\eqn{caccp}
$$
 \mint_{B_{\rr/2}(x_{0})}(\snr{Du}^{2}+\mu^{2})^{p/2}\dx +\int_{B_{\rr/2}(x_{0})}\mint_{B_{\rr/2}(x_{0})}\frac{\snr{u(x)-u(y)}^{\gamma}}{\snr{x-y}^{n+s\gamma}}\dx\dy
\leq c\, \caccs(u,B_{\varrho}(x_0)) 
$$
holds whenever $B_{\rr}\equiv B_{\rr}(x_0)\Subset \Omega$ with $\rr\in (0,1]$, where $c\equiv c(n,p,s,\gamma, \Lambda)$ and $\delta\ge s\gamma$.
\end{lemma}
\begin{proof}
All the balls will be centred at $x_{0}$. We denote $\um:=u-(u)_{B_{\rr}}$, fix $\eta\in C^{1}_{0}(B_{\rr})$ such that
$
\mathds{1}_{B_{\rr/2}}\le \eta\le \mathds{1}_{B_{3\rr/4}} $ and $\snr{D\eta}\lesssim 1/\rr,
$ and set $m:=\max\{p,\gamma\}$. Note that $\varphi:=\eta^{m}\um\in \mathbb{X}_{0}(\Omega)$, so that it can be used in \eqref{el2}; this yields
\begin{flalign*}
0&=\mint_{B_{\rr}}\left[\partial_{z}F(Du)\cdot D(\eta^{m}\um)-\eta^{m}f\um\right]\dx\nonumber \\
&\qquad  +\snr{B_{\rr}}^{-1}\int_{\mathbb{R}^{n}}\int_{\mathbb{R}^{n}}|u(x)-u(y)|^{\gamma-2}(u(x)-u(y))(\eta^{m}(x)\um(x)-\eta^{m}(y)\um(y))\sK(x,y)\dx\dy\\
& =:\mbox{(I)}+\mbox{(II)}.
\end{flalign*}
The estimation of (I) goes via \rif{usa} and Young and Sobolev inequalities as follows:
\begin{flalign*}
 \mbox{(I)} &\geq c\mint_{B_{\rr}}\eta^{m}(\snr{Du}^{2}+\mu^{2})^{p/2}\dx-c\rr^{-p}\mint_{B_{\rr}}|\um|^{p}\dx-c-\left(\mint_{B_{\rr}}\snr{f}^{n}\dx\right)^{1/n}\left(\mint_{B_{\rr}}|\eta^{m}\um|^{p^{*}}\dx\right)^{1/p^{*}}\nonumber \\
&\ge c \mint_{B_{\rr}}\eta^{m}(\snr{Du}^{2}+\mu^{2})^{p/2}\dx-c\rr^{-p}\mint_{B_{\rr}}|\um|^{p}\dx-c\nr{f}_{L^{n}(B_{\rr})}\left(\mint_{B_{\rr}}\snr{D(\eta^{m}\um)}^{p}\dx\right)^{1/p}-c\nonumber \\
&\geq c\mint_{B_{\rr}}\eta^{m}(\snr{Du}^{2}+\mu^{2})^{p/2}\dx-c\rr^{-p}[\av_{p}(\rr)]^{p}-c\nr{f}_{L^{n}(B_{\rr})}^{p/(p-1)}-c
\end{flalign*}
with $c\equiv c(n,p,\Lambda)$. Here $p^*$ is the Sobolev conjugate exponent as described at the beginning of the proof of Proposition \ref{boundprop}. Using \rif{cacc.0} we find
\begin{flalign*} 
\mbox{(II)}
&=\int_{B_{\rr}}\mint_{B_{\rr}}\snr{\um(x)-\um(y)}^{\gamma-2}(\um(x)-\um(y))(\eta^{m}(x)\um(x)-\eta^{m}(y)\um(y))\sK(x,y)\dx\dy\nonumber \\
& \ +2\int_{\mathbb{R}^{n}\setminus B_{\rr}}\mint_{B_{\rr}}\snr{\um(x)-\um(y)}^{\gamma-2}(\um(x)-\um(y))\eta^{m}(x)\um(x)\sK(x,y)\dx\dy\nonumber \\
&=:\mbox{(II)}_{1}+\mbox{(II)}_{2}\,.
\end{flalign*}
We now observe that 
\begin{flalign}
\mbox{(II)}_{1}&\ge \frac 1c\int_{B_{\rr}}\mint_{B_{\rr}}\frac{\snr{\eta^{m/\gamma}(x)\um(x)-\eta^{m/\gamma}(y)\um(y)}^{\gamma}}{\snr{x-y}^{n+s\gamma}}\dx\dy\nonumber \\
&\quad -c\int_{B_{\rr}}\mint_{B_{\rr}}\frac{\max\{|\um(x)|,|\um(y)|\}^{\gamma}\snr{\eta^{m/\gamma}(x)-\eta^{m/\gamma}(y)}^{\gamma}}{\snr{x-y}^{n+s\gamma}}\dx\dy \label{dimisi00}
\end{flalign}
where $c\equiv c (p,\gamma)$. 
Indeed, let us set $\mathcal T(x, y):= \snr{\um(x)-\um(y)}^{\gamma-2}(\um(x)-\um(y))(\eta^{m}(x)\um(x)-\eta^{m}(y)\um(y))$. We first consider the case
$\eta(x)\geq \eta(y)$ and rewrite 
$
\mathcal T(x, y) =\mathcal T_1(x, y)+\mathcal T_2(x, y), 
$
where 
$\mathcal T_1(x, y):= \snr{\um(x)-\um(y)}^{\gamma}\eta^{m}(x)$ 
and 
$\mathcal T_2(x, y):=\snr{\um(x)-\um(y)}^{\gamma-2}(\um(x)-\um(y))(\eta^{m}(x)-\eta^{m}(y))\um(y).$ 
Mean Value Theorem yields
$|\mathcal T_2(x, y)| \leq c |\eta^{m(\gamma-1)/\gamma}(x)|
\snr{\eta^{m/\gamma}(x)-\eta^{m/\gamma}(y)}\snr{\um(x)-\um(y)}^{\gamma-1}|\um(y)|$ and, by Young's inequality, we obtain
$
\mathcal T_1(x, y) \leq c\mathcal T(x, y)+ c\snr{\eta^{m/\gamma}(x)-\eta^{m/\gamma}(y)}^{\gamma}|\um(y)|^{\gamma}
$. When $\eta(x)<\eta(y)$, we note that $\mathcal T(x, y)=\mathcal T(y,x)$ and exchanging  the role of $x$ and $y$ in the above argument, in any case we conclude with  
$
\snr{\um(x)-\um(y)}^{\gamma}\eta^{m}(x) \leq c\mathcal T(x, y)+c\max\{|\um(x)|,|\um(y)|\}^{\gamma}\snr{\eta^{m/\gamma}(x)-\eta^{m/\gamma}(y)}^{\gamma}, 
$ with $c\equiv c(p,\gamma)$. 
From this and triangle inequality \rif{dimisi00} follows via easy manipulations; in turn, \rif{dimisi00} implies
\begin{flalign}
\mbox{(II)}_{1}
&\geq \frac 1c\int_{B_{\rr}}\mint_{B_{\rr}}\frac{\snr{\eta^{m/\gamma}(x)\um(x)-\eta^{m/\gamma}(y)\um(y)}^{\gamma}}{\snr{x-y}^{n+s\gamma}}\dx\dy\nonumber \\
&\quad -c\rr^{-\gamma}\int_{B_{\rr}}\mint_{B_{\rr}}\frac{\max\{|\um(x)|,|\um(y)|\}^{\gamma}}{\snr{x-y}^{n+\gamma(s-1)}}\dx\dy\nonumber \\
&\geq \frac 1c\int_{B_{\rr}}\mint_{B_{\rr}}\frac{\snr{\eta^{m/\gamma}(x)\um(x)-\eta^{m/\gamma}(y)\um(y)}^{\gamma}}{\snr{x-y}^{n+s\gamma}}\dx\dy-c\rr^{-s\gamma}[\av_{\gamma}(\rr)]^{\gamma}\nonumber \\
&\geq \frac 1c\int_{B_{\rr/2}}\mint_{B_{\rr/2}}\frac{\snr{u(x)-u(y)}^{\gamma}}{\snr{x-y}^{n+s\gamma}}\dx\dy-c \, \caccs(\rr), \label{dimisi}
\end{flalign}
for $c\equiv c(n,p,s,\gamma,\Lambda)$. For (II)$_{2}$, note that
\eqn{cacc.1}
$$
x\in B_{3\rr/4}, \ y\in \mathbb{R}^{n}\setminus B_{\rr} \ \Longrightarrow \ 1\leq \frac{\snr{y-x_{0}}}{\snr{x-y}} \le 4
$$
and then, recalling that $\eta$ is supported in $B_{3\rr/4}$, we have
\begin{eqnarray*}
\snr{\mbox{(II)}_{2}}&\stackrel{\eqref{cacc.0}}{\le}&c\int_{\mathbb{R}^{n}\setminus B_{\rr}}\mint_{B_{\rr}}\frac{\snr{\um(x)-\um(y)}^{\gamma-1}|\um(x)|\eta^{m}(x)}{\snr{x-y}^{n+s\gamma}}\dx\dy\nonumber \\
&\stackrel{\eqref{cacc.1}}{\le}&c\int_{\mathbb{R}^{n}\setminus B_{\rr}}\mint_{B_{\rr}}\frac{\max\left\{\snr{\um(x)},\snr{\um(y)}\right\}^{\gamma-1}|\um(x)|}{\snr{y-x_0}^{n+s\gamma}}\dx\dy\nonumber \\
&\leq &c\rr^{-s\gamma}\mint_{B_{\rr}}\snr{\um}^{\gamma}\dx\nonumber +c\int_{\mathbb{R}^{n}\setminus B_{\rr}}\frac{\snr{\um(y)}^{\gamma-1}}{\snr{y-x_{0}}^{n+s\gamma}} \dy\left(\mint_{B_{\rr}}|\um|^{\gamma}\dx\right)^{1/\gamma}\nonumber \\
&\stackrel{\eqref{cacc.2}}{\le}&c\rr^{-s\gamma}[\av_{\gamma}(\rr)]^{\gamma}\nonumber +c\left(\int_{\mathbb{R}^{n}\setminus B_{\rr}}\frac{\snr{\um(y)}^{\gamma}}{\snr{y-x_{0}}^{n+s\gamma}} \dy\right)^{1-1/\gamma}\rr^{-s}\av_{\gamma}(\rr)\nonumber \\
&\leq &\rr^{-s\gamma}[\av_{\gamma}(\rr)]^{\gamma}+c\rr^{-\delta}[\snail_{\delta}(\rr)]^{\gamma}\leq c \, \caccs(\rr),
\end{eqnarray*}
whenever $\delta\ge s\gamma$, and where $c\equiv c(n,s,\gamma,\Lambda)$. Combining the estimates for the terms $\mbox{(I)}$-$\mbox{(II)}$, and recalling that $\eta\equiv 1$ on $B_{\rr/2}$,  we arrive at \eqref{caccp}. \end{proof}
\subsection{Step 2: Localization}\label{clp}  We define $h\in u+W^{1,p}_{0}(B_{\rr/4}(x_{0}))$ as the (unique) solution to
\eqn{pdd}
$$
h\mapsto \min_{w \in u+W^{1,p}_{0}(B_{\rr/4}(x_{0}))} \int_{B_{\rr/4}(x_{0})}F(Dw)\dx\,.
$$
The function $h$ solves the Euler-Lagrange equation
\eqn{elpd}
$$
\int_{B_{\rr/4}(x_{0})}\partial_{z}F(Dh)\cdot D\varphi\dx=0\qquad \mbox{for every} \ \varphi\in W^{1,p}_{0}(B_{\rr/4})\,.
$$
Moreover, by minimality of $h$, \rif{assf}$_1$ and \eqref{caccp} we gain
\eqn{enes}
$$
\mint_{B_{\rr/4}}(\snr{Dh}^{2}+\mu^{2})^{p/2}\dx \le \Lambda^{2}\mint_{B_{\rr/4}}(\snr{Du}^{2}+\mu^{2})^{p/2}\dx \leq 
c \, \caccs(\rr) 
$$
with $c \equiv c(n,p,s,\gamma,\Lambda)$. The standard Maximum Principle 
\eqn{bhh}
$$
 \nr{h}_{L^{\infty}(B_{\rr/4})}\le \nr{u}_{L^{\infty}(B_{\rr/4})}\,.
$$
This last inequality is only going to be used when $\gamma > p$, that is when we know that the right-hand side is finite by Proposition \ref{boundprop}. Finally, we recall the $L^\infty$-$L^p$ inequality for $p$-harmonic type functions (see \cite{manth1, manth2})
\eqn{7}
$$
\nr{Dh}_{L^{\infty}(B_{\rr/8})}^p\le c\mint_{B_{\rr/4}}(\snr{Dh}^{2}+\mu^{2})^{p/2}\dx \stackrel{\rif{enes}}{\leq} 
c \, \caccs(\rr) 
$$
that holds with $c\equiv c(n,p,s,\gamma,\Lambda)$. 
\begin{lemma}\label{har}
Let $h\in u+W^{1,p}_{0}(B_{\rr/4}(x_{0}))$ be as in \eqref{pdd}. There exists $\sigma\equiv \sigma(p,s,\gamma)\in (0,1)$ such that
\eqn{cl.11}
$$
\mint_{B_{\rr/4}(x_{0})}\snr{u-h}^{p}\dx \le c\rr^{\theta\sigma}[\GG(u,B_{\rr}(x_{0}))]^{p}
$$
holds for every $\theta \in (0,1)$, where $c\equiv c(\datar)$ and $\datar$ is defined in \eqref{datta}.
\end{lemma}
\begin{proof}
We are going to use Lemma \ref{cacclem} with 
\eqn{dp}
$$
\delta\in (s\gamma,p)
$$
in \rif{snaildef}, which makes sense by $p>s\gamma$. We keep this choice until the end of the proof of Theorem \ref{t2}; later on, in Step 3,  we shall choose $\delta$ suitably close to $p$. We preliminary observe that 
\eqn{caccg}
$$
\caccs(\rr)\le  c\, \cacc(\rr)
$$
holds with $c\equiv c(\datar)$. Indeed, recalling \rif{caccica}-\rif{caccdef}, it is sufficient to estimate the term $\rr^{-s\gamma}[\av_{\gamma}(\rr)]^{\gamma}$ appearing in the definition of $\caccs(\rr)$; for this, still denoting $\av_{\textnormal{\texttt{q}}}(t)\equiv \av_{\textnormal{\texttt{q}}}(u, B_t(x_0))$ for every $\texttt{q}>0$ and $t \leq \rr$, observe that
\begin{flalign}
\nonumber \rr^{-s\gamma}[\av_{\gamma}(\rr)]^{\gamma} &\leq c \|u\|_{L^\infty(B_{\rr})}^{(1-\vartheta)\gamma} \varrho^{-s\gamma}
[\av_{\vartheta\gamma}(\rr)]^{\vartheta\gamma}\\ 
\nonumber&\leq c\|u\|_{L^\infty(B_{\rr})}^{(1-\vartheta)\gamma}\rr^{(\vartheta-s)\gamma}
[ \varrho^{-p}[\av_{p}(\rr)]^p]^{\vartheta\gamma/p}
\\ & \leq c\|u\|_{L^\infty(B_{\rr})}^{(1-\vartheta)\gamma}\rr^{(\vartheta-s)\gamma}
[\cacc(\rr)]^{\vartheta\gamma/p} \leq  c\,\cacc(\rr) \,, \label{intermedia}
\end{flalign}
from which
\rif{caccg} follows, with the required dependence of the constants (recall Proposition \ref{boundprop} in the case $\gamma >p$); we have used \rif{centralass2} and that $\cacc(\rr) \geq 1\geq \rr$. We now extend $h\equiv u$ outside $B_{\rr/4}$, thereby getting, in particular, that $h \in W^{1,p}(\Omega)$, and in addition, when $\gamma >p$, we have $h \in L^{\infty}(\er^n)$ by Proposition \ref{boundprop} and \rif{bhh}. If we set $w:=u-h$, then $w \in  W^{1,p}_0(B_{\varrho})$, and also $w \in  L^{\infty}(B_{\varrho})$ when $\gamma >p$. Lemma \ref{bmfinal} implies $w \in  W^{s,\gamma}(B_{\varrho})$ and, since $w \equiv 0$ in $B_{\rr}\setminus B_{\rr/4}$, by \cite[Lemma 5.1]{guide} it follows that $w \in W^{s, \gamma}(\er^n)$. In this way  $w\in \mathbb{X}_{0}(\Omega)$ and can be used as a test function both in \eqref{el2} and \rif{elpd}. 
Setting 
$\mathcal{V}^{2}:=\snr{V_{\mu}(Du)-V_{\mu}(Dh)}^{2}$, with $V_{\mu}(\cdot)$ being defined in \rif{vpvqm}, we have
\begin{eqnarray}
&& \mint_{B_{\rr/4}}\mathcal{V}^{2}\dx\stackrel{\eqref{monoin}}{\le} c\mint_{B_{\rr/4}}(\partial_{z}F(Du)-\partial_{z}F(Dh))\cdot Dw\dx\nonumber \stackrel{\eqref{elpd}}{=}c\mint_{B_{\rr/4}}\partial_{z}F(Du)\cdot Dw\dx\nonumber \\
&&\stackrel{\eqref{el2}}{=}c\mint_{B_{\rr/4}}fw\dx-c\int_{B_{\rr/2}}\mint_{B_{\rr/2}}\snr{u(x)-u(y)}^{\gamma-2}(u(x)-u(y))(w(x)-w(y))\sK(x,y)\dx\dy\nonumber \\
&& \qquad -2c\int_{\mathbb{R}^{n}\setminus B_{\rr/2}}\mint_{B_{\rr/2}}\snr{u(x)-u(y)}^{\gamma-2}(u(x)-u(y))w(x)\sK(x,y)\dx\dy\nonumber \\
&&\ \,  =:\mbox{(I)}+\mbox{(II)}+\mbox{(III)}\,,\label{daje}
\end{eqnarray}
where $c\equiv c(n,p,\Lambda)$. H\"older and Sobolev inequalities (as in Lemma \ref{cacclem}) yield
\begin{eqnarray}
\nonumber \snr{\mbox{(I)}}&\le &\nr{f}_{L^{n}(B_{\rr/4})}\left(\mint_{B_{\rr/4}}\snr{Dw}^{p}\dx\right)^{1/p} \\
\nonumber &\stackrel{\rif{enes}}{\le}& c\nr{f}_{L^{n}(B_{\rr/4})}[\caccs(\rr)]^{1/p}\\ &
\stackrel{\rif{caccg}}{\le} & c\nr{f}_{L^{n}(B_{\rr/4})}[\cacc(\rr)]^{1/p},
\label{cl.3}
\end{eqnarray}
with $c\equiv c(n,p,s,\gamma, \Lambda)$. Again by H\"older's inequality, it is
\begin{eqnarray}
\snr{\mbox{(II)}}&\le&c\left(\int_{B_{\rr/2}}\mint_{B_{\rr/2}}\frac{\snr{u(x)-u(y)}^{\gamma}}{\snr{x-y}^{n+s\gamma}}\dx\dy\right)^{1-1/\gamma}\left(\int_{B_{\rr/4}}\mint_{B_{\rr/4}}\frac{\snr{w(x)-w(y)}^{\gamma}}{\snr{x-y}^{n+s\gamma}}\dx\dy\right)^{1/\gamma}\nonumber \\
&\stackrel{\eqref{caccp}, \eqref{caccg}}{\le}&c[\cacc(\rr)]^{1-1/\gamma}\left(\int_{B_{\rr/4}}\mint_{B_{\rr/4}}\frac{\snr{w(x)-w(y)}^{\gamma}}{\snr{x-y}^{n+s\gamma}}\dx\dy\right)^{1/\gamma}\nonumber \\
&\stackleq{bm111} &c[\cacc(\rr)]^{1-1/\gamma}\nr{w}_{L^{\infty}(B_{\rr/4})}^{1-\vartheta}\rr^{\vartheta-s}\left(\mint_{B_{\rr/4}}\snr{Dw}^{p}\dx\right)^{\vartheta/p}\nonumber \\
&\stackleq{bhh} &c[\cacc(\rr)]^{1-1/\gamma}\nr{u}_{L^{\infty}(B_{\rr/4})}^{1-\vartheta}\rr^{\vartheta-s}\left(\mint_{B_{\rr/4}}\left(\snr{Du}^{p}+\snr{Dh}^{p}\right)\dx\right)^{\vartheta/p}\notag \\
&\stackrel{\eqref{enes}, \eqref{caccg}}{\le} &c\rr^{\vartheta-s}[\cacc(\rr)]^{1-1/\gamma+\vartheta/p}\,,
\label{cl.4}
\end{eqnarray}
with $c\equiv c(\datar)$. Note that in the last line we have also used the content of Proposition \ref{boundprop} in the case $\gamma >p$; again, no appearance of $\nr{w}_{L^{\infty}},\nr{u}_{L^{\infty}}$ takes place when $\gamma \leq p$. For $\mbox{(III)}$ we note that we can replace $u$ by 
$u-(u)_{B_{\rr/2}}$ and use that 
$x\in B_{\rr/4}, y\in \er^n\setminus B_{\rr/2}$ imply $\snr{y-x_{0}}/\snr{x-y}\le 2$. Recalling that $w$ is supported in $B_{\varrho/4}$, we then have
\begin{eqnarray}
\snr{\mbox{(III)}}&\leq&c\int_{\mathbb{R}^{n}\setminus B_{\rr/2}}\mint_{B_{\rr/2}}\frac{\max\{\snr{u(x)-(u)_{B_{\rr/2}}},\snr{u(y)-(u)_{B_{\rr/2}}}\}^{\gamma-1}\snr{w(x)}}{\snr{x-y}^{n+s\gamma}} \dx\dy\nonumber \\
&\leq &c\int_{\mathbb{R}^{n}\setminus B_{\rr/2}}\mint_{B_{\rr/2}}\frac{\max\{\snr{u(x)-(u)_{B_{\rr/2}}},\snr{u(y)-(u)_{B_{\rr/2}}}\}^{\gamma-1}\snr{w(x)}}{\snr{y-x_0}^{n+s\gamma}} \dx\dy\nonumber \\
&\le&c\rr^{-s\gamma}\mint_{B_{\rr/2}}\snr{u-(u)_{B_{\rr/2}}}^{\gamma-1}\snr{w}\dx \nonumber  +c\int_{\mathbb{R}^{n}\setminus B_{\rr/2}}\frac{\snr{u(y)-(u)_{B_{\rr/2}}}^{\gamma-1}}{\snr{y-x_0}^{n+s\gamma}} \dy \mint_{B_{\rr/4}}\snr{w}\dx\nonumber \\
&\stackleq{cacc.2} &c\left[\rr^{-s\gamma}[\av_{\gamma}(\rr/2)]^{\gamma-1}+\rr^{-s}\left(\int_{\mathbb{R}^{n}\setminus B_{\rr/2}}\frac{\snr{u(y)-(u)_{B_{\rr/2}}}^{\gamma}}{\snr{y-x_{0}}^{n+s\gamma}} \dy\right)^{1-1/\gamma}\right]\left(\mint_{B_{\rr/4}}\snr{w}^{\gamma}\dx\right)^{1/\gamma}\nonumber\\
&\leq&c\rr^{-s}\left[\left(\rr^{-s\gamma}[\av_{\gamma}(\rr/2)]^\gamma\right)^{1-1/\gamma}+\left(\rr^{-\delta}[\snail_{\delta}(\rr/2)]^{\gamma}\right)^{1-1/\gamma}\right]\left(\mint_{B_{\rr/4}}\snr{w}^{\gamma}\dx\right)^{1/\gamma}\nonumber\\
&\leq&c\rr^{-s}\left[\left(\rr^{-s\gamma}[\av_{\gamma}(\rr)]^\gamma\right)^{1-1/\gamma}+\left(\rr^{-\delta}[\snail_{\delta}(\rr)]^{\gamma}\right)^{1-1/\gamma}\right]\left(\mint_{B_{\rr/4}}\snr{w}^{\gamma}\dx\right)^{1/\gamma}\nonumber\\
&\stackleq{caccica} &c\rr^{-s}[\caccs(\rr)]^{1-1/\gamma}\left(\mint_{B_{\rr/4}}\snr{w}^{\gamma}\dx\right)^{1/\gamma}\notag \\
&\stackleq{caccg} &c[\cacc(\rr)]^{1-1/\gamma}\left(\rr^{-s\gamma}\mint_{B_{\rr/4}}\snr{w}^{\gamma}\dx\right)^{1/\gamma}, \label{cl.5}
\end{eqnarray}
for $c\equiv c(\datar)$; we have used \eqref{scasnail}-\eqref{averages} in the third-last line. Similarly to \rif{intermedia}, we have 
\begin{flalign}
\notag \rr^{-s\gamma}\mint_{B_{\rr/4}} \snr{w}^{\gamma}\dx
&\leq  c \left(\|u\|_{L^\infty(B_{\rr/4})}+\|h\|_{L^\infty(B_{\rr/4})}\right)^{(1-\vartheta)\gamma} \rr^{(\vartheta-s)\gamma}
\left(\rr^{-p}\mint_{B_{\rr/4}} \snr{w}^{p}\dx\right)^{\vartheta\gamma/p}\\
&\leq  c\|u\|_{L^\infty(B_{\rr})}^{(1-\vartheta)\gamma} \rr^{(\vartheta-s)\gamma}
\left(\mint_{B_{\rr/4}} \snr{Dw}^{p}\dx\right)^{\vartheta\gamma/p} \stackrel{\rif{enes}, \rif{caccg}}{\leq}  c\rr^{(\vartheta-s)\gamma}[ \cacc(\rr)]^{\vartheta\gamma/p}\,.\label{intermedia2}
\end{flalign}
Combining the content of the last displays we conclude with 
$$
 \snr{\mbox{(III)}} \leq  c\rr^{\vartheta-s}[\cacc(\rr)]^{1-1/\gamma+\vartheta/p}\,,
$$
again with $c\equiv c(\datar)$. 
Using this last estimate with \rif{cl.3}-\rif{cl.4} in \rif{daje} we conclude that
\eqn{cl.6}
$$
\mint_{B_{\rr/4}}\mathcal{V}^{2}\dx \le c\nr{f}_{L^{n}(B_{\rr})}[\cacc(\rr)]^{1/p}+c\rr^{\vartheta-s}[\cacc(\rr)]^{1-1/\gamma+\vartheta/p}\,,
$$
holds with $c\equiv c(\datar)$. For the specific dependence of the constant on $\datar$, see also Remark \ref{dipendenza}. To proceed, for the moment we consider the case $p\not= \gamma$, when $[p(\gamma-1)+\vartheta\gamma]/(p\gamma) <1$ and  $[2\gamma-(p-\vartheta\gamma)]/(2\gamma)<1$ are true by \rif{centralass2}; these facts will be used in the cases $p\geq 2$ and $1<p<2$, respectively. Now, if $p\ge 2$, we take $\theta\in (0,1)$ as in \rif{rightdef} and estimate, via PoincarÃ© and Young's inequality
\begin{eqnarray}
\notag \mint_{B_{\rr/4}}\snr{u-h}^{p}\dx&\le&c\rr^{p}\mint_{B_{\rr/4}}\snr{Du-Dh}^{p}\dx\stackrel{\eqref{Vm}}{\le}c\rr^{p}\mint_{B_{\rr/4}}\mathcal{V}^{2}\dx \nonumber \\
\notag&\stackrel{\eqref{cl.6}}{\le}&c\rr^{p-1}\nr{f}_{L^{n}(B_{\rr})}\left(\rr^{p\pm \theta(p-1)/2}\cacc(\rr)\right)^{1/p}\nonumber \\
\notag&& \ +c\rr^{\frac{p-s\gamma}{\gamma}}\left(\rr^{p\pm\frac{\theta(p-\vartheta\gamma)}{2[p(\gamma-1)+\vartheta\gamma]}}\cacc(\rr)\right)^{\frac{p(\gamma-1)+\vartheta\gamma}{p\gamma}}\nonumber \\
\notag&\le&c\left(\rr^{\frac{\theta(p-1)}{2}}+\rr^{\frac{\theta(p-\vartheta\gamma)}{2[p(\gamma-1)+\vartheta\gamma]}}\right)\rr^p\cacc(\rr)+c\rr^{p-\theta/2}\left(\nr{f}_{L^{n}(B_{\rr})}^{p/(p-1)}+\rr^{\frac{p\gamma(\vartheta-s)}{p-\vartheta\gamma}}\right)\nonumber \\
&\stackrel{\eqref{triviala}}{\le}&c\rr^{\theta\sigma} [\GG(\rr)]^{p} ,\label{uh1}
\end{eqnarray}
where $\sigma:=\frac{1}{2}\min\left\{\frac{p-\vartheta\gamma}{p(\gamma-1)+\vartheta\gamma},1\right\}>0$ and $c\equiv c(\datar)$. When $p< 2$, we instead 
estimate
\begin{eqnarray}
\notag \mint_{B_{\rr/4}}\snr{u-h}^{p}\dx &\leq &c\rr^{p}\mint_{B_{\rr/4}}\snr{Du-Dh}^{p}\dx\notag \\ 
\notag&\stackleq{Vm}&c\rr^{p}\left(\mint_{B_{\rr/4}}\mathcal{V}^{2}\dx\right)^{p/2}\left(\mint_{B_{\rr/4}}(\snr{Du}^{2}+\snr{Dh}^{2}+\mu^2)^{p/2}\dx\right)^{1-p/2}\nonumber \\
\notag&\stackrel{\eqref{enes},\eqref{cl.6}}{\le}&c\rr^{p}\left[\nr{f}_{L^{n}(B_{\rr})}[\cacc(\rr)]^{1/p}+\rr^{\vartheta-s}[\cacc(\rr)]^{1-1/\gamma+\vartheta/p}\right]^{p/2}
[\caccs(\rr)]^{1-p/2}\nonumber \\
&\stackrel{\eqref{caccp}, \eqref{caccg}}{\le}&c\rr^{\frac{p(p-1)}{2}}\nr{f}_{L^{n}(B_{\rr})}^{p/2}\left(\rr^{p\pm\frac{\theta(p-1)}{2(3-p)}}\cacc(\rr)\right)^{\frac{3-p}{2}}\nonumber \\
&& \ +c\rr^{\frac{p(p-s\gamma)}{2\gamma}}\left(\rr^{p\pm\frac{\theta(p-\vartheta\gamma)}{2[2\gamma-(p-\vartheta\gamma)]}}\cacc(\rr)\right)^{\frac{2\gamma-(p-\vartheta\gamma)}{2\gamma}}\nonumber \\
&\le&c\left(\rr^{\frac{\theta(p-1)}{2(3-p)}}+\rr^{\frac{\theta(p-\vartheta\gamma)}{2[2\gamma-(p-\vartheta\gamma)]}}\right)\rr^p\cacc(\rr)\nonumber +c\rr^{p-\theta/2}\left(\nr{f}_{L^{n}(B_{\rr})}^{p/(p-1)}+\rr^{\frac{p\gamma(\vartheta-s)}{p-\vartheta\gamma}}\right)\nonumber \\
&\stackrel{\eqref{triviala}}{\le}&c\rr^{\theta\sigma} [\GG(\rr)]^{p},\label{uh2}
\end{eqnarray}
where $\sigma:=\frac{1}{2}\min\left\{\frac{p-1}{3-p},\frac{p-\vartheta\gamma}{2\gamma-(p-\vartheta\gamma)}\right\}>0$ and $c\equiv c(\datar)$. We have so far proved \rif{cl.11} in the case $p\not= \gamma$. When $p=\gamma$ we partially proceed as in \rif{uh1}-\rif{uh2}. When $p\geq 2$, from \rif{cl.6} we directly gain
\eqn{uh3}
$$
\mint_{B_{\rr/4}}\snr{u-h}^{p}\dx \le c\left(\rr^{\theta(p-1)/2}+\rr^{1-s}\right)\rr^p\cacc(\rr)+ c\rr^{p-\theta/2}\nr{f}_{L^{n}(B_{\rr})}^{p/(p-1)}
$$
with $c\equiv c(\datar)$, so that \rif{cl.11} follows via \rif{triviala}, with $\sigma:=(1-s)/2$. If $p<2$, we have
\eqn{uh4}
$$
\mint_{B_{\rr/4}}\snr{u-h}^{p}\dx 
\le c\left(\rr^{\frac{\theta(p-1)}{2(3-p)}}+\rr^{\frac{p(1-s)}{2}}\right)\rr^p\cacc(\rr)+c\rr^{p-\theta/2}\nr{f}_{L^{n}(B_{\rr})}^{p/(p-1)}\,,
$$
where $c\equiv c(\datar)$, so that \rif{cl.11} follows with $ \sigma:=\frac{1}{2}\min\left\{\frac{p-1}{3-p},p(1-s)\right\}$. \end{proof}
\subsection{Step 3: H\"older integral decay and conclusion}\label{intit}
With $t \leq \rr/8$, we bound
\begin{eqnarray}
\av_{p}(t)&\stackleq{basicav}&c\,\left(\mint_{B_{t}}\snr{h-(h)_{B_t}}^p\dx \right)^{1/p}+c\left(\frac{\rr}{t}\right)^{n/p}\left(\mint_{B_{\rr/4}}\snr{u-h}^p\dx \right)^{1/p}\nonumber\\
&\stackrel{\mbox{Poincar\'e}}{\leq}&ct\,\left(\mint_{B_{t}}\snr{Dh}^p\dx \right)^{1/p}+c\left(\frac{\rr}{t}\right)^{n/p}\left(\mint_{B_{\rr/4}}\snr{u-h}^p\dx \right)^{1/p}\nonumber\\
&\stackrel{\eqref{cl.11}}{\le}&ct\nr{Dh}_{L^{\infty}(B_{t})}+c\rr^{\theta\sigma/p}\left(\frac{\rr}{t}\right)^{n/p}\GG(\rr)\nonumber \\
&\stackrel{\eqref{7},\rif{caccg}}{\le}&ct[\cacc(\rr)]^{1/p}+c\rr^{\theta\sigma/p}\left(\frac{\rr}{t}\right)^{n/p}\GG(\rr)\nonumber \\
&\stackrel{\eqref{triviala}}{\le}&c\left[\left(\frac{t}{\rr}\right)+c\rr^{\theta\sigma/p}\left(\frac{\rr}{t}\right)^{n/p}\right]\GG(\rr)\label{8.1}
\end{eqnarray}
with $c\equiv c(\datar)$; the same inequality holds in the case $\rr/8\leq t \leq \rr$ by \rif{averages}. It follows 
\eqn{8}
$$
\begin{cases}
\av_{\gamma}(t) \leq 2\nr{u}_{L^{\infty}(B_{t})}^{1-\vartheta}[\av_{p}(t)]^{\vartheta} \leq c [\av_{p}(t)]^{\vartheta} \\
  \av_{\gamma}(t)  \leq c[(t/\rr)^{\vartheta}+\rr^{\vartheta\theta\sigma/p}(\rr/t)^{n\vartheta/p}][\GG(\rr)]^{\vartheta} \qquad \forall\, t \leq \rr\,,
\end{cases}
$$
for $c\equiv c (\datar)$. Indeed, \rif{8}$_1$ follows as in \rif{intermedia}, while \rif{8}$_2$ follows from \rif{8.1} and \rif{8}$_1$. 
Taking $t\equiv \tau \rr$ in \rif{8.1} with $\tau\in (0,1/8)$, we find, in particular
\eqn{9}
$$\av_{p}( \tau\rr)\le c(\tau+\rr^{\theta\sigma/p}\tau^{-n/p})\, \GG( \rr)$$
with $c\equiv c(\datar)$. In order to get a full decay estimate for $\GG(\cdot)$ from \rif{9}, we need to evaluate the \texttt{snail} and the  \texttt{rhs} terms. For this we use \eqref{scasnail}, that yields 
\begin{flalign}\label{1000}
[\snail_{\delta}(\tau\rr)]^{\gamma} &\leq c\tau^{\delta}[\snail_{\delta}(\rr)]^{\gamma}\nonumber +c(\tau\rr)^{\delta}\left(\int_{\tau\rr}^{\rr}\frac{\av_{\gamma}(\nu)}{\nu^s} \,\frac{\dtau}{\nu}\right)^{\gamma}\nonumber +c\tau^{\delta}\rr^{\delta-s\gamma}[\av_{\gamma}(\rr)]^{\gamma} \\
& =:S_1+S_2+S_3\,.
\end{flalign}
We have $S_1 \leq c \tau^{\delta} [\GG(\rr)]^{p}$ by \rif{eccessi} and \rif{ggdef}. By \rif{8}$_2$ and Young's inequality (recall \rif{centralass2}), we have 
\begin{flalign*}
S_2 & \leq c \tau^{\delta}\rr^{\delta-\vartheta \gamma} \left(\int_{\tau\rr}^{\rr} \,\frac{\dtau}{\nu^{1+s-\vartheta}}\right)^{\gamma}[\GG(\rr)]^{\vartheta\gamma}+ c \tau^{\delta}\rr^{\delta+(\theta\sigma+n)\vartheta\gamma/p}
\left(\int_{\tau\rr}^{\rr} \,\frac{\dtau}{\nu^{1+s+n\vartheta/p}}\right)^{\gamma}[\GG(\rr)]^{\vartheta\gamma}\\
& \leq c \tau^{\delta}\rr^{\delta-s\gamma}\log^{\gamma}\left(\frac 1\tau\right)[\GG(\rr)]^{\vartheta\gamma}
+ c \tau^{\delta-s\gamma -n\vartheta\gamma/p }\rr^{\delta -s\gamma + \theta \sigma \vartheta \gamma/p}
[\GG(\rr)]^{\vartheta\gamma}\\
&\leq c  \left[\tau^{\delta}\log^{p/\vartheta}\left(\frac 1\tau\right)+ 
 \rr^{ \theta \sigma}\tau^{-n-sp/\vartheta}
\right][\GG(\rr)]^{p}+c(\mathds{A}_{\gamma}+\mathds{B}_{\gamma})\tau^{\delta}\ \rr^{\frac{p(\delta - s\gamma)}{p-\vartheta \gamma}}
\,,
\end{flalign*}
where $c\equiv c (\datar)$ and $\mathds{A}_{\gamma},\mathds{B}_{\gamma},\mathds{C}_{\gamma}$ are defined in \rif{t1t2}.   
Using again Young's inequality, we have 
$$
S_3\stackrel{\eqref{8}_1}{\leq}  c\tau^{\delta}\rr^{\delta-s\gamma}[\av_{p}(\rr)]^{\vartheta\gamma}  \leq  c   \tau^{\delta}[\GG(\rr)]^{p}+  c(\mathds{A}_{\gamma}+\mathds{B}_{\gamma}) \tau^{\delta}\rr^{\frac{p(\delta - s\gamma)}{p-\vartheta \gamma}}\,.
$$
Connecting the above inequalities for $S_1, S_2, S_3$, and gathering terms, leads to 
\eqn{9s}
$$[\snail_{\delta}(\tau\rr)]^{\gamma}  \leq c \left[\tau^{\delta}\log^{p/\vartheta}\left(\frac 1\tau\right)+ \rr^{ \theta \sigma}\tau^{-n-sp/\vartheta}\right]
[\GG(\rr)]^{p}+ c (\mathds{A}_{\gamma}+\mathds{B}_{\gamma}) \tau^{\delta} \rr^{\frac{p(\delta - s\gamma)}{p-\vartheta \gamma}}\,.
$$
Noting that 
$[\rhs(\tau\rr)]^{p}\leq \tau^{p-\theta}[\rhs(\rr)]^{p}$, recalling \rif{ggdef}, and 
connecting \rif{9} and \rif{9s}, gives 
\eqn{primadecay0}
$$
\GG(\tau \rr) \leq c \left[\tau^{\delta/p}\log^{1/\vartheta}\left(\frac 1\tau\right)+ \tau^{1-\theta/p}
+ \rr^{ \theta \sigma/p} \tau^{-n/p-s/\vartheta}\right]
\GG(\rr)+ c(\mathds{A}_{\gamma}+\mathds{B}_{\gamma}) \tau^{\delta/p}  \rr^{\frac{\delta - s\gamma}{p-\vartheta \gamma}}
$$  
with $c\equiv c(\datar)$. From now on we consider balls $B_{\rr}\equiv B_{\rr}(x_0)\subset B_{r}(x_0)\equiv B_{r}\Subset \Omega$ with $r \leq r_* \leq 1$; further restrictions on $r_*$ will be put in a few lines. We now fix $\alpha$ such that $0< \alpha <1$ and set $\alpha_1 := (1+\alpha)/2$. We then find
$\theta\equiv \theta (p, \alpha) \in (0,1)$ sufficiently small and then $\delta\equiv \delta(p,s,\gamma,\alpha) \in (s\gamma, p)$ sufficiently close to $p$, such that 
\eqn{condy1}
$$
\alpha_1< 1-\frac{\theta}{p}\,,\qquad  \alpha_1  <  \frac{\delta}{p}\,,\qquad 1-\frac{\theta}{p} \leq \frac{\delta - s\gamma}{p-\vartheta \gamma}$$ 
(this last condition is not required when $p=\gamma$). 
Also note that \rif{condy1} imply
$$
(\mathds{A}_{\gamma}+\mathds{B}_{\gamma}) \rr^{\frac{\delta - s\gamma}{p-\vartheta \gamma}} \leq \rr^{1-\frac{\theta}{p}} \leq  \rhs(\rr)\,.
$$
Using this inequality in \rif{primadecay0}, and recalling the definitions in Section \ref{lalista}, yields
\eqn{primadecay}
$$
\GG(\tau \rr) \leq c_1 \left[\tau^{\delta/p}\log^{1/\vartheta}\left(\frac 1\tau\right)+ \tau^{1-\theta/p}
+ \rr^{ \theta \sigma/p} \tau^{-n/p-s/\vartheta}\right]
\GG(\rr)
$$  
and $c_1\equiv c_1(\datar)$. We eventually determine $\tau\equiv \tau (\datar, \alpha)\leq 1/8$ such that 
\eqn{cadil}
$$
3c_1\tau^{\delta/p-\alpha_1}\log^{1/\vartheta}\left(\frac{1}{\tau}\right) \leq 1\,, \qquad 3c_1\tau^{1-\theta/p-\alpha_1}\leq 1\,\quad  \mbox{and} \quad  \tau^{(1-\alpha)/2}\leq \frac 12\,.$$ 
Once $\tau$ has been determined as a function of the $\datar$ and $\alpha$, we find $r_*\equiv r_*(\datar, \alpha)$ such that if 
$
\rr \leq r \leq r_* $, then 
$
3c_1\rr^{ \theta \sigma/p} \tau^{-n/p-s/\vartheta-\alpha_1}\leq 1.
$ 
With such choices \rif{primadecay} becomes
\eqn{analoga}
$$
\GG(\tau \rr)  \leq \tau^{\alpha_1} \GG(\rr)\,,
$$ that now holds whenever $\rr \leq r\leq r_{*}$. We now introduce the sharp fractional maximal type operator
\eqn{analoganc}
$$
 \texttt{M}\,(x_0,\rr):= \sup_{\nu \leq \rr}\,  \nu^{-\alpha} \GG(u, B_\nu(x_0))
$$
and its truncated version 
$$
 \texttt{M}_{\eps}\,(x_0,\rr):= \sup_{\eps r\leq \nu \leq \rr}\,  \nu^{-\alpha} \GG(u, B_\nu(x_0))\,,  \qquad 0<\eps <\tau\,.
$$
Multiplying both sides of \rif{analoga} by $(\tau\rr)^{-\alpha}$, taking the sup with respect to $\rr \in (\eps r , r)$, we arrive at
\begin{eqnarray*}
 \texttt{M}_{\eps}\,(x_0,\tau r)  &\leq & \sup_{\eps \tau r\leq \nu \leq \tau r}\,  \nu^{-\alpha} \GG(\rr)\\ &
  \leq & \tau^{(1-\alpha)/2}  \texttt{M}_{\eps}\,(x_0,r)  \\ &\stackleq{cadil} & \frac 12  \texttt{M}_{\eps}\,(x_0, \tau r)  +  \sup_{\tau r \leq \nu \leq r} \nu^{-\alpha}\GG(u, B_{\nu}(x_0))\,,
\end{eqnarray*}
 that in turn implies, reabsorbing terms (note that $ \texttt{M}_{\eps}$ is always finite), and recalling that $\tau \equiv \tau (\datar, \alpha)$
 $$
  \texttt{M}_{\eps}\,(x_0,r)  \leq \frac{c}{r^\alpha}  \sup_{\tau r \leq \nu \leq r} \GG(u, B_\nu(x_0)) \,.
$$
Letting $\eps \to 0$ yields
\eqn{maximalestimate}
$$
  \texttt{M}\,(x_0,r)  \leq \frac{c}{r^\alpha}  \sup_{\tau r \leq \nu \leq r} \GG(u, B_\nu(x_0)) \,,
$$
with again $c \equiv c (\datar, \alpha)$. 
In order to estimate the right hand side we use \rif{scasnail}-\rif{averages}, that yields
\begin{flalign}
\nonumber   \texttt{M}\,(x_0,r) & \leq \frac{c}{r^\alpha}  \left( \av_{p}(r) +r^{(\delta-s\gamma)/p} [\snail_{s\gamma}(r)]^{\gamma/p}+ r^{(\delta-s\gamma)/p}[\av_{\gamma}(r)]^{\gamma/p}+ \rhs(r)\right)\\
 & \leq \frac{c}{r^\alpha}  \left( \av_{p}(r) + [\snail_{s\gamma}(r)]^{\gamma/p}+r^{\alpha_1}\|f\|_{L^n(B_r)}^{1/(p-1)}+r^{\alpha_1}\right)\,, \label{senza}
\end{flalign}
where $c \equiv c (\datar, \alpha)$. From \rif{senza}, recalling the definition in \rif{analoganc}, estimate \rif{campanato} and Theorem \ref{t5} follow via elementary manipulations; see Remark \ref{dipendenza} below.  
Moreover, estimating $\av_{p}(r_*) +\snail_{s\gamma}(r_*)  \leq c (\datah)$, we have proved the following:
\begin{proposition}\label{outcome}
Under assumptions \eqref{assf}-\eqref{assk} and \eqref{g33}, let $u\in \mathbb{X}_{g}(\Omega)$ be as in \eqref{fun}. For every $\alpha \in (0,1)$ there exist $r_{*}\equiv r_*(\datah, \alpha) \in (0,1)$ and $c\equiv c(\datah,\alpha)\geq 1$, such that 
\eqn{stimetta}
$$ 
\mint_{B_{\rr}}\snr{u-(u)_{B_{\rr}}}^{p}\dx \leq c \left(\frac{\rr}{r_*}\right)^{\alpha p}
$$
holds whenever $B_{\rr} \Subset \Omega$ and $\rr \leq r_*$. 
\end{proposition}
Theorem \ref{t2} now follows from \rif{stimetta} and the classical Campanato-Meyers integral characterization of H\"older continuity (via a standard covering argument); see for instance \cite{giu} and Remark \ref{devore-re}. 
\begin{remark}\label{dipendenza}{\em The constant $c$ and the radius $r_*$ appearing in \rif{campanato}, depend on $n,p,s,\gamma, \Lambda$ and $\|u\|_{L^{\infty}}^{1-\vartheta}$, see the definition in \rif{datta}. In turn, in the case $\gamma >p$, via Proposition \ref{boundprop}, $\|u\|_{L^{\infty}}$ can be bounded via a constant depending on $\datab$ and this justifies the final dependence of $c, r_*$ described in the statement of Theorem \ref{t5}. Accordingly, by \rif{t1t2}, when $\gamma \leq p$ no dependence on $\|u\|_{L^{\infty}}$ occurs in the estimates as $\vartheta=1$ and this explains the peculiar definition of $\datah$ in \rif{idati}$_{2.3}$. In fact, when $\gamma \leq p$, we are directly proving H\"older estimates on $u$ without using any bound on $\|u\|_{L^{\infty}}$ and this justifies the claim in Theorem \ref{t2} that we can avoid using assumption \rif{g33}.}
\end{remark}
\begin{remark}\label{devore-re} \emph{When neglecting the presence of the $\snail_{\delta}$ and $\rhs$ in the definition of $\GG$ in \rif{ggdef}, that is, when considering the purely local, homogenous setting, we have that 
\rif{analoganc} turns into
$$
 \texttt{M}\,(x_0,\rr)= \sup_{\nu \leq \rr}\,  \nu^{-\alpha} \left(\mint_{B_{\nu}(x_0)}|u-(u)_{B_{\nu}(x_0)}|^p\dx \right)^{1/p}\,.
$$
This is nothing but the classical local and fractional variant of the Feffermain-Stein Sharp Maximal Operator widely used in \cite{devore}. Moreover, note that a bound of the type in \rif{senza} immediately implies the local H\"older continuity of $u$ as 
$$
|u(x)-u(y)| \leq \frac{c}{\alpha}\left[ \texttt{M}\,(x,\rr)+ \texttt{M}\,(y,\rr)\right]\snr{x-y}^{\alpha}
$$
holds whenever $x, y \in B_{\rr/4}$, for every ball $B_{\rr}\subset \er^n$ (see \cite{devore}).}
\end{remark}
\section{Proof of Theorem \ref{t3}}\label{prova3}
In this section we permanently work under the assumptions of Theorem \ref{t3}, that is \eqref{assf}-\eqref{assk} and \eqref{g3}. The proof goes in seven different steps. 
\subsection{Step 1: Flattening of the boundary and global diffeomorphisms}\label{flbo} The classical flattening-of-the-boundary procedure needs to be upgraded here, as we are in a nonlocal setting. We first recall the standard local procedure, as for instance described in \cite{beck, beck2, krime, kronz}, and summarize its main points. Let us consider $x_0 \in \partial\Omega$; without loss of generality (by translation) we can assume that $x_0\in \{x_n=0\}$ and that $\Omega$ touches $\{x_n=0\}$ tangentially, so that its normal at $x_0$ is $e_n$, where $\{e_i\}_{i\leq n}$ is the standard basis of $\er^n$. By the assumption $\partial \Omega \in C^{1, \alpha_b}$, there exists a radius $r_0\equiv r_{x_0}$, depending on $x_0$, and a $C^{1, \alpha_b}$-regular diffeomorphism $\TT\equiv \TT_{x_0}\colon B_{4r_{0}}(x_0)  \mapsto \er^n$ such that $\TT(x_0)=x_0$, $B_{2r_{0}}^+(x_0) \subset \TT(\Omega_{3r_{0}}(x_0)) \subset B_{4r_{0}}^+(x_0) $, $\Gamma_{2\rr}(x_0)\subset \TT(\partial \Omega \cap B_{2r_0}(x_0)) \subset \Gamma_{3\rr}(x_0)$ and $|z|/c_* \leq |D\TT(x)z|\leq c^*|z|$, $x \in B_{4r_0}(x_0)$, where $c_*\in (1, 10/9)$ can be chosen close to $1$ at will taking a smaller $r_0$. Moreover, it is 
\begin{flalign}\label{diff0}
\begin{cases}
 \|\TT\|_{C^{1,\alpha_b}(B_{4r_{0}}(x_0))}+ \|\mathcal{T}^{-1}\|_{C^{1,\alpha_b}(B_{4r_{0}}(x_0))}<\infty\\
\nr{\mathcal{J}_{\ti{\mathcal{T}}}}_{L^{\infty}(B_{4r_{0}}(x_0))}+\nr{\mathcal{J}_{\ti{\mathcal{T}}^{-1}}}_{L^{\infty}(B_{4r_{0}}(x_0))}<\infty\,, \end{cases}
\end{flalign}
where $\mathcal{J}_{\mathcal{T}}$ and $\mathcal{J}_{\mathcal{T}^{-1}}$ denote the Jacobian determinants of 
$\mathcal{T}$ and $\mathcal{T}^{-1}$, respectively. 
We refer for instance to \cite[Section 3.2]{beck} and \cite[pages 306 and 318]{beck2} for the full details and for the explicit expression of the map $\TT$ considered here; see also \cite{krime, kronz}. We next extend $\TT$ to a $C^{1}$-regular global diffeomorphism of $\mathbb{R}^{n}$ into itself, following a discussion we found in math stackexchange \footnote{\url{https://math.stackexchange.com/questions/148808/the-extension-of-diffeomorphism}}. With $\eta\in C^{\infty}_{0}(B_{4r_{0}}(x_0))$ being such that
$
\mathds{1}_{B_{3r_{0}}}\le \eta\le \mathds{1}_{B_{4r_{0}}}$ and $ \snr{D\eta}\lesssim 1/r_{0}
$, we define 
\eqn{estendi}
$$
\begin{cases} 
\ \TT_{x_0}(x):= \TT(x_0)+D\TT(x_0)\cdot (x-x_0)\\
\ \ti{\TT}_{x_0}(x):=(1-\eta(x))\TT_{x_0}(x)  +\eta(x)\TT(x)\,.
\end{cases}
$$ It follows that $\ti{\TT}_{x_0}$ is $ C^{1,\alpha_b}$-regular and, being $D\TT(x_0)$ invertible, that $
\TT_{x_0}$ is a smooth global diffeomorphism of $\mathbb{R}^{n}$. We now use that 
the set of $C^1$-diffeomorphisms of $\er^n$ (into itself) is open in the (strong) topology of $C^{1}(\er^n,\er^n)$ (see \cite[Chapter 2, Theorem 1.6]{hirsch}, also for the relevant definitions). For this, we take $\texttt{r}_{x_0}>0$, such that if $\mathcal H \in C^{1}(\er^n,\er^n)$ and $\|\mathcal H-\TT_{x_0}\|_{C^1(\er^n,\er^n)}< \texttt{r}_{x_0}$, then $\mathcal H$ is a global $C^1$-regular diffeomorphism. By using \rif{diff} and mean value theorem, it now easily follows that 
\eqn{ildiffeo}
$$\|\ti{\TT}_{x_0}-\TT_{x_0}\|_{C^1(\er^n)}\leq c \|\TT\|_{C^{1,\alpha_b}(B_{4r_{0}}(x_0))} r_0^{\alpha_b}\equiv cr_0^{\alpha_b}\,, $$ with $c$ depending again on $x_0$, so that, by taking $r_0$ such that $ c r_0^{\alpha_b}< \texttt{r}_{x_0}$, we obtain that $\ti{\TT}_{x_0}$ (from now on also denoted by $\TT$) is a $C^1$-regular global diffeomorphism. Summarizing, and recalling the explicit expression of $\ti{\TT}_{x_0}$ in \rif{estendi}, we have that for every $x_0\in \partial \Omega$, there exists a global $C^1$-regular diffeomorphism $\TT \equiv \ti{\TT}_{x_0}$ such that 
\begin{flalign}\label{diff}
\begin{cases}
 \ \|D\ti{\mathcal{T}}\|_{L^{\infty}(\mathbb{R}^{n})},\|D\ti{\mathcal{T}}^{-1}\|_{L^{\infty}(\mathbb{R}^{n})}\le c_{0}<\infty\\
\ \nr{\mathcal{J}_{\ti{\mathcal{T}}}}_{L^{\infty}(\mathbb{R}^{n})},\nr{\mathcal{J}_{\ti{\mathcal{T}}^{-1}}}_{L^{\infty}(\mathbb{R}^{n})}\le c_{0}<\infty 
\end{cases}
\end{flalign}
 (here we are further enlarging $c_0$) and which is $C^{1, \alpha_b}$-regular diffeomorphism on $B_{2r_0}$. A comment needs perhaps to be made here, on the inequalities in \rif{diff}. Since $\ti{\TT}_{x_0}$ is a $C^1$-regular diffeomorphism, then \rif{diff} holds when replacing $\er^n$ by $B_{4r_0}(x_0)$ by compactness, for a suitable constant $c_0$; on the other hand $\ti{\TT}_{x_0}$ is affine on $\er^n\setminus B_{4r_0}(x_0)$ and it is $D\ti{\TT}_{x_0}=D\TT(x_0)$, which is invertible as $\TT$ is a local diffeomorphism in $B_{2r_0}$. Therefore \rif{diff} holds as stated, by eventually enlarging $c_0$. Note that, at this stage, the constant $c_0$ appearing in \rif{diff} is still depending on the point $x_0$ via the diffeomorphism $\TT$. As we are going to flatten the entire boundary $\partial \Omega$ with maps as $\TT$, by compactness we can assume that $r_0$ and $c_0$ are independent of $x_0 \in \partial \Omega$; see also Remark \ref{appiatti} below for more on this aspect. 
\subsection{Step 2: The flattened functional around a point $x_0\in \partial \Omega$}\label{listabordo} We set $\ti{\Omega}:= \TT(\Omega)$, so that $\Omega:= \TT^{-1}(\ti{\Omega})$, and also set $\tilde{g}:=g\circ\TT ^{-1}$. Note that if $w \in \mathbb{X}_{g}(\Omega)$, then $\ti{w}:=w\circ\TT ^{-1} \in \mathbb{X}_{\tilde{g}}(\ti{\Omega})$; on the other hand, any $\ti{w} \in \mathbb{X}_{\tilde{g}}(\ti{\Omega})$ can be written as $\ti{w}=w\circ\TT ^{-1}$ where $w \in \mathbb{X}_{g}(\Omega)$ is simply defined by $w:=\ti{w}\circ \TT$. By \rif{g3} and \rif{diff} it follows
\begin{eqnarray}\label{ttgg}
\begin{cases}
\ti{g}\in W^{1,q}(\ti{\Omega})\cap W^{s,\gamma}(\mathbb{R}^{n})\cap W^{a,\chi}(\mathbb{R}^{n})\\
\nr{\ti{g}}_{W^{1,q}(\ti{\Omega})}+\nr{\ti{g}}_{W^{s,\gamma}(\er^n)}+\nr{\ti{g}}_{W^{a,\chi}(\er^n)}\leq c (\data)\,.
\end{cases}
\end{eqnarray}
We then define the (flattened) functional 
$$
\mathbb{X}_{\tilde g}(\ti{\Omega}) \ni \ti{w} \mapsto  \tilde{\mathcal{F}}(\ti{w}):=\int_{\ti{\Omega}}\ccc(x)[\tilde{F}(x,D\ti{w})-\ti{f}\ti{w}]\dx+\int_{\mathbb{R}^{n}}\int_{\mathbb{R}^{n}}\Phi(\ti{w}(x)-\ti{w}(y))\tilde{K}(x,y)\dx\dy
$$
where 
$$
\begin{cases}
\tilde{F}(x,z):=F(zD\TT(\TT^{-1}(x))),\quad \ccc(x):=\snr{\mathcal{J}_{\mathcal{T}^{-1}}(x)},\\
\ti{f}(x):=f(\TT^{-1}(x)),\quad \tilde{K}(x,y):=\ccc(x)\ccc(y)K(\TT^{-1}(x),\TT^{-1}(y))\,.
\end{cases}
$$
Defining $\ti{u}:=u\circ\TT ^{-1} $, by \rif{fun} we have 
 \eqn{funb}
 $$
\mathbb{X}_{\tilde{g}}(\ti{\Omega})\ni \ti{u} \mapsto\min_{\ti{w} \in \mathbb{X}_{\tilde{g}}(\ti{\Omega})} \tilde{\mathcal{F}}(\ti{w})\,.
$$
By the very definition of $\ti{u}$, Proposition \ref{boundprop}, and directly from \rif{ttgg}, we also find 
\eqn{boundprop2}
$$
\|\ti{u}\|_{L^{\infty}(\er^n)} +\|\ti{g}\|_{L^{\infty}(\er^n)}+\|\ti{f}\|_{L^{n}(\ti{\er^n})}\leq c (\data)\,.
$$
From now on, any dependence of the various constants from $\TT$, that is $\nr{\mathcal{T}}_{C^{1, \alpha_b}(B_{r_0}(x_0))}$, $\nr{\mathcal{T}}_{C^{1}(\er^n)}$ and the like, will be incorporated in the dependence on $\Omega$, and therefore on $\data$ (compare with \rif{idati}$_4$). It follows from the very definitions given, \eqref{assk} and \eqref{diff} that $\ccc(\cdot)$ is continuous and 
\eqn{assc}
$$
\begin{cases} 
\ \snr{\ccc(x)-\ccc(y)}\le \tilde{\Lambda}\snr{x-y}^{\alpha_b}\,, \qquad \forall \ x,y\in B_{r_0}(x_0)\\
\ 0<\tilde{\Lambda}^{-1}\le \ccc(x)\le \tilde{\Lambda}\,, \qquad \forall \  x \in \er^n\\
 \ti{\Lambda}^{-1}\snr{x-y}^{-n-s\gamma} \leq \tilde{K}(x,y)\leq \ti{\Lambda} \snr{x-y}^{-n-s\gamma}\,,\quad \forall \ x,y\in \er^n, \ x\not = y\,.
\end{cases}
$$
Again by \eqref{assf} and \eqref{diff}, as for the new integrand $\tilde{F}(\cdot)$, we have
\begin{flalign}\label{assft}
\begin{cases}
\, z\mapsto \tilde{F}(x,z)\in C^{2}(\mathbb{R}^{n}\setminus \{0\})\cap C^{1}(\mathbb{R}^{n})\\
\, \tilde{\Lambda}^{-1}(\snr{z}^{2}+\mu^{2})^{p/2}\le \tilde{F}(x,z)\le \tilde{\Lambda}(\snr{z}^{2}+\mu^{2})^{p/2}\\
\ (\snr{z}^{2}+\mu^{2})\snr{\partial_{zz}\tilde{F}(x,z)}+(\snr{z}^{2}+\mu^{2})^{1/2}\snr{\partial_{z}\tilde{F}(x,z)}\le \tilde{\Lambda}(\snr{z}^{2}+\mu^{2})^{p/2}\\
\,  \tilde{\Lambda}^{-1}(\snr{z}^{2}+\mu^{2})^{(p-2)/2}\snr{\xi}^{2}\leq \partial_{zz}\tilde{F}(x,z)\xi\cdot \xi \\
\, \snr{\partial_{z}\tilde{F}(x,z)-\partial_{z}\ti{F}(y,z)}\le \tilde{\Lambda}\snr{x-y}^{\alpha_b}(\snr{z}^{2}+\mu^{2})^{(p-1)/2},
\end{cases}
\end{flalign}
for all $\xi \in \mathbb{R}^{n}$, $z\in \mathbb{R}^{n}\setminus \{0\}$, $x,y\in B_{r_{0}}(x_0)$. In \eqref{assc} and \eqref{assft} it is $\ti{\Lambda} \equiv \ti{\Lambda}(\data)\geq 1$. The Euler-Lagrange equation corresponding to \rif{funb} is now
\begin{flalign}\label{elb0}
 &\int_{\ti{\Omega}}\ccc(x)\left[\partial_{z}\ti{F}(x,D\tilde{u})\cdot D\ti{\varphi}-\ti{f}\ti{\varphi}\right]\dx\nonumber \\
 & \qquad +\int_{\mathbb{R}^{n}}\int_{\mathbb{R}^{n}}\Phi'(\ti{u}(x)-\ti{u}(y))(\ti{\varphi}(x)-\ti{\varphi}(y))\ti{K}(x,y)\dx\dy =0\,,
\end{flalign}
and holds for all $\ti{\varphi}\in \mathbb{X}_{0}(\ti{\Omega})$. Performing the same transformation described in Section \ref{rewrite},  
we can use 
\begin{flalign}\label{elb}
 &\int_{\ti{\Omega}}\ccc(x)\left[\partial_{z}\ti{F}(x,D\tilde{u})\cdot D\ti{\varphi}-\ti{f}\ti{\varphi}\right]\dx\nonumber \\
 & \qquad +\int_{\mathbb{R}^{n}}\int_{\mathbb{R}^{n}}|\ti{u}(x)-\ti{u}(y)|^{\gamma-2}(\ti{u}(x)-\ti{u}(y))(\ti{\varphi}(x)-\ti{\varphi}(y))\ti{K}_{\texttt{s}}(x,y)\dx\dy=0\,,
\end{flalign}
with the new kernel $\Kst(\cdot)$ that can be obtained by $\ti{K}(\cdot)$ as explained in \rif{cacc.0} and satisfies 
\eqn{asskt}
$$
\ti{K}_{\texttt{s}}(x,y)=\ti{K}_{\texttt{s}}(y,x)\quad \mbox{and}\quad \Kst(x,y)\approx_{\ti{\Lambda}} \frac{1}{\snr{x-y}^{n+s\gamma}}
$$
for every $x,y\in\er^n$, $x\not = y$.
\begin{remark}\label{piattitutto}\emph{
The various constants generically appealed to as $\tilde{\Lambda}$, $c_0$ and $c \equiv c (\data)$ from Sections \ref{flbo} and \ref{listabordo}, actually depend on the point $x_0$ via the features of the map $\TT$ considered; this dependence has been omitted above, and we will continue to do so. Indeed, by a standard compactness argument, we can cover and flatten the whole boundary $\partial \Omega$ by using a finite number of such diffeomorphisms $\{\TT_i\}_{i\leq k}$ (and points $\{x_i\}_{\leq k}$), generating the corresponding constants in the estimates. Eventually, we take the largest constants/lowest and make all the resulting constants independent of the specific point $x_i$ considered. We note that all such dependences will be incorporated in $\data$, since this last one also depends on $\Omega$. Similarly, we can assume that the size of the radius $r_0$, that can be decreased at will, is independent of the point $x_0$; we remark that such reasoning is standard \cite{beck, beck2, krime, kronz}.}
\end{remark}\label{appiatti}
 \subsection{Step 3: Localized regularity} In order to prove Theorem \ref{t3} it is now sufficient to show that $u \in C^{0, \alpha}(\Omega)$ holds for every $\alpha < \kappa$, with $[u]_{0, \alpha;\Omega} \leq c(\data, \alpha)$, and where $\kappa$ is defined in \eqref{g3}$_3$. This follows from the fact that $u \in g+W^{1,p}_0(\Omega)$ and $g \in W^{a, \chi}(\er^n)$, and therefore $g \in C^{0, a-n/\chi}(\er^n)$, as $W^{a, \chi}(\er^n) \subset  C^{0, a-n/\chi}(\er^n)$ with $
 \|g\|_{C^{0, \kappa}(\er^n)} \leq c \|g\|_{W^{a, \chi}(\er^n)}
 $. This is implied by  \rif{g3}$_3$ and \cite[Theorem 8.2]{guide}. The last two estimates also give $[u]_{0, \alpha;\er^n} \leq c(\data, \alpha)$ as claimed in Theorem \ref{t3}. Finally, to get that $u \in C^{0, \alpha}(\Omega)$ for every $\alpha <1$ when $g \in W^{1,\infty}(\er^n)$, it is then sufficient to note that a careful reading of the (forthcoming) proof of Theorem \ref{t3} reveals that Theorem \ref{t3} continuous to hold when replacing the assumption $g \in W^{a, \chi}(\er^n)$ by $g \in W^{a, \chi}_{\loc}(\er^n)$ and $g \in C^{0, \kappa}(\er^n)$ (or even by taking $g \in W^{a, \chi}(\Omega')$ with $\Omega \Subset \Omega'$). If $g \in W^{1,\infty}(\er^n)$, then these new conditions are obviously satisfied. Also taking Remark \ref{piattitutto} into account, via a standard covering argument, we are left to prove the following fact, from which Theorem \ref{t3} follows:
\begin{proposition} \label{campi}
Let $\ti{u}\in \mathbb{X}_{\ti{g}}(\ti{\Omega})$ be the solution to \eqref{funb}. Then $\ti{u}\in C^{0,\alpha}(\bar{B}_{r_{0}/2}^{+}(x_0))$ for every $\alpha < \kappa$. Moreover, there exists a constant $c \equiv c (\data, \alpha)$ such that $[\ti{u}]_{0, \alpha; \bar{B}_{r_{0}/2}^{+}(x_0)}\leq c$. 
\end{proposition}
For the proof of Proposition \ref{campi}, from now on we shall consider points $\tilde x_0 \in \Gamma_{r_{0}/2}(x_0)$, radii
$\rr \leq r_{0}/4 \leq 1/4$, and upper balls $B_{\rr}\equiv B_{\rr}^+(\ti{x}_0)\subset B_{r_0}^+(x_0)$. Unless otherwise stated, all the upper balls will be centred at $\ti{x}_0$, and $\ti{x}_0$ will be a fixed, but generic point as just specified. In analogy to the interior case, with $\delta$ being such that 
$s\gamma < \delta < p$ (such a choice is allowed by \rif{centralass}), 
we define the boundary analog of the quantities introduced in Section \ref{lalista} as follows:
\eqn{eccessig}
$$
\displaystyle \ecc^{+}(\rr)\equiv \ecc^{+}(\ti{u},B_{\rr}(\ti{x}_{0})):=\left(\mint_{B_{\rr}^{+}(\ti{x}_{0})}\snr{\ti{u}-\ti{g}}^p\dx \right)^{1/p}+\left[\snail_{\delta}(\ti{u},B_{\rr}(\ti{x}_{0}))\right]^{\gamma/p}\,,
$$
\begin{flalign}
[\rhsp(\rr)]^{p}\equiv [\rhsp(B_{\varrho}(\tilde x_0))]^{p}&:=\rr^{p-\theta}\left(\nr{\ti{f}}_{L^{n}(B_{\rr}^{+}(\ti{x}_{0}))}^{p/(p-1)}+1\right)+\left(\rr^{q}\mint_{B_{\rr}^{+}(\ti{x}_{0})}\snr{D\tilde{g}}^{q}\dx\right)^{p/q}\nonumber\\
&\qquad  \ +\left(\rr^{a\chi}\int_{B_{\rr}(\ti{x}_{0})}\mint_{B_{\rr}(\ti{x}_{0})}\frac{\snr{\ti{g}(x)-\ti{g}(y)}^{\chi}}{\snr{x-y}^{n+a\chi}}\dx\dy\right)^{p/(\vartheta\chi)},\label{rightb}
\end{flalign}
where $\vartheta$ has been defined in \rif{t1t2},
\begin{flalign}
\notag \cacc^{+}(\rr)\equiv \cacc^{+}(\ti{u},B_{\varrho}(\tilde x_0))&:=\rr^{-p}\mint_{B_{\rr}^{+}(\ti{x}_{0})}\snr{\ti{u}-\ti{g}}^{p}\dx +\rr^{-\delta}[\snail_{\delta}(\ti{u},B_{\rr}(\ti{x}_{0}))]^{\gamma}\nonumber \\
\notag & \qquad +\left(\nr{\ti{f}}_{L^{n}(B_{\rr}^{+}(\ti{x}_{0}))}^{p/(p-1)}+1\right)+\left(\mint_{B_{\rr}^{+}(\ti{x}_{0})}\snr{D\ti{g}}^{q}\dx\right)^{p/q}\nonumber \\
& \qquad  +\left(\rr^{\chi(a-s)}\int_{B_{\rr}(\ti{x}_{0})}\mint_{B_{\rr}(\ti{x}_{0})}\frac{\snr{\ti{g}(x)-\ti{g}(y)}^{\chi}}{\snr{x-y}^{n+a\chi}}\dx\dy\right)^{\gamma/\chi}\,,\label{cccb}
\end{flalign}
and, finally
\begin{flalign}
\notag [\GGp(\rr)]^{p} &\equiv [\GGp(\ti{u},B_{\varrho}(\tilde x_0))]^{p}\\
& :=[\ecc^{+}(\ti{u},B_{\varrho}(\tilde x_0))]^{p}+[\rhsp(B_{\varrho}(\tilde x_0))]^{p}+(\mathds{A}_{\gamma}+\mathds{B}_{\gamma}) \rr^{\frac{p(\delta - s\gamma)}{p-\vartheta \gamma}}\,,\label{ilglobale}
\end{flalign}
where $\mathds{A}_{\gamma},\mathds{B}_{\gamma}$ are defined in \rif{t1t2}. Thanks to \rif{centralass2}, by Young's inequality, $\delta <p$ and $\rr \leq 1$, we find
\begin{flalign*}
\rr^p\left(\rr^{\chi(a-s)}\int_{B_{\rr}}\mint_{B_{\rr}}\frac{\snr{\ti{g}(x)-\ti{g}(y)}^{\chi}}{\snr{x-y}^{n+a\chi}}\dx\dy\right)^{\frac{\gamma}{\chi}}
& \leq \left(\rr^{a\chi}\int_{B_{\rr}}\mint_{B_{\rr}}\frac{\snr{\ti{g}(x)-\ti{g}(y)}^{\chi}}{\snr{x-y}^{n+a\chi}}\dx\dy\right)^{\frac{p}{\vartheta\chi}} \\
 & \ \ \quad  + (\mathds{A}_{\gamma}+\mathds{B}_{\gamma}) \rr^{\frac{p(\delta - s\gamma)}{p-\vartheta \gamma}}\,.
\end{flalign*} 
The above definitions, and the content of the last display, yield
\eqn{late}
$$
\rr^{p}\cacc^{+}(\rr)\leq c[\GGp(\rr)]^{p}
$$
with $c\equiv c(s,\gamma,p)$. 
We shall often use the inequality 
\eqn{holtri}
$$\int_{B}\mint_{B}\frac{\snr{\ti{g}(x)-\ti{g}(y)}^{\gamma}}{\snr{x-y}^{n+s\gamma}}\dx\dy \leq c\left(|B|^{\chi(a-s)/n}\int_{B}\mint_{B}\frac{\snr{\ti{g}(x)-\ti{g}(y)}^{\chi}}{\snr{x-y}^{n+a\chi}}\dx\dy\right)^{\gamma/\chi}\,,
$$
that follows by a simple application of H\"older's inequality. 
\subsection{Step 4: Boundary Caccioppoli type inequality} 
We begin the proof of Proposition \ref{campi} with 
\begin{lemma}\label{caccb}
The inequality
\eqn{cacc+}
$$
\mint_{B_{\rr/2}^{+}(\ti{x}_{0})}(\snr{D\ti{u}}^{2}+\mu^{2})^{p/2}\dx +\int_{B_{\rr/2}(\ti{x}_{0})}\mint_{B_{\rr/2}(\ti{x}_{0})}\frac{\snr{\ti{u}(x)-\ti{u}(y)}^{\gamma}}{\snr{x-y}^{n+s\gamma}}\dx\dy \leq c\, \cacc^{+}(u,B_{\varrho}(\tilde x_0))
$$
holds with $c\equiv c(\data)$.
\end{lemma}
\begin{proof}
Fix parameters $\rr/2  \leq  \tau_{1}<\tau_{2}  \le \rr$, a function $\eta\in C^{1}_{0}(B_{\tau_{2}})$ such that $
\mathds{1}_{B_{\tau_{1}}}\le \eta\le \mathds{1}_{B_{(3\tau_{2}+\tau_{1})/4}}$ and $\snr{D\eta}\lesssim 1/(\tau_{2}-\tau_{1})$.
With $m:=\max\{\gamma,p\}$, set $\umm:=\tilde{u}-(\ti{u})_{B_{\tau_{2}}}$, $\gmm:=\tilde{g}-(\ti{u})_{B_{\tau_{2}}}$, $\wmm:=\umm-\gmm = \ti{u}-\ti{g}$ and consider $\ti{\varphi}:=\eta^{m}\wmm$. By its very definition, $\ti{\varphi}$ vanishes outside $B_{\tau_{2}}^{+}\subset B_{\rr}^{+}\subset B_{r_0}^{+}(x_0)$, so that \rif{ttgg} implies $\varphi\in \mathbb{X}_{0}(B_{\rr}^{+})$. Testing \eqref{elb} with $\ti{\varphi}$ we find 
\begin{flalign*}
0&=\int_{B_{\rr}^{+}}\eta^{m}\ccc(x)\left[\partial_{z}\ti{F}(x,D\ti{u})\cdot D\wmm-\ti{f}\wmm\right]\dx \nonumber  +m\int_{B_{\rr}^{+}}\eta^{m-1}\wmm\ccc(x)\partial_{z}\ti{F}(x,D\ti{u})\cdot D\eta\dx\nonumber \\
& \quad +\int_{B_{\tau_{2}}}\int_{B_{\tau_{2}}}|\ti{u}(x)-\ti{u}(y)|^{\gamma-2}(\ti{u}(x)-\ti{u}(y))(\eta^{m}(x)\wmm(x)-\eta^{m}(y)\wmm(y))\Kst(x,y)\dx\dy\nonumber \\
& \quad  +2\int_{\mathbb{R}^{n}\setminus B_{\tau_{2}}}\int_{B_{\tau_{2}}}|\ti{u}(x)-\ti{u}(y)|^{\gamma-2}(\ti{u}(x)-\ti{u}(y))\eta^{m}(x)\wmm(x)\Kst(x,y)\dx\dy\nonumber \\
&=:\mbox{(I)}+\mbox{(II)}+\mbox{(III)}+\mbox{(IV)}.
\end{flalign*}
Via \rif{usa}, $\eqref{assc}$, $\eqref{assft}$, and Sobolev, Poincar\'e and Young's inequalities (as in Lemma \ref{cacclem}) we obtain
\begin{flalign*}
\mbox{(I)}+\mbox{(II)}&\ge\frac 1c_{*}\int_{B_{\rr}^{+}}\eta^{m}(\snr{D\ti{u}}^{2}+\mu^{2})^{p/2}\dx-c\snr{B_{\rr}}\left(\mint_{B_{\tau_{2}}^{+}}\snr{D\ti{g}}^{q}\dx\right)^{p/q}\nonumber \\
&\quad -\frac{c}{(\tau_{2}-\tau_{1})^{p}}\int_{B_{\rr}^{+}}\snr{\ti{u}-\ti{g}}^{p}\dx-c\snr{B_{\rr}}\left(\nr{\ti{f}}_{L^{n}(B^{+}_{\rr})}^{p/(p-1)}+1\right) ,
\end{flalign*}
where $c\equiv c(\data)$. We then write $\mbox{(III)}$ as
\begin{flalign*}
\mbox{(III)}&=\int_{B_{\tau_{2}}}\int_{B_{\tau_{2}}}|\umm(x)-\umm(y)|^{\gamma-2}(\umm(x)-\umm(y))(\eta^{m}(x)\umm(x)-\eta^{m}(y)\umm(y))\Kst(x,y)\dx\dy\nonumber \\
&\quad -\int_{B_{\tau_{2}}}\int_{B_{\tau_{2}}}|\umm(x)-\umm(y)|^{\gamma-2}(\umm(x)-\umm(y))(\eta^{m}(x)\gmm(x)-\eta^{m}(y)\gmm(y))\Kst(x,y)\dx\dy\nonumber \\
&=:\mbox{(III)}_{1}+\mbox{(III)}_{2}.
\end{flalign*}
The term $\mbox{(III)}_{1}$ can be estimated similarly to \rif{dimisi00}-\rif{dimisi}, i.e.:
\begin{flalign*}
\mbox{(III)}_{1}&\ge \frac 1{c_*}\int_{B_{\tau_{2}}}\int_{B_{\tau_{2}}}\frac{\snr{\eta^{m/\gamma}(x)\umm(x)-\eta^{m/\gamma}(y)\umm(y)}^{\gamma}}{\snr{x-y}^{n+s\gamma}}\dx\dy\nonumber \\
&\quad -c\int_{B_{\tau_{2}}}\int_{B_{\tau_{2}}}\frac{\max\left\{\umm(x),\umm(y)\right\}^{\gamma}\snr{\eta^{m/\gamma}(x)-\eta^{m/\gamma}(y)}^{\gamma}}{\snr{x-y}^{n+s\gamma}}\dx\dy\nonumber \\
&\ge \frac{1}{c_*}[\tilde u]_{s, \gamma;B_{\tau_1}}^{\gamma} -\frac{c\rr^{(1-s)\gamma}}{(\tau_{2}-\tau_{1})^{\gamma}}\int_{B_{\tau_2}}|\umm|^{\gamma}\dx
\ge \frac{1}{c_*}[\tilde u]_{s, \gamma;B_{\tau_1}}^{\gamma} -\frac{c\rr^{(1-s)\gamma}}{(\tau_{2}-\tau_{1})^{\gamma}}\int_{B_{\rr}}\snr{u-(u)_{B_{\rr}}}^{\gamma}\dx,
\end{flalign*}
for $c, c_*\equiv c, c_*(\data)$. As for $\mbox{(III)}_{2}$, we have 
\begin{flalign}
\notag \snr{\mbox{(III)}_{2}}&\le \frac 1{2c_*} [\tilde u]_{s, \gamma;B_{\tau_2}}^{\gamma}\nonumber+c\int_{B_{\tau_{2}}}\int_{B_{\tau_{2}}}\frac{\snr{\eta(x)\gmm(x)-\eta(y)\gmm(y)}^{\gamma}}{\snr{x-y}^{n+s\gamma}}\dx\dy\nonumber \\
\notag&\le \frac 1{2c_*}[\tilde u]_{s, \gamma;B_{\tau_2}}^{\gamma} +c[\tilde g]_{s, \gamma;B_{\tau_2}}^{\gamma}+\frac{c\tau_{2}^{(1-s)\gamma}}{(\tau_{2}-\tau_{1})^{\gamma}}\int_{B_{\tau_{2}}}\snr{\ti{g}-(\ti{u})_{B_{\tau_2}}}^{\gamma}\dx\nonumber \\
\notag &\leq \frac 1{2c_*}[\tilde u]_{s, \gamma;B_{\tau_2}}^{\gamma} +c[\tilde g]_{s, \gamma;B_{\tau_2}}^{\gamma}+\frac{c\tau_2^{(1-s)\gamma}}{(\tau_{2}-\tau_{1})^{\gamma}}\int_{B_{\tau_{2}}}\left(\snr{\ti{u}-\ti{g}}^{\gamma}+\snr{\ti{u}-(\ti{u})_{B_{\tau_2}}}^{\gamma}\right)\dx\nonumber \\
\notag &\leq \frac 1{2c_*}[\tilde u]_{s, \gamma;B_{\tau_2}}^{\gamma}+\frac{c\rr^{(1-s)\gamma}}{(\tau_{2}-\tau_{1})^{\gamma}}\int_{B_{\rr}^+}\snr{\ti{u}-\ti{g}}^{\gamma}\dx
+\frac{c\rr^{(1-s)\gamma}}{(\tau_{2}-\tau_{1})^{\gamma}}\int_{B_{\rr}}\snr{\ti{u}-(\ti{u})_{B_{\rr}}}^{\gamma}\dx\nonumber \\
& \quad \  +c\snr{B_{\rr}}\left(\rr^{\chi(a-s)}\int_{B_{\rr}}\mint_{B_{\rr}}\frac{\snr{\ti{g}(x)-\ti{g}(y)}^{\chi}}{\snr{x-y}^{n+a\chi}}\dx\dy\right)^{\gamma/\chi} \label{asas}
\end{flalign}
with $c, c_*\equiv c, c_*(\data)$, and we can assume that the constant $c_*$ appearing in the last two displays is the same. Note that in the last line we have also used \rif{averages} and \rif{holtri}. In order to estimate $\mbox{(IV)}$, we note 
\eqn{bb.0}
$$
x\in B_{(3\tau_{2}+\tau_{1})/4},\ y\in \mathbb{R}^{n}\setminus B_{\tau_{2}} \ \Longrightarrow \ 1 \leq  \frac{\snr{y-\ti{x}_{0}}}{\snr{x-y}}\le 1+ \frac{3\tau_{2}+\tau_{1}}{\tau_{2}-\tau_{1}}\leq \frac{4 \tau_{2}}{\tau_{2}-\tau_{1}}\,.
$$
Recalling that  $\eta$ is supported in $B_{(3\tau_{2}+\tau_{1})/4}$, and using \eqref{assc} and \eqref{asskt}, we get
\begin{eqnarray*}
\snr{\mbox{(IV)}}&\stackrel{\eqref{bb.0}}{\leq}&\frac{c\tau_2^{n+s\gamma}}{(\tau_2-\tau_1)^{n+s\gamma}}\int_{\mathbb{R}^{n}\setminus B_{\tau_{2}}}\int_{B_{\tau_{2}}}\frac{\snr{\umm(x)-\umm(y)}^{\gamma-1}\eta^{m}(x)\snr{\wmm(x)}}{\snr{y-\tilde x_0}^{n+s\gamma}}\dx\dy\nonumber \\
&\le&\frac{c\tau_2^{n}}{(\tau_2-\tau_1)^{n+s\gamma}}\int_{B_{\tau_{2}}}\snr{\umm}^{\gamma-1}\snr{\wmm}\dx\nonumber+\frac{c\tau_2^{n+s\gamma}}{(\tau_2-\tau_1)^{n+s\gamma}}\int_{\mathbb{R}^{n}\setminus B_{\tau_{2}}}\frac{\snr{\umm(y)}^{\gamma-1}}{\snr{y-\tilde{x}_0}^{n+s\gamma}} \dy\int_{B_{\tau_{2}}^{+}}\snr{\wmm}\dx\nonumber \\
&\stackrel{\eqref{cacc.2}}{\le}&\frac{c\tau_2^{n}}{(\tau_2-\tau_1)^{n+s\gamma}}\int_{B_{\tau_{2}}}\snr{\umm}^{\gamma}\dx +\frac{c\tau_2^{n}}{(\tau_2-\tau_1)^{n+s\gamma}}\int_{B_{\tau_{2}}^{+}}\snr{\wmm}^{\gamma}\dx\nonumber \\
&& \ +\frac{c\tau_2^{n+s(\gamma-1)}}{(\tau_{2}-\tau_{1})^{n+s\gamma}}\left(\int_{\mathbb{R}^{n}\setminus B_{\tau_{2}}}\frac{\snr{\ti{u}(y)-(\ti{u})_{B_{\tau_{2}}}}^{\gamma}}{\snr{y-\ti{x}_{0}}^{n+s\gamma}} \dy\right)^{1-1/\gamma}\int_{B_{\tau_{2}}^{+}}\snr{\wmm}\dx\nonumber \\
&\leq&\frac{c \tau_2^{n}}{(\tau_2-\tau_1)^{n+s\gamma}}\int_{B_{\tau_{2}}}\snr{\umm}^{\gamma}\dx +\frac{c\tau_2^{n}}{(\tau_2-\tau_1)^{n+s\gamma}}\int_{B_{\tau_{2}}^{+}}\snr{\wmm}^{\gamma}\dx\nonumber  \\
&& \ +\frac{c\tau_2^{n+s\gamma}}{(\tau_2-\tau_1)^{n+s\gamma}}\tau_2^{-\delta (1-1/\gamma)}|B_{\tau_2}|^{1-1/\gamma}[\snail_{\delta}(\tau_{2})]^{\gamma-1}\tau_2^{-s}\left(\int_{B_{\tau_{2}}^{+}}\snr{\wmm}^{\gamma}\dx\right)^{1/\gamma}\nonumber \\
&\le&\frac{c \rr^{n}}{(\tau_2-\tau_1)^{n+s\gamma}}\int_{B_{\tau_{2}}}\snr{\umm}^{\gamma}\dx +\frac{c\rr^{n}}{(\tau_2-\tau_1)^{n+s\gamma}}\int_{B_{\tau_{2}}^{+}}\snr{\wmm}^{\gamma}\dx\nonumber \\
&& \ +\frac{c\rr^{n+s\gamma}}{(\tau_{2}-\tau_{1})^{n+s\gamma}}\snr{B_{\tau_{2}}} \tau_{2}^{-\delta}[\snail_{\delta}(\tau_{2})]^{\gamma}\,.
\end{eqnarray*}
By further using \rif{scasnail} and \rif{averages}, we find
\begin{flalign*}
\snr{\mbox{(IV)}}&\leq  \frac{c \rr^{n}}{(\tau_2-\tau_1)^{n+s\gamma}}\int_{B_{\rr}}\snr{\ti{u}-(\ti{u})_{B_{\rr}}}^{\gamma}\dx +\frac{c\rr^{n}}{(\tau_2-\tau_1)^{n+s\gamma}}\int_{B_{\rr}^{+}}\snr{\ti{u}-\ti{g}}^{\gamma}\dx\nonumber \\
& \quad +\frac{c\rr^{n+s\gamma}}{(\tau_{2}-\tau_{1})^{n+s\gamma}}\left(\snr{B_{\rr}}\rr^{-\delta}[\snail_{\delta}(\rr)]^{\gamma}+\rr^{-s\gamma}\int_{B_{\rr}}\snr{\ti{u}-(\ti{u})_{B_{\rr}}}^{\gamma}\dx\right)
\end{flalign*}
for $c\equiv c(\data)$. 
Merging the estimates for terms $\mbox{(I)}$-$\mbox{(IV)}$, and again using \rif{averages}, yields
\begin{flalign*}
& \int_{B_{\tau_1}^{+}}(\snr{D\ti{u}}^{2}+\mu^{2})^{p/2}\dx +[\tilde u]_{s, \gamma;B_{\tau_1}}^{\gamma} \leq \frac12  [\tilde u]_{s, \gamma;B_{\tau_2}}^{\gamma}+\frac{c}{(\tau_{2}-\tau_{1})^{p}}\int_{B_{\rr}^{+}}\snr{\ti{u}-\ti{g}}^{p}\dx\nonumber\\
&  \qquad  \qquad \qquad\qquad +c\left[ \frac{ \rr^{n}}{(\tau_{2}-\tau_{1})^{n+s\gamma}}+\frac{\rr^{(1-s)\gamma}}{(\tau_{2}-\tau_{1})^{\gamma}}\right]
\left(\int_{B_{\rr}^{+}}\snr{\tilde u -\tilde g}^{\gamma}\dx+\int_{B_{\rr}}\snr{\ti{u}-(\ti{u})_{B_{\rr}}}^{\gamma}\dx\right)\nonumber\\
&\qquad  \qquad \qquad   \qquad+ \frac{c \rr^{n+s\gamma}}{(\tau_{2}-\tau_{1})^{n+s\gamma}}\snr{B_{\rr}}\rr^{-\delta}[\snail_{\delta}(\rr)]^{\gamma}
\nonumber+c\snr{B_{\rr}}\left(\nr{\ti{f}}_{L^{n}(B^{+}_{\rr})}^{p/(p-1)}+1\right)\\
&\qquad  \qquad  \qquad  \qquad+c\snr{B_{\rr}}\left(\mint_{B_{\rr}^{+}}\snr{D\ti{g}}^{q}\dx\right)^{p/q}\nonumber+c\snr{B_{\rr}}\left(\rr^{\chi(a-s)}\int_{B_{\rr}}\mint_{B_{\rr}}\frac{\snr{\ti{g}(x)-\ti{g}(y)}^{\chi}}{\snr{x-y}^{n+a\chi}}\dx\dy\right)^{\gamma/\chi}\nonumber
\end{flalign*}
with $c\equiv c(\data)$. Applying Lemma \ref{l5} with the choice 
$$
h(t):=\int_{B_{t}^{+}}(\snr{D\ti{u}}^{2}+\mu^{2})^{p/2}\dx + [\tilde u]_{s, \gamma;B_{t}}^{\gamma} 
$$
now yields, after a few manipulations,and recalling he definition in \rif{cccb}
\begin{eqnarray}
\notag && \mint_{B_{\rr/2}^{+}}(\snr{D\ti{u}}^{2}+\mu^{2})^{p/2}\dx + |B_{\rr}|^{-1}[\tilde u]_{s, \gamma;B_{\rr/2}}^{\gamma} \nonumber \leq c\rr^{-p}\mint_{B_{\rr}^{+}}\snr{\ti{u}-\ti{g}}^{p}\dx\\
\notag && \qquad\quad+  c\rr^{-s\gamma}\mint_{B_{\rr}^+}\snr{\tilde u -\tilde g}^{\gamma}\dx+  c\rr^{-s\gamma}\mint_{B_{\rr}}\snr{\ti{u}-(\ti{u})_{B_{\rr}}}^{\gamma}\dx+c\rr^{-\delta}[\snail_{\delta}(\rr)]^{\gamma}\nonumber\\
\notag &&  \qquad \quad+c\left(\nr{\ti{f}}_{L^{n}(B^{+}_{\rr})}^{p/(p-1)}+1\right)+\left(\mint_{B_{\rr}^{+}}\snr{D\ti{g}}^{q}\dx\right)^{p/q}+c \left(\rr^{\chi(a-s)}\int_{B_{\rr}}\mint_{B_{\rr}}\frac{\snr{\ti{g}(x)-\ti{g}(y)}^{\chi}}{\snr{x-y}^{n+a\chi}}\dx\dy\right)^{\gamma/\chi}\\
&& \quad\  \leq  c\rr^{-s\gamma}[\av_{\gamma}(\rr)]^{\gamma}+c\rr^{-s\gamma}\mint_{B_{\rr}^+}\snr{\tilde u -\tilde g}^{\gamma}\dx+ c \, \cacc^{+}(\rr)\,.\label{sending}
\end{eqnarray}
Then we have
\begin{eqnarray*}
\rr^{-s\gamma}[\av_{\gamma}(\rr)]^{\gamma}& \stackleq{averages} &
c\rr^{-s\gamma}\mint_{B_{\rr}^+} \snr{\ti{u}-\ti{g}}^{\gamma}\dx  + c\rr^{-s\gamma} \mint_{B_{\rr}} \snr{\ti{g}- (\ti{g})_{B_{\rr}}}^{\gamma}\dx \\
&  \stackleq{fraso}  &
c\rr^{-s\gamma}\mint_{B_{\rr}^+} \snr{\ti{u}-\ti{g}}^{\gamma}\dx  + c|B_{\rr}|^{-1}[\tilde g]_{s,\gamma;B_{\rr}}^\gamma\\
&  \stackleq{cccb}  & \rr^{-s\gamma}\mint_{B_{\rr}^+} \snr{\ti{u}-\ti{g}}^{\gamma}\dx+ c \left(\rr^{\chi(a-s)}\int_{B_{\rr}}\mint_{B_{\rr}}\frac{\snr{\ti{g}(x)-\ti{g}(y)}^{\chi}}{\snr{x-y}^{n+a\chi}}\dx\dy\right)^{\gamma/\chi}\\
& \leq &\rr^{-s\gamma}\mint_{B_{\rr}^+} \snr{\ti{u}-\ti{g}}^{\gamma}\dx+ c \, \cacc^{+}(\rr)\,.
\end{eqnarray*}
On the other hand, proceeding as in the proof of \rif{intermedia2}, we obtain
\begin{flalign}
\nonumber \rr^{-s\gamma}\mint_{B_{\rr}^+} \snr{\ti{u}-\ti{g}}^{\gamma}\dx
&\leq  c \left(\|\ti{u}\|_{L^\infty(\er^n)}+\|\ti{g}\|_{L^\infty(\er^n)}\right)^{(1-\vartheta)\gamma} \rr^{(\vartheta-s)\gamma}
\left(\rr^{-p}\mint_{B_{\rr}^+} \snr{\ti{u}-\ti{g}}^{p}\dx\right)^{\vartheta\gamma/p}\\
&\hspace{-1.5mm}\stackleq{boundprop2}  c\rr^{(\vartheta-s)\gamma}[ \cacc^{+}(\rr)]^{\vartheta\gamma/p} \leq c \, \cacc^{+}(\rr)
\,,\label{intermedia3}
\end{flalign}
with $c\equiv c (\data)$, as $\cacc^{+}(\rr)\geq 1 \geq \rr$ and $p \geq \vartheta \gamma$, and therefore, from the content of the last two displays, we conclude with 
\eqn{bb.8}
$$
\rr^{-s\gamma}[\av_{\gamma}(\rr)]^{\gamma}\leq c\, \cacc^{+}(\rr)\,.
$$
Using the last two inequalities in \rif{sending} finally leads to \eqref{cacc+}.
\end{proof}
\subsection{Step 5: Boundary $p$-harmonic functions}\label{clb} Here we have 
\begin{lemma}\label{bplemma}
Let $\ti{h}\in \ti{u}+W^{1,p}_{0}(B_{\rr/4}^{+}(\ti{x}_0))$ be the solution to
\eqn{pddb}
$$
\ti h\mapsto \min_{\ti{w} \in \ti{u}+W^{1,p}_{0}(B_{\rr/4}^+(\ti{x}_{0}))} \int_{B_{\rr/4}(\ti{x}_{0})}\ccc(\ti{x}_0)\ti{F}(\ti{x}_{0},D\ti{w}) \dx\,.
$$ Then 
\eqn{bb.12}
$$
\mint_{B_{\rr/4}^{+}(\ti{x}_0)}\snr{\ti{u}-\ti{h}}^{p}\dx \le c\rr^{\theta\ti{\sigma} }[\GGp(\ti{u},B_{\rr}(\ti{x}_0))]^{p}
$$
holds for any $\theta\in (0,1)$, where $c\equiv c(\data)$. Here $\ti{\sigma}\equiv \ti{\sigma}(p,s,\gamma, \alpha_b)\in (0,1)$ is given by $\ti{\sigma}:=\min\{\sigma, \alpha_b, p\alpha_b/2\}$ and $\sigma$ comes from \eqref{cl.11}. 
\end{lemma}
\begin{proof}
We shall abbreviate, as usual, $B_{\rr}^{+}\equiv B_{\rr}^{+}(\ti{x}_{0})$. From \rif{pddb} it follows that 
\eqn{elpdb}
$$
\int_{B_{\rr/4}^{+}}\ccc(\ti{x}_{0})\partial_{z}\ti{F}(\ti{x}_{0},D\ti{h}) \cdot D\varphi\dx=0\qquad \mbox{for all} \ \ \varphi\in W^{1,p}_{0}(B_{\rr/4}^{+})
$$
and, as for \rif{enes}-\rif{bhh}
\eqn{b5}
$$
\mint_{B_{\rr/4}^{+}}(\snr{D\ti{h}}^{2}+\mu^{2})^{p/2}\dx \le \ti{\Lambda}^2\mint_{B_{\rr/4}^{+}}(\snr{D\ti{u}}^{2}+\mu^{2})^{p/2}\dx \,, \quad \nr{\ti{h}}_{L^{\infty}(B_{\rr/4}^{+})}\le \nr{\ti{u}}_{L^{\infty}(B_{\rr/4}^{+})}
$$
hold. As $\ti{h}= \ti{u}$ on $\partial B_{\rr/4}^{+}$ (in the sense of traces), we define $\ti{w}:=\ti{u}-\ti{h}\in W^{1,p}_0(B_{\rr/4}^{+})$ and extend it to the whole $\er^n$ by setting $\ti{w}\equiv 0$ in $\mathbb{R}^{n}\setminus B_{\rr/4}^{+}$. This implies $\ti{w}\in \mathbb{X}_{0}(\ti{\Omega})$, 
so that $\ti{w}$ is an admissible test function for both \eqref{elb} and \eqref{elpdb}. Indeed, note that 
$\ti{w} \in W^{1,p}_0(B_{\rr/2})\cap L^{\infty}(\er^n)$ and therefore by Lemma \ref{bmfinal} it follows that $\ti{w}\in W^{s,\gamma}(B_{\rr/2})$. As $\ti{w}\equiv 0$ outside $B_{\rr/4}$, it follows that $\ti{w}\in W^{s, \gamma}(\er^n)$ by \cite[Lemma 5.1]{guide}, and therefore $\ti{w}\in \mathbb{X}_{0}(\ti{\Omega})$. This means that $\ti{w}$ can be used as a test function both in \rif{elb} and in \rif{elpdb}. Moreover, by \rif{cacc+} and \rif{b5}, it follows that 
\eqn{stimaww}
$$
\mint_{B_{\rr/4}^{+}}(\snr{D\ti{w}}^{2}+\mu^{2})^{p/2}\dx\leq c  \mint_{B_{\rr/4}^{+}}(\snr{D\ti{u}}^{2}+\mu^{2})^{p/2}\dx \leq c\, \cacc^{+}(\rr)\,.
$$
With $\ti{\mathcal{V}}^{2}:=\snr{V_{\mu}(D\ti{u})-V_{\mu}(D\ti{h})}^{2}$, we estimate (via inequality \rif{monoin} applied to $\partial_z \tilde F$, as allowed by \rif{assft}$_4$)
\begin{eqnarray}
\nonumber \frac1c\mint_{B_{\rr/4}^{+}}\ti{\mathcal{V}}^{2}\dx&\stackleq{monoin} &\mint_{B_{\rr/4}^{+}}\ccc(\ti{x}_{0})(\partial \ti{F}(\ti{x}_{0},D\ti{u})-\partial\ti{F}(\ti{x}_{0},D\ti{h}))\cdot D\ti{w}\dx\nonumber\\
&\stackrel{\rif{elpdb}}{=}&\mint_{B_{\rr/4}^{+}}\ccc(\ti{x}_{0})\partial \ti{F}(\ti{x}_{0},D\ti{u})\cdot D\ti{w}\dx\nonumber\\
&\stackrel{\rif{elb}}{=} &\mint_{B_{\rr/4}^{+}}[\ccc(\ti{x}_{0})-\ccc(x)]\partial\ti{F}(\ti{x}_{0},D\ti{u})\cdot D\ti{w} \dx\nonumber \\
\nonumber && \ +\mint_{B_{\rr/4}^{+}}\ccc(x)( \partial \ti{F}(\ti{x}_{0},D\ti{u})-\partial \ti{F}(x,D\ti{u}))\cdot D\ti{w}\dx+\mint_{B_{\rr/4}^{+}}\ccc(x)\ti{f}\ti{w}\dx\nonumber \\
\nonumber &&\ -\snr{B_{\rr/4}^{+}}^{-1}\int_{\mathbb{R}^{n}}\int_{\mathbb{R}^{n}}|\ti{u}(x)-\ti{u}(y)|^{\gamma-2}(\ti{u}(x)-\ti{u}(y))(\ti{w}(x)-\ti{w}(y))\Kst(x,y)\dx\dy\nonumber \\
\nonumber &\stackrel{\eqref{assft}_5}{\le}&c\rr^{\alpha_b}\mint_{B_{\rr/4}^{+}}(\snr{D\ti{u}}^{2}+\mu^{2})^{(p-1)/2}\snr{D\ti{w}}\dx+c\mint_{B_{\rr/4}^{+}}\snr{\ti{f}\ti{w}}\dx\nonumber \\
\nonumber && \ +c\int_{B_{\rr/2}}\mint_{B_{\rr/2}}\frac{\snr{\ti{u}(x)-\ti{u}(y)}^{\gamma-1}\snr{\ti{w}(x)-\ti{w}(y)}}{\snr{x-y}^{n+s\gamma}}\dx\dy\nonumber \\
\nonumber && \ +c\int_{\mathbb{R}^{n}\setminus B_{\rr/2}}\mint_{B_{\rr/2}}\frac{\snr{\ti{u}(x)-\ti{u}(y)}^{\gamma-1}\snr{\ti{w}(x)}}{\snr{x-y}^{n+s\gamma}}\dx\dy\nonumber \\
&=:&\mbox{(O)}+\mbox{(I)}+\mbox{(II)}+\mbox{(III)},\label{mettila}
\end{eqnarray}
where $c\equiv c(n,p,\ti{\Lambda})$; we have also used \rif{assc}. The first two terms can be controlled via Sobolev inequality
\begin{eqnarray*}
\mbox{(O)}+\mbox{(I)}
&\stackrel{\eqref{cacc+}}{\le}&c\left[\rr^{\alpha_b}[\cacc^{+}(\rr)]^{1-1/p}+\nr{\ti{f}}_{L^{n}(B_{\rr/4}^{+})}\right]\left(\mint_{B_{\rr/4}^{+}}\snr{D\ti{w}}^{p}\dx\right)^{1/p}\nonumber \\
&\stackrel{\eqref{stimaww}}{\le}&c\rr^{\alpha_b}\cacc^{+}(\rr)+c\nr{\ti{f}}_{L^{n}(B_{\rr/4}^{+})}[\cacc^{+}(\rr)]^{1/p},
\end{eqnarray*}
with $c\equiv c(\data)$ (also recall \rif{cl.3}). The term $\mbox{(II)}$ can be estimated as the homonym term in Lemma \ref{har}, but this time using  \rif{boundprop2} and \rif{stimaww}; this yields
$$
\mbox{(II)} \stackrel{\eqref{cacc+}}{\le}c[\cacc^{+}(\rr)]^{1-1/\gamma}\left(\int_{B_{\rr/4}}\mint_{B_{\rr/4}}\frac{\snr{w(x)-w(y)}^{\gamma}}{\snr{x-y}^{n+s\gamma}}\dx\dy\right)^{1/\gamma}\stackrel{\eqref{bm111}}{\leq} c\rr^{\vartheta-s}[\cacc^{+}(\rr)]^{1-1/\gamma+\vartheta/p}
$$
where $\vartheta$ is in \eqref{t1t2} and $c\equiv c(\data)$. Now, similarly to \rif{intermedia2}, but using \rif{boundprop2} and \rif{b5}-\rif{stimaww}, we find
\eqn{bb.7}
$$
\left(\mint_{B_{\rr/4}^{+}}\snr{\ti{w}}^{\gamma}\dx\right)^{1/\gamma}  \leq   
c\nr{\ti{u}}_{L^\infty(B_{\rr/4}^{+})}^{1-\vartheta}\rr^{\vartheta}
\left(\mint_{B_{\rr/4}^{+}}\snr{D\ti{w}}^{p}\dx\right)^{\vartheta/p} \leq c\rr^{\vartheta}[\cacc^{+}(\rr)]^{\vartheta/p}\,.
$$
We then have
\begin{eqnarray*}
\mbox{(III)}&\le&c\int_{\mathbb{R}^{n}\setminus B_{\rr/2}}\mint_{B_{\rr/2}}\frac{\max\{\snr{\ti{u}(x)-(\ti{u})_{B_{\rr/2}}},\snr{\ti{u}(y)-(\ti{u})_{B_{\rr/2}}}\}^{\gamma-1}\snr{\ti{w}(x)}}{\snr{y-\ti{x}_0}^{n+s\gamma}}\dx\dy\nonumber \\
&\le&c\rr^{-s\gamma}\left(\mint_{B_{\rr/2}}\snr{\ti{u}(x)-(\ti{u})_{B_{\rr/2}}}^{\gamma}\dx\right)^{1-1/\gamma}\left(\mint_{B_{\rr/4}^{+}}\snr{\ti{w}}^{\gamma}\dx\right)^{1/\gamma}\nonumber \\
&& \quad  +c \int_{\mathbb{R}^{n}\setminus B_{\rr/2}}\frac{\snr{\ti{u}(y)-(\ti{u})_{B_{\rr/2}}}^{\gamma-1}}{\snr{y-\ti{x}_0}^{n+s\gamma}} \dy \mint_{B_{\rr/4}^{+}}\snr{\ti{w}}\dx\nonumber \\
&\stackrel{\eqref{fraso},\eqref{cacc.2}}{\le}&c\rr^{-s}\left(\int_{B_{\rr/2}}\mint_{B_{\rr/2}}\frac{\snr{\ti{u}(x)-\ti{u}(y)}^{\gamma}}{\snr{x-y}^{n+s\gamma}}\dx\dy\right)^{1-1/\gamma}\left(\mint_{B_{\rr/4}^{+}}\snr{\ti{w}}^{\gamma}\dx\right)^{1/\gamma}\nonumber \\
&& \quad +c \rr^{-s}\left(\int_{\mathbb{R}^{n}\setminus B_{\rr/2}}\frac{\snr{\ti{u}(y)-(\ti{u})_{B_{\rr/2}}}^{\gamma}}{\snr{y-\ti{x}_{0}}^{n+s\gamma}} \dy\right)^{1-1/\gamma}\left(\mint_{B_{\rr/4}^{+}}\snr{\ti{w}}^{\gamma}\dx\right)^{1/\gamma} \nonumber \\
&\stackrel{\eqref{cacc+},\eqref{bb.7}}{\le}&c\rr^{\vartheta-s}[\cacc^{+}(\rr)]^{1-1/\gamma+\vartheta/p}+c\rr^{\vartheta-s-\delta(\gamma-1)/\gamma}[\snail_{\delta}(\rr/2)]^{\gamma-1}[\cacc^{+}(\rr)]^{\vartheta/p}\nonumber \\
&\stackrel{\eqref{scasnail},\rif{averages}}{\le}&c\rr^{\vartheta-s}[\cacc^{+}(\rr)]^{1-1/\gamma+\vartheta/p}+c\rr^{\vartheta-s}\left(\rr^{-\delta}[\snail_{\delta}(\rr)]^{\gamma}\right)^{1-1/\gamma}[\cacc^{+}(\rr)]^{\vartheta/p}\nonumber \\
&& \quad +c\rr^{\vartheta-s}\left(\rr^{-s\gamma}[\av_{\gamma}(\rr)]^{\gamma}\right)^{1-1/\gamma}[\cacc^{+}(\rr)]^{\vartheta/p}\\
&\stackleq{bb.8}&c\rr^{\vartheta-s}[\cacc^{+}(\rr)]^{1-1/\gamma+\vartheta/p},
\end{eqnarray*}
with $c\equiv c(\data)$. Combining the estimates for the terms $\mbox{(O)}, \mbox{(I)}, \mbox{(II)} $ and $\mbox{(III)}$ with \rif{mettila}, we obtain
\eqn{bb.9}
$$
\mint_{B_{\rr/4}^{+}}\ti{\mathcal{V}}^{2}\dx \le c\rr^{\alpha_b}\, \cacc^{+}(\rr)+c\nr{\ti{f}}_{L^{n}(B_{\rr/4}^{+})}[\cacc^{+}(\rr)]^{1/p}+c\rr^{\vartheta-s}[\cacc^{+}(\rr)]^{1-1/\gamma+\vartheta/p},
$$
for $c\equiv c(\data)$. This is the boundary analog of \rif{cl.6}. We can then proceed as in \rif{uh1}-\rif{uh4}, but using \rif{bb.9} instead of \rif{cl.6}, and \rif{late} instead of \rif{triviala}, to obtain
$$
\mint_{B_{\rr/4}^{+}}\snr{\ti{u}-\ti{h}}^{p}\dx\le c \rr^{\alpha_b\min\{1,p/2\}+p}\cacc^{+}(\rr)+ c\rr^{\theta\sigma}[\GGp(\rr)]^{p}\,,
$$
where $\sigma$ is as in Lemma \ref{har}, and from which \rif{bb.12} follows again using \rif{late}. 
\end{proof}
\subsection{Step 6: Completion of the proof Theorem \ref{t4}}\label{iterb} We keep on using half-balls centred at a generic point $\ti{x}_0$ as described in Section \ref{listabordo}. We start with a further decay estimate satisfied by $\ti{h}$ defined in \rif{pddb}. This is
\eqn{breg}
$$
\int_{B_{t}^{+}}(\snr{D\ti{h}}^{2}+\mu^{2})^{p/2}\dx\le c\left(\frac{t}{\rr}\right)^{\texttt{b}}\int_{B_{\rr/4}^{+}}(\snr{D\ti{h}}^{2}+\mu^{2})^{p/2}\dx +ct^{n(1-p/q)}\left(\int_{B_{\rr/4}^{+}}\snr{D\ti{g}}^{q}\dx\right)^{p/q}
$$
that holds whenever $t \leq \rr/4$ and $\texttt{b}$ such that $0 \leq \texttt{b} <n $ and $c\equiv c(\data,q,\texttt{b})$. We postpone the proof of \rif{breg} to Section \ref{stimaduz} below. We begin considering positive $\texttt{b}$ such that 
\eqn{condizioneb}
$$
 \frac{n-\texttt{b}}{p} <   \frac{n}{q}  \Longrightarrow \texttt{b}> n\left(1-\frac pq\right)\,.
$$
This fixes $\texttt{b}$ as a function of $n,p,s,q$.  
For positive $t \leq \rr/8$, recalling that $\tilde h \equiv \ti{g}$ on $\Gamma_{t}(\ti{x}_0)$, we have
\begin{eqnarray*}
\left(\mint_{B_{t}^{+}}\snr{\ti{u}-\ti{g}}^p\dx \right)^{1/p}&\leq& 
\left(\mint_{B_{t}^{+}}\snr{\ti{u}-\ti{h}}^p\dx \right)^{1/p}+
\left(\mint_{B_{t}^{+}}\snr{\ti{h}-\ti{g}}^p\dx \right)^{1/p} \notag
\\
&\stackrel{\eqref{bb.12}}{\le}&c\rr^{\theta\ti{\sigma}/p}\left(\frac{\rr}{t}\right)^{n/p}\GGp(\rr)\nonumber  +ct\left(\mint_{B_{t}^{+}}(\snr{D\ti{h}}^p+\snr{D\ti{g}}^p)\dx\right)^{1/p}\nonumber \\
&\stackrel{\eqref{b5},\eqref{breg}}{\le}&c\rr^{\theta\ti{\sigma}/p}\left(\frac{\rr}{t}\right)^{n/p}\GGp(\rr)\nonumber \\
&& \ +ct\left(\frac{t}{\rr}\right)^{\texttt{b}/p-n/p}\left(\mint_{B_{\rr/4}^{+}}(\snr{D\ti{u}}^{2}+\mu^{2})^{p/2}\dx\right)^{1/p}\nonumber \\
&& \ +c\left(\frac{t}{\rr}\right)^{1-n/q}\left(\rr^{q}\mint_{B_{\rr/4}^{+}}\snr{D\ti{g}}^{q}\dx\right)^{1/q}\nonumber \\
&\stackrel{\eqref{cacc+}}{\le}&c\rr^{\theta\ti{\sigma}/p}\left(\frac{\rr}{t}\right)^{n/p}\GGp(\rr)\nonumber \\
&& \ +c\left(\frac{t}{\rr}\right)^{1+\texttt{b}/p-n/p}[\rr^{p}\cacc^{+}(\rr)]^{1/p}+c\left(\frac{t}{\rr}\right)^{1-n/q}\rhsp(\rr)\,.
\end{eqnarray*}
By using \rif{late} and recalling \rif{condizioneb}, we conclude with 
\eqn{bb.13}
$$
\left(\mint_{B_{t}^{+}}\snr{\ti{u}-\ti{g}}^p\dx \right)^{1/p}\leq c\left[\left(\frac{t}{\rr}\right)^{1-n/q}+\rr^{\theta\ti{\sigma}/p}\left(\frac{\rr}{t}\right)^{n/p}\right]\GGp(\rr)\,
$$
with $c\equiv c(\data)$. Next observe that, using \rif{holtri} and recalling the definitions in Section \rif{listabordo}, we find 
\begin{flalign}
\notag \left(\int_{B_{t}}\mint_{B_{t}}\frac{\snr{\ti{g}(x)-\ti{g}(y)}^{\gamma}}{\snr{x-y}^{n+s\gamma}}\dx\dy\right)^{1/\gamma}& \leq 
c\left(t^{\chi(a-s)}\int_{B_{t}}\mint_{B_{t}}\frac{\snr{\ti{g}(x)-\ti{g}(y)}^{\chi}}{\snr{x-y}^{n+a\chi}}\dx\dy\right)^{1/\chi}\\
\notag & \leq c \left(\frac{t}{\rr}\right)^{a-s-n/\chi} 
\left(\rr^{\chi(a-s)}\int_{B_{\rr}}\mint_{B_{\rr}}\frac{\snr{\ti{g}(x)-\ti{g}(y)}^{\chi}}{\snr{x-y}^{n+a\chi}}\dx\dy\right)^{1/\chi}\\
& \leq \frac{c}{t^{s}} \left(\frac{t}{\rr}\right)^{a-n/\chi} [\rhsp(\rr)]^{\vartheta}
\leq \frac{c}{t^{s}} \left(\frac{t}{\rr}\right)^{a-n/\chi} [\GGp(\rr)]^{\vartheta}\label{bopi}\,,
\end{flalign}
where $\vartheta$ is defined in \eqref{t1t2}. 
Using also \rif{basicav} we have
\begin{eqnarray}
\av_{\gamma}(t)&\le&c\left(\mint_{B_{t}^{+}}\snr{\ti{u}-\ti{
g}}^\gamma\dx \right)^{1/\gamma}+c\left(\mint_{B_{t}^{+}}\snr{\ti{g}-(\ti{
g})_{B_t}}^\gamma\dx \right)^{1/\gamma}\nonumber \\
&\stackrel{\eqref{fraso}, \eqref{intermedia3}}{\le}&
c\left(\mint_{B_{t}^{+}}\snr{\ti{u}-\ti{g}}^p\dx \right)^{\vartheta/p}+ct^{s}\left(\int_{B_{t}}\mint_{B_{t}}\frac{\snr{\ti{g}(x)-\ti{g}(y)}^{\gamma}}{\snr{x-y}^{n+s\gamma}}\dx\dy\right)^{1/\gamma}\,.
\label{disum}
\end{eqnarray}
This with \eqref{bb.13} and \eqref{bopi} gives
\eqn{bb.14}
$$
 \av_{\gamma}(t)  \leq c\left[\left(\frac{t}{\rr}\right)^{1-n/q}+\rr^{\theta\ti{\sigma}/p}\left(\frac{\rr}{t}\right)^{n/p}\right]^{\vartheta}[\GGp(\rr)]^{\vartheta}+c\left(\frac{t}{\rr}\right)^{a-n/\chi}[\GGp(\rr)]^{\vartheta}\,
$$
with $c\equiv c(\data)$. Estimates \rif{bopi}-\rif{disum} continue to work when $\rr/8 < t \leq \rr$, so that 
\eqn{bb.17}
$$
\av_{\gamma}(\rr) \leq c\left(\mint_{B_{\rr}^{+}}\snr{\ti{u}-\ti{g}}^p\dx \right)^{\vartheta/p}+c[\rhsp(\rr)]^{\vartheta}\leq c [\GGp(\rr)]^{\vartheta}
$$
holds and we can conclude that \rif{bb.14} takes place in the full range $0 < t \leq \rr$.  
Taking $t=\tau\rr$ in \rif{bb.13}, with $0< \tau \leq 1/8$, yields
\eqn{ecce1}
$$
\left(\mint_{B_{\tau\rr}^{+}}\snr{\ti{u}-\ti{g}}^{p}\dx\right)^{1/p} \leq c\left(\tau^{1-n/q}+\rr^{\theta\ti{\sigma}/p}\tau^{-n/p}\right)\GGp(\rr),
$$
for $c\equiv c(\data)$. 
As for the \texttt{snail}, we have
\begin{eqnarray}
\notag [\snail_{\delta}(\tau\rr)]^{\gamma}&\stackrel{\eqref{scasnail}}{\le} & c\tau^{\delta}[\snail_{\delta}(\rr)]^{\gamma}+c(\tau\rr)^{\delta}\left(\int_{\tau\rr}^{\rr}\frac{\av_{\gamma}(\nu)}{\nu^s} \, \frac{\dtau}{\nu}\right)^{\gamma}\nonumber+c\tau^{\delta}\rr^{\delta-s\gamma}[\av_{\gamma}(\rr)]^{\gamma}\\ &=: & S_5+S_6+S_7\,.
\label{bb.16}
\end{eqnarray}
We have $S_5\leq c \tau^{\delta}[\GGp(\rr)]^{p}$ by \rif{ilglobale}. For $S_6$ we use \rif{bb.14} to estimate $\av_{\gamma}(\nu)$ inside the integral, and in turn estimate separately the resulting three pieces $S_{6.1}, S_{6.2}$ and $S_{6.3}$ generated by the terms appearing in the right-hand side of \rif{bb.14}. 
To estimate $S_{6.1}$ we first consider the case $s \leq 1-n/q$; we have
\begin{flalign*}
S_{6.1} & \leq c \tau^{\delta} \rr^{\delta- \vartheta\gamma(1-n/q)}
\left(\int_{\tau\rr}^{\rr} \frac{\dtau }{\nu^{1+s-\vartheta(1-n/q)}}
\right)^{\gamma}[\GGp(\rr)]^{\vartheta\gamma}\\
& \leq c\mathds{A}_{\gamma}\tau^{\delta-s\gamma n/q}\rr^{\delta-s\gamma} [\GGp(\rr)]^{s\gamma}+ c(\mathds{B}_{\gamma}+\mathds{C}_{\gamma})\tau^{\delta} \log^{\gamma} \left(\frac{1}{\tau}\right)\rr^{\delta-s\gamma} [\GGp(\rr)]^{\gamma}\\
& \leq c \tau^{\delta-s\gamma n/q}[\GGp(\rr)]^{p}+  c(\mathds{A}_{\gamma}+\mathds{B}_{\gamma}) \tau^{\delta-s\gamma n/q}\rr^{\frac{p(\delta - s\gamma)}{p-\vartheta \gamma}} \\ &\leq c \tau^{\delta-s\gamma n/q}[\GGp(\rr)]^{p}\leq c \tau^{\delta- np/q}[\GGp(\rr)]^{p} \,.
\end{flalign*}
The other case is when $s > 1-n/q$, and we have, similarly
\begin{flalign*}
S_{6.1} & \leq c \tau^{\delta} \rr^{\delta- \vartheta\gamma(1-n/q)}
\left(\int_{\tau\rr}^{\infty} \frac{\dtau }{\nu^{1+s-\vartheta(1-n/q)}}
\right)^{\gamma}[\GGp(\rr)]^{\vartheta\gamma}\\
& \leq c\mathds{A}_{\gamma}\tau^{\delta-s\gamma n/q}\rr^{\delta-s\gamma} [\GGp(\rr)]^{s\gamma}+ c(\mathds{B}_{\gamma}+\mathds{C}_{\gamma})\tau^{\delta-s\gamma+\gamma(1-n/q)}\rr^{\delta-s\gamma} [\GGp(\rr)]^{\gamma}\\
& \leq c\mathds{A}_{\gamma}\tau^{\delta-s\gamma n/q}\rr^{\delta-s\gamma} [\GGp(\rr)]^{s\gamma}+ c(\mathds{B}_{\gamma}+\mathds{C}_{\gamma})\tau^{\delta-s\gamma n/q}\rr^{\delta-s\gamma} [\GGp(\rr)]^{\gamma}\\ &\leq c \tau^{\delta-s\gamma n/q}[\GGp(\rr)]^{p}\\ &\leq c \tau^{\delta- np/q}[\GGp(\rr)]^{p}\,.
\end{flalign*}
Note that here we have used $\delta-s\gamma n/q < \delta-s\gamma+\gamma(1-n/q)$, implied by $s<1$. Moreover, 
\begin{flalign*}
S_{6.2}  & \leq c \tau^{\delta} \rr^{\delta +\vartheta \gamma(\theta\ti{\sigma}+n)/p}\left(
\int_{\tau\rr}^{\rr} \frac{\dtau}{\nu^{1+s+\vartheta n/p}}
\right)^{\gamma}[\GGp(\rr)]^{\vartheta\gamma}
\\& 
\leq c \tau^{\delta-s\gamma-n\vartheta\gamma/p}\rr^{\theta \ti{\sigma} \vartheta\gamma/p}\rr^{\delta -s\gamma}  [\GGp(\rr)]^{\vartheta\gamma}\\ 
& \leq    \rr^{ \theta \ti{\sigma} \vartheta\gamma/p}\tau^{-n\vartheta \gamma/p}
[\GG(\rr)]^{p}+ c(\mathds{A}_{\gamma}+\mathds{B}_{\gamma}) \rr^{ \theta \ti{\sigma} \vartheta\gamma/p}\tau^{-n\vartheta \gamma/p} \rr^{\frac{p(\delta - s\gamma)}{p-\vartheta \gamma}}\\
& \leq c \rr^{ \theta \ti{\sigma} \vartheta\gamma/p}\tau^{-n\vartheta \gamma/p}[\GG(\rr)]^{p}\\ & 
\leq c \rr^{ \theta \ti{\sigma} \vartheta\gamma/p}\tau^{-n}[\GG(\rr)]^{p}\,.
\end{flalign*}
For $S_{6.3} $ we first consider the case $a -\chi/n \geq s$, and we have
\begin{flalign*}
S_{6.3} &\leq c \tau^{\delta} \rr^{\delta-\gamma(a-n/\chi)}
\left(\int_{\tau\rr}^{\rr} \frac{\dtau }{\nu^{1+s-a+n/\chi}}
\right)^{\gamma}[\GGp(\rr)]^{\vartheta\gamma}\\
& \leq c\tau^{\delta}\log^{\gamma} \left(\frac 1\tau\right)\rr^{\delta-s\gamma} [\GGp(\rr)]^{\vartheta\gamma}\\
& \leq c\tau^{\delta} \log^{\gamma} \left(\frac{1}{\tau}\right)[\GGp(\rr)]^{p}+  c(\mathds{A}_{\gamma}+\mathds{B}_{\gamma}) \tau^{\delta} \log^{\gamma} \left(\frac{1}{\tau}\right)\rr^{\frac{p(\delta - s\gamma)}{p-\vartheta \gamma}} \\ &\leq c \tau^{\delta} \log^{\gamma} \left(\frac{1}{\tau}\right)[\GGp(\rr)]^{p}\\ &\leq c \tau^{\delta(a-n/\chi)}[\GGp(\rr)]^{p}
\,.
\end{flalign*} When $a -\chi/n < s$, using also that $\delta(a-n/\chi)<\delta-s\gamma+\gamma(a-n/\chi)$, we instead have
\begin{flalign*}
S_{6.3} &\leq c \tau^{\delta} \rr^{\delta-\gamma(a-n/\chi)}
\left(\int_{\tau\rr}^{\infty} \frac{\dtau }{\nu^{1+s-a+n/\chi}}
\right)^{\gamma}[\GGp(\rr)]^{\vartheta\gamma}\\
& \leq c\tau^{\delta-s\gamma+\gamma(a-n/\chi)}\rr^{\delta-s\gamma} [\GGp(\rr)]^{\vartheta\gamma}\\
& \leq c\tau^{\delta(a-n/\chi)}\rr^{\delta-s\gamma} [\GGp(\rr)]^{\vartheta\gamma}\\
& \leq c\tau^{\delta(a-n/\chi)}[\GGp(\rr)]^{p}+  c(\mathds{A}_{\gamma}+\mathds{B}_{\gamma}) \tau^{\delta(a-n/\chi)}\rr^{\frac{p(\delta - s\gamma)}{p-\vartheta \gamma}}\\ & \leq c \tau^{\delta(a-n/\chi)}[\GGp(\rr)]^{p}
\,.
\end{flalign*}
The last term is dealt with as
$$
S_7 \stackleq{bb.17}  c\tau^{\delta}\rr^{\delta-s\gamma}[\GGp(\rr)]^{\vartheta\gamma} \leq c \tau^{\delta} [\GGp(\rr)]^{p}
 + c\tau^{\delta}(\mathds{A}_{\gamma}+\mathds{B}_{\gamma})  \rr^{\frac{p(\delta - s\gamma)}{p-\vartheta \gamma}}\leq c \tau^{\delta} [\GGp(\rr)]^{p}\,.
$$
Connecting the estimates found for $S_5, S_6$ and $S_7$ to \rif{bb.16} we obtain \eqn{ecce2}
$$
[\snail_{\delta}(\tau\rr)]^{\gamma/p}
 \le c \left(\tau^{\delta/p- n/q}+\tau^{\delta(a-n/\chi)/p}+\rr^{ \theta \ti{\sigma} \vartheta\gamma/p^2}\tau^{-n/p}\right) \GGp(\rr)\,.
$$
On the other hand, by the very definition in \rif{rightb}, we trivially have
\eqn{ecce3}
$$
\rhsp(\tau\rr)\le\left(\tau^{1-\theta/p}+\tau^{1-n/q}+\tau^{a-n/\chi}\right)\GGp(\rr)\,.
$$
Connecting \rif{ecce1}, \rif{ecce2} and \rif{ecce3}, and yet keeping \rif{condizioneb} in mind, we arrive at
\eqn{ecce4}
$$
\GGp(\tau\rr)
 \le c_1 \left(\tau^{1-\theta/p}+\tau^{\delta/p- n/q}+\tau^{\delta(a-n/\chi)/p}+\rr^{ \theta \ti{\sigma} \vartheta\gamma/p^2}\tau^{-n/p}\right) \GGp(\rr)\,,
$$
where $c_1\equiv c_1 (\data)$. 
With $
\kappa>0$ being defined in \rif{g3}$_3$, we select a positive 
$\alpha < \kappa$ and then set $\alpha_1:= (\alpha +\kappa)/2$, so that $\alpha < \alpha_1 < \kappa$. We can find  $\delta\equiv \delta (n,p,q,a,\chi, \alpha)$ (close enough to $p$) and $\theta\equiv \theta (n,p,q,a,\chi, \alpha)$ (close enough to zero), such that  
$
\min
\{ 
1-\theta/p,\delta/p- n/q,\delta(a-n/\chi)/p\} >  \alpha_1.
$
Then we take $\tau \equiv \tau (\data, \alpha)$ small enough to have 
$$
c_1 \left(\tau^{1-\theta/p-\alpha_1}+\tau^{\delta/p- n/q-\alpha_1}+\tau^{\delta(a-n/\chi)/p-\alpha_1}\right) \leq \frac 12 \qquad \mbox{and}\qquad \tau^{(\kappa-\alpha)/2}\leq \frac 12\,.
$$
With $\tau$ being determined, we now select a positive radius $r_{**}\equiv r_{**}(\data, \alpha)\leq r_0/4$ such that $\rr \leq r_{**}$ implies  
$
c_1\rr^{ \theta \ti{\sigma} \vartheta\gamma/p^2}\tau^{-n/p-\alpha_1} \leq 1/2.
$
Using this last inequality, and the one in the last display, in \rif{ecce4}, implies $
\GGp(\tau\rr) \leq \tau^{\alpha_1} 
\GGp(\rr),$
which is the boundary analog of \rif{analoga}. This leads to consider the maximal operators
$$
 \texttt{M}^{+}(\tilde{x}_0,\rr):= \sup_{\nu \leq \rr} \nu^{-\alpha} \GGp(u, B_\nu(\ti{x}_0))\,,\qquad 
 \texttt{M}^{+}_{\eps}(\tilde{x}_0,\rr):= \sup_{\eps \rr  \leq \nu \leq \rr} \nu^{-\alpha} \GGp(u, B_\nu(\ti{x}_0))
$$
for $\eps < \tau$. 
Proceeding as after \rif{analoganc}, and taking into account \rif{snail-l} and \rif{boundprop2}, we arrive a
$
  \texttt{M}\,(x_0,r)  \leq c (\data)$. From this and the fact that the chosen point $\ti{x}_0$ is arbitrary, we conclude with 
\eqn{arriveatat}
 $$
 \sup_{\tilde{x}_0\in \Gamma_{r_{0}/2}}\,  \sup_{\rr \leq r_{**}}\,\mint_{B_{\rr}^{+}(x_0)} |\ti{u}-\ti{g}|^p\, dx \leq c \rr^{\alpha p}\,.
$$
Here recall that $r_{**}\equiv r_{**}(\data, \alpha)$. Using Sobolev-Morrey embedding theorem, we find 
$$
\mint_{B_{\rr}^+(\tilde x_0)} |\ti{g}-(\ti{g})_{B_{\rr}(\tilde x_0)}|^p\, dx \leq
\left( \osc_{B_{\rr}^+(x_{0})}\, \ti{g}\right)^p \leq c \rr^{(1-n/q)p} \|D\ti{g}\|_{L^q(B_{\rr}^+)} \leq c \rr^{\kappa p}\leq c \rr^{\alpha p}\,,
$$
where $c\equiv c(\data)$. 
Combining the two inequalities above, and yet using \rif{averages}, we finally get that
\eqn{primadis}
$$  \sup_{\tilde{x}_0\in \Gamma_{r_{0}/2}}\,  \sup_{\rr \leq r_{**}}\,\mint_{B_{\rr}^{+}(\tilde x_0)} |\ti{u}-(\ti{u})_{B_{\rr}^+(\tilde x_0)}|^p\, dx\leq c \rr^{\alpha p} 
$$
holds whenever $\rr \leq r_{**}$, where $c \equiv c (\data, \alpha)$. On the other hand, by Proposition \ref{outcome}  there exists $c \equiv c (\data)\geq 1$ and another positive radius $r_* \equiv r_*(\data, \alpha)\leq r_0/4$, such that 
$$ \mint_{B_{\rr}(y)} |\ti{u}-(\ti{u})_{B_{\rr}(y)}|^p\, dx\leq c \rr^{\alpha p} 
$$
holds whenever $\rr \leq r_*$ and $B_{\rr}(y) \Subset B_{r_{0}}^+(x_0)$. Combining the information in the last two displays in a standard way yields that now \rif{primadis} holds not only when $\ti{x}_0$ belongs to $ \Gamma_{r_{0}/2}$ as in \rif{arriveatat}, but whenever $\ti{x}_0\in B_{r_{0}/2}^{+}(x_0)$ and $\rr \leq \min\{r_*, r_{**}\}/8\leq r_0/4$. This implies the validity of Proposition \ref{campi} via Campanato-Meyers integral characterization of H\"older continuity. 


\subsection{Step 7: Estimate \rif{breg}}\label{stimaduz} Estimates like \rif{breg} can be found in various places in the literature under additional structure conditions and assumptions. We did not find and explicit reference for it and therefore we offer a rapid derivation here for the sake of completeness. We denote $F_0(z):=\ccc(\ti{x}_0)\ti{F}(\ti{x}_{0},z)$, using the same notation of Section \ref{clb}. Note that $\ti{w}=\ti{h}-\ti{g}$ solves 
\eqn{unabordo}
$$
\begin{cases}
\, -\diver\, \partial_zF_0(D\ti{g}+D\ti{w})=0 & \mbox{in $B_{\rr/4}^{+}$} \\
\, \ti{w} \equiv 0 & \mbox{on $\Gamma_{\rr/4}$\,.}
\end{cases}
$$ 
We denote by $\tilde v\in  \ti{w}+W^{1,p}_0(B_{\rr/4}^{+})$ as the solution to 
\eqn{duabordo}
$$
\begin{cases}
\, -\diver\, \partial_zF_0(D\ti{v})=0 & \mbox{in $B_{\rr/4}^{+}$} \\
\, \ti{v} \equiv \ti{w} & \mbox{on $\partial B_{\rr/4}^{+}$\,.}
\end{cases}
$$ 
By \cite[Theorem 2.2]{CM} we obtain that
\eqn{lasty}
$$
\|D\ti{v}\|_{L^{\infty}(B_{\rr/8}^{+})}^p \leq c \mint_{B_{\rr/4}^{+}} (|D\ti{v}|^2+\mu^2)^{p/2} \dx  \leq 
c \mint_{B_{\rr/4}^{+}} (|D\ti{w}|^2+\mu^2)^{p/2} \dx 
$$ 
with $c\equiv c (n,p,\ti{\Lambda})$ (note that \cite[Theorem 2.2]{CM} is stated for the degenerate case $\mu=0$, but the proof applies verbatim in the non-degenerate case $\mu >0$, which is actually simpler). The former inequality in \rif{lasty} follows from a delicate barrier argument, and the latter is a consequence of minimality of $\ti{v}$ (it solves an Euler-Lagrange equation). In turn, also using the minimality of $\ti{h}$ in \rif{pddb}, we find 
$$
 \int_{B_{\rr/4}^{+}} (|D\ti{w}|^2+\mu^2)^{p/2} \dx  \leq 
c \int_{B_{\rr/4}^{+}} (|D\ti{h}|^2+|D\ti{g}|^2+\mu^2)^{p/2} \dx  \leq 
c \int_{B_{\rr/4}^{+}} (|D\ti{u}|^2+|D\ti{g}|^2+\mu^2)^{p/2} \dx
$$ 
with $c\equiv c (n,p,\ti{\Lambda})$. 
On the other hand, this time being $\mathcal V^2:=|V_{\mu}(D\ti{v})-V_{\mu}(D\ti{w})|^2$, we have
\begin{eqnarray*}
\int_{B_{\rr/4}^{+}} \mathcal V^2\dx & \stackleq{monoin} & c \int_{B_{\rr/4}^{+}}  (\partial_zF_0(D\ti{v})-\partial_zF_0(D\ti{w}))\cdot 
(D\ti{v}-D\ti{w})\dx\\
& \stackrel{\eqref{unabordo}}{=}  &c \int_{B_{\rr/4}^{+}}  (\partial_zF_0(D\ti{g}+D\ti{w})-\partial_zF_0(D\ti{w}))\cdot 
(D\ti{v}-D\ti{w})\dx\\
& \stackrel{\eqref{algebra},\eqref{assft}_3}{\leq} &c \int_{B_{\rr/4}^{+}}  (|D\ti{g}|^2+|D\ti{w}|^2+\mu^2)^{(p-2)/2}|D\ti{g}||D\ti{v}-D\ti{w}|\dx\,.
\end{eqnarray*}
In the case $p\geq 2$, \rif{Vm} implies $|D\ti{v}-D\ti{w}|^p \leq c \mathcal V^2$ and, by repeated use of Young's inequality, and reabsorbing terms, we find 
\eqn{aagain}
$$
\int_{B_{\rr/4}^{+}} |D\ti{v}-D\ti{w}|^p\dx  \leq  \eps  \int_{B_{\rr/4}^{+}} (|D\ti{w}|^2+\mu^2)^{p/2} \dx + c_{\eps}\int_{B_{\rr/4}^{+}} |D\ti{g}|^{p} \dx
$$
for every $\eps \in (0,1)$, where $c_\eps$ depends on $n,p,\ti{\Lambda}, \eps$. 
In the case $1<p<2$, as in \rif{uh2}, we instead find 
\begin{flalign*}
\int_{B_{\rr/4}^{+}}|D\ti{v}-D\ti{w}|^p\dx &\leq c\left(\int_{B_{\rr/4}^{+}}\mathcal{V}^{2}\dx\right)^{p/2}\left(\int_{B_{\rr/4}^{+}}(\snr{D\ti{v}}^{p}+\snr{D\ti{w}}^{p})\dx\right)^{1-p/2}\\
&\leq c\left(\int_{B_{\rr/4}^{+}}  |D\ti{g}|^{p-1}|D\ti{v}-D\ti{w}|\dx\right)^{p/2}\left(\int_{B_{\rr/4}^{+}}\snr{D\ti{w}}^{p}\dx\right)^{1-p/2}\\
&\leq c\left(\int_{B_{\rr/4}^{+}}  |D\ti{v}-D\ti{w}|^p\dx\right)^{1/2}
\left(\int_{B_{\rr/4}^{+}}  |D\ti{g}|^{p}\dx\right)^{\frac{p-1}{2}}
\left(\int_{B_{\rr/4}^{+}}\snr{D\ti{w}}^{p}\dx\right)^{\frac{2-p}{2}}\,,
\end{flalign*}
from which \rif{aagain} follows again via Young's inequality with conjugate exponents $(1/(p-1), 1/(2-p))$. 
Combining \rif{lasty} with \rif{aagain} in a standard way, we arrive at 
\begin{align*}
&\int_{B_{t}^{+}}(|D\ti{w}|^2+\mu^2)^{p/2}\dx \\ & \quad \leq c \left[\left(\frac{t}{\rr}\right)^n+\eps\right]
\int_{B_{\rr/4}^{+}}(|D\ti{w}|^2+\mu^2)^{p/2}\dx+ c\left(\int_{B_{\rr/4}^{+}}|D\ti{g}|^{q}\dx\right)^{p/q}
 \rr^{n(1-p/q)}\,,
 \end{align*}
for all $t \leq \rr/4$, where $c\equiv c (n, p, \ti{\Lambda})$. By recalling the definition of $\ti{w}$, the above inequality holds with $\ti{w}$ replaced by $\ti{h}$, so that \rif{breg} follows applying Lemma \ref{l5bis} with the choice $h(t):= \|(D\ti{w}|^2+\mu^2)^{p/2}\|_{L^1(B_{t})}.$

 \section{Proof of Theorem \ref{t4}} \label{ghol}
In the following we select arbitrary open subsets 
$
\Omega_0 \Subset \Omega_1 \Subset \Omega
$, 
and denote $\texttt{d} :=\min\{ \dist(\Omega_0, \Omega_1),\linebreak \dist(\Omega_1, \Omega), 1\}$. We take $B_{\rr}\equiv B_{\rr}(x_{0})\Subset \Omega_1$ with $x_0\in \Omega_0$ and $0<\rr\leq \texttt{d}/4$ and all the balls used in the following will be centred at $x_0$. Moreover, $\beta, \lambda$ will be numbers verifying
$s<\beta<1$ and $\lambda >0$; their precise value will depend on the context they are going to be employed in. We shall often use Theorem \ref{t2} in the form $\|u\|_{C^{0, \beta}(\Omega_1)}\leq c \equiv c (\datah,\beta, \ttdd)$, for every $\beta <1$. 
\begin{lemma}
Under the assumptions on Theorem \ref{t4}
\begin{itemize}
\item If $ s<\beta<1$, then
\eqn{launa}
$$ 
\int_{B_{\rr/2}}\mint_{B_{\rr/2}}\frac{\snr{u(x)-u(y)}^{\gamma}}{\snr{x-y}^{n+s\gamma}}\dx\dy \leq c \rr^{(\beta-s)\gamma} 
$$
holds with $c \equiv c (\datah, \ttdd,\beta)$.
\item The inequality
\eqn{snailetto}
$$t^{-\delta}[\snail_{\delta}(t)]^{\gamma}\equiv t^{-\delta}[\snail_{\delta}(u, B_t(x_0))]^{\gamma} \leq c$$
holds whenever $0<t \leq \rr$, where $c \equiv c (\datah,\ttdd)$. 
\item If $\lambda>0$, then 
\eqn{g.2}
$$
\mint_{B_{\rr/2}}(\snr{Du}^{2}+\mu^{2})^{p/2}\dx
\le c\rr^{-p \lambda}
$$
holds with
$c\equiv c(\datah,
\ttdd,\lambda)$.
\end{itemize}
\end{lemma} 
\begin{proof} Estimate in \rif{launa} follows from Theorem \ref{t2} with $\alpha \equiv \beta$. To prove \rif{snailetto} we estimate as follows:
\begin{flalign*}
t^{-\delta}[\snail_{\delta}(t)]^{\gamma}
&\le c\int_{\mathbb{R}^{n}\setminus B_{t}}\frac{\snr{u(y)-u(x_{0})}^{\gamma}}{\snr{y-x_{0}}^{n+s\gamma}} \dy+ct^{-s\gamma}\snr{u(x_{0})-(u)_{B_{t}(x_{0})}}^{\gamma}\nonumber \\
&\le c\int_{\mathbb{R}^{n}\setminus B_{\texttt{d}}}\frac{\snr{u(y)-u(x_{0})}^{\gamma}}{\snr{y-x_{0}}^{n+s\gamma}} \dy+\int_{B_{\ttdd}\setminus B_{t}}\frac{\snr{u(y)-u(x_{0})}^{\gamma}}{\snr{y-x_{0}}^{n+s\gamma}} \dy+ct^{(\beta-s)\gamma}[u]_{0,\beta;\Omega_1}^{\gamma}\nonumber \\
&\le c\ttdd^{-s\gamma}+ c\int_{B_{\ttdd}}\snr{y-x_{0}}^{-n-(s-\beta)\gamma} \dy\, [u]_{0,\beta;\Omega_1}^{\gamma}+ct^{(\beta-s)\gamma}[u]_{0,\beta;\Omega_1}^{\gamma}\\
&\le c\ttdd^{-s\gamma} + c \ttdd^{(\beta-s)\gamma}+ c \leq c   \,,
\end{flalign*}
where $c\equiv c(\datah,\beta)$, that is \rif{snailetto} if we choose $\beta := (1+s)/2$. Finally, to prove \rif{g.2} we use \eqref{caccp} and estimate the various terms stemming from $\caccs(\rr)\equiv \caccs(u,B_{\varrho}(x_0))$, whose definition is in \rif{caccica}. Again by Theorem \ref{t2}, we have 
that $
\rr^{-p}[\av_{p}(\rr)]^{p}+\rr^{-s\gamma}[\av_{\gamma}(\rr)]^{\gamma}\le c\rr^{p(\beta-1)}+c\rr^{\gamma(\beta-s)}\le c\rr^{p(\beta-1)}
$
holds with $c\equiv c(\datah,\ttdd,\beta)$. By \rif{snailetto} we on the hand have $\rr^{-\delta}[\snail_{\delta}(\rr)]^{\gamma}+\nr{f}^{p/(p-1)}_{L^n(B_{\rr})}+1\leq c\leq  c\rr^{p(\beta-1)}$. Choosing $\beta$ such that  $1-\beta\leq   \lambda$, we arrive at \rif{g.2}.  
\end{proof}
\begin{lemma}\label{harfinal}
If $h\in u+W^{1,p}_{0}(B_{\rr/4}(x_{0}))$ is as in \eqref{pdd}, then
\eqn{cl.11f}
$$
\mint_{B_{\rr/4}(x_{0})}\snr{Du-Dh}^{p}\dx \le c\rr^{\sigma_2p}
$$
holds where $\sigma_2\equiv \sigma_2(n,p,s,\gamma,d) \in (0,1)$, and 
 $c\equiv c(\datah,\ttdd,\lambda)$.
\end{lemma}
\begin{proof} We go back to Lemma \ref{har}, estimate \rif{daje}, and, adopting the notation introduced there, we improve the estimates for the terms $\mbox{(I)}$-$\mbox{(III)}$. As in \rif{cl.3} and in Lemma \ref{cacclem}, we find
\begin{flalign}
 \snr{\mbox{(I)}}&\stackrel{\eqref{enes}}{\le}c\nr{f}_{L^{n}(B_{\rr/4})}\left(\mint_{B_{\rr/4}}(\snr{Du}^{2}+\mu^2)^{p/2}\dx\right)^{1/p}\stackrel{\eqref{g.2}}{\le}c\nr{f}_{L^{d}(B_{\rr/4})}\rr^{1-n/d-\lambda}\,,\label{prendila}
\end{flalign}
for every $\lambda>0$, where $c\equiv c(\datah,\ttdd,\lambda)$. 
In order to estimate terms (II) and (III), we recall that a basic consequence of the maximum principle is
\eqn{massimo}
$$
\osc_{B_{\rr/4}}h\le \osc_{ B_{\rr/4}}u\,.
$$
Recall also that $u$ is H\"older continuous; by the Maz'ya-Wiener boundary regularity theory, $h$ is continuous on $\bar B_{\rr/4}$ and therefore
\eqn{g.6}
$$
\nr{w}_{L^{\infty}(B_{\rr/4})}= \nr{u-h}_{L^{\infty}(B_{\rr/4})}\stackrel{\rif{massimo}}{\leq} 2\osc_{\partial B_{\rr/4}}u \leq 4 [u]_{0, \beta; B_{\rr/4}}\rr^{\beta}\leq c\rr^{\beta}\,,
$$
where $c\equiv c(\data,\ttdd,\beta)$. For (II), as in \rif{cl.4}, we have, with $w=u-h$ (defined and extended as in Lemma \ref{har}, so that $w\equiv 0$ outside $B_{\rr/4}$)
\begin{eqnarray}
\nonumber \snr{\mbox{(II)}}&\stackleq{launa} &  c\rr^{(\beta-s)(\gamma-1)}\left(\int_{B_{\rr/2}}\mint_{B_{\rr/2}}\frac{\snr{w(x)-w(y)}^{\gamma}}{\snr{x-y}^{n+s\gamma}}\dx\dy\right)^{1/\gamma}\\
&\stackrel{\eqref{bm111}}{\leq}& c\rr^{(\beta-s)(\gamma-1)+\vartheta-s} \nr{w}_{L^{\infty}(B_{\rr/4})}^{1-\vartheta}
\left(\mint_{B_{\rr/4}}\snr{Dw}^{p}\dx\right)^{\vartheta/p}\nonumber \\
&\stackrel{\eqref{enes},\eqref{g.6}}{\leq}& c\rr^{(\beta-s)(\gamma-1)+\vartheta-s+\beta(1-\vartheta)} 
\left(\mint_{B_{\rr/4}}(\snr{Du}^{2}+\mu^2)^{p/2}\dx\right)^{\vartheta/p}\nonumber \\
&\stackrel{ \eqref{g.2}}{\le}& c\rr^{(\beta-s)(\gamma-1)-\lambda} \label{prendila2}
\end{eqnarray}
whenever $s < \beta <1$ and $\lambda>0$, where $c\equiv c(\datah,\ttdd,\beta,\lambda)$. 
To estimate $\mbox{(III)}$ we restart from the fifth line of \rif{cl.5}, and using also \rif{snailetto} and \rif{g.6}, we easily find
\begin{eqnarray}
\snr{\mbox{(III)}}&\leq&c\left[\rr^{-s\gamma}[\av_{\gamma}(\rr)]^{\gamma-1}+\rr^{-s}\left(\rr^{-\delta}[\snail_{\delta}(\rr)]^{\gamma}\right)^{1-1/\gamma}\right]\nr{w}_{L^{\infty}(B_{\rr/4})}\nonumber\\
&\leq&c\rr^{(\beta-s)\gamma}+c\rr^{\beta-s}\leq c\rr^{\beta-s}\label{prendila3}\,.
\end{eqnarray}
In \rif{prendila}, \rif{prendila2} and \rif{prendila3}, the numbers $\beta, \lambda$ are arbitrary and such that $s < \beta<1$, $\lambda >0$ and the constants denoted by $c$ depend on $\datah,\ttdd,\beta,\lambda$. We then choose $\beta,   \lambda$ such that  
$$
\beta:= \frac{1+s}{2}, \qquad 0<\lambda \leq \min\left\{
\frac{(1-s)(\gamma-1)}{4}, \frac 12\left(1-\frac{n}{d}\right)\right\}$$
and plug \rif{prendila},\rif{prendila2} and \rif{prendila3} into \eqref{daje}, to obtain
\eqn{x5}
$$
\mint_{B_{\rr/4}}|V_{\mu}(Du)-V_{\mu}(Dh)|^{2}\dx\le c\rr^{\sigma_1p}, \quad \sigma_1 :=\frac{1}{p}\min\left\{
 \frac 12\left(1-\frac{n}{d}\right),\frac{(1-s)(\gamma-1)}{4}, \frac{1-s}{2}\right\}>0
$$
for $c\equiv c(\datah,\ttdd)$. Now, we want to prove that 
\eqn{provalo}
$$
\mint_{B_{\rr/4}}\snr{Du-Dh}^{p}\dx\leq c\rr^{\sigma_2p}\,,
$$
where $\sigma_2=\sigma_1$ if $p \geq 2$ and $\sigma_2:=\sigma_1p/4$ if $1<p<2$. Indeed, If $p\ge 2$, then \rif{provalo}
follows thanks to \rif{Vm} and \rif{x5}. 
When $p\in (1,2)$, as in \rif{uh2}, and using \rif{enes}, we have
\begin{eqnarray*}
\mint_{B_{\rr/4}}\snr{Du-Dh}^{p}\dx &\le&\left(\mint_{B_{\rr/4}}|V_{\mu}(Du)-V_{\mu}(Dh)|^{2}\dx\right)^{p/2}\left(\mint_{B_{\rr/4}}(\snr{Du}^{2}+\mu^2)^{p/2}\dx\right)^{1-p/2}\nonumber \\
&\stackrel{\eqref{g.2},\eqref{x5}}{\le}&c\rr^{[\sigma_1p/2-\lambda(1-p/2)]p}\,.
\end{eqnarray*}
By choosing $\lambda$ such that $\sigma_1p/2-\lambda(1-p/2) > \sigma_1p/4$, we finally conclude with \rif{provalo}. 
\end{proof}
Once \rif{provalo} is established, we can conclude with the local H\"older continuity of $Du$ by means of a by now classical comparison argument (see for instance \cite{manth2}). We briefly report it here for the sake of completeness. We recall the following classical decay estimate, which is satisfied by $h$
\eqn{x7}
$$
\osc_{B_{t}} Dh \leq c\left(\frac{t}{\rr}\right)^{\alpha_0} \left(\mint_{B_{\rr/4}}{(\snr{Dh}^2+\mu^2)^{p/2}}\right)^{1/p} \,,
$$
that holds whenever $0 < t \leq \rr/8 $, where $c \equiv c (n, p, \Lambda)\geq 1$ and $\alpha_0 \equiv \alpha_0(n,p,\Lambda)\in (0,1)$; see \cite{manth1, manth2}. We estimate, also using \rif{averages}
\begin{eqnarray}\label{x8}
\mint_{B_{t}}\snr{Du-(Du)_{B_t}}^{p}\dx&\le&c\, \left(\osc_{B_{t}} Dh\right)^{p}+c\left(\frac{\rr}{t}\right)^{n}\mint_{B_{\rr/4}}\snr{Du-Dh}^{p}\dx\nonumber \\
&\stackrel{\eqref{cl.11f},\eqref{x7}}{\le}&c\left(\frac{t}{\rr}\right)^{\alpha_{0}p}
\mint_{B_{\rr/4}}{(\snr{Dh}^2+\mu^2)^{p/2}}\dx +c\left(\frac{\rr}{t}\right)^{n}\rr^{\sigma_{2}p}\nonumber \\
&\stackrel{\eqref{enes},\eqref{g.2}}{\le}&c\left(\frac{t}{\rr}\right)^{\alpha_{0}p}\rr^{-\lambda p}+c\left(\frac{\rr}{t}\right)^{n}\rr^{\sigma_{2}p},
\end{eqnarray}
with $c\equiv c(\datah,\ttdd,\lambda)$. In the above inequality, we take $t=\rr^{1+\sigma_2p/(2n)}/8$ and choose $\lambda := \sigma_2p\alpha_0/(4n)$ in \rif{g.2}. We conclude with 
$$
\mint_{B_{t}}\snr{Du-(Du)_{B_t}}^{p}\dx \le ct^{\alpha p}\,, \qquad \alpha:=\frac{\sigma_2\alpha_0}{2\sigma_2 p + 4n}\,,
$$
where $c\equiv c(\datah,\ttdd,\lambda)$. This holds whenever $B_t\Subset \Omega$ is s ball centred in $\Omega_0$, with $t\leq  \ttdd^{1+\sigma_2p/(2n)}/c(n,p)$. As the $\Omega_0 \Subset \Omega_1 \Subset\Omega$ are arbitrary, this implies the local $C^{1,\alpha}$-regularity of $Du$ in $\Omega$, via the classical Campanato's integral characterization of $Du$ together with the estimate for $[Du]_{0,\alpha; \Omega_0}$, and the proof is complete.

\end{document}